\documentclass[11pt, reqno]{amsart}
\usepackage{xcolor}
\usepackage[margin=1in]{geometry}
 
\usepackage{latexsym}
\usepackage{amsmath}
\usepackage{amssymb}
\usepackage{amsthm}
\usepackage{qcircuit}
\usepackage{physics}
\usepackage{braket}
\usepackage{tikz}
\usetikzlibrary{shapes,backgrounds}
\usepackage{tikzsymbols}
\usepackage{graphicx}
\usepackage{amsfonts,amssymb,amsthm, bm}
\usepackage{amsxtra}
\usepackage{array}    
\usepackage{tabularx} 
\usepackage{booktabs} 
\usepackage{array}
\usepackage{subcaption}

\theoremstyle{plain}%
\newtheorem{theorem}{Theorem}[section]
\newtheorem{corollary}[theorem]{Corollary}
\newtheorem{lemma}[theorem]{Lemma}

\theoremstyle{definition}

\newtheorem{example}[theorem]{Example}

\theoremstyle{remark}

\newtheorem{remark}[theorem]{Remark}

\numberwithin{equation}{section}

\usepackage{amsmath,amsthm,amssymb, graphicx, multicol, array, mathtools}
\usepackage{hyperref}
\usepackage{ stmaryrd }
\usepackage{ mathrsfs }
\usepackage{ upgreek }
\usepackage[shortlabels]{enumitem}
\usepackage{ tipa }
\usepackage{mathrsfs}
\usepackage{dirtytalk}
\usepackage{graphicx}
\graphicspath{  {./images/} }

\newcommand{\PP}{\mathbb{P}}
\newcommand{\E}{\mathbb{E}}
\newcommand{\Z}{\mathbb{Z}}
\newcommand{\R}{\mathbb{R}}

\newcommand{\defeq}{\vcentcolon=}

\newcommand{\aaa}{\alpha}

\newcommand{\hh}{\mathfrak{h}}

\usepackage{tikz}
\usetikzlibrary{arrows.meta,positioning,calc}

\newcommand{\vertiii}[1]{{\left\vert\kern-0.25ex\left\vert\kern-0.25ex\left\vert #1 
    \right\vert\kern-0.25ex\right\vert\kern-0.25ex\right\vert}}

\usepackage{hyperref}

\newcounter{bln}

\allowdisplaybreaks

\title[Bilinear Differential--Difference Equations in Integrable KPZ Models]{Bilinear Differential--Difference Equations and One-Point Distributions of Some KPZ-Class Models}
\author{C. Alexander Rodriguez} 
\address{Department of Mathematics, University of Toronto, 40 St.\ George Street, Toronto, Ontario, Canada M5S2E4}
\email{alexander.rodriguez@mail.utoronto.ca} 
\date{September 19, 2025}

\begin{document}
\begin{abstract}
     We introduce a collection of nonlinear integrable partial differential--difference equations that are satisfied by the one-point distribution functions of some classical integrable KPZ models. Moreover, these equations can be regarded as reparametrizations or as scaling limits of the Hirota bilinear difference equation (HBDE), a canonical discretization for many important integrable systems such as the Korteweg--de Vries (KdV) equation, the Kadomtsev--Petviashvili (KP) equation, and the two-dimensional Toda lattice (2DTL).  Our contributions are threefold: (i) general Fredholm determinant solutions; (ii) verification that known formulas for classical integrable KPZ models fit within our framework; and (iii) zero-curvature/Lax pair formulations. As an application, we derive formal scaling limits of the equations, including the KP limit under 1:2:3 KPZ scaling.  
\end{abstract}
\maketitle
\tableofcontents
\section{Introduction}\label{sec: 1}
In a seminal 1986 paper \cite{KPZ}, Kardar, Parisi, and Zhang proposed a paradigmatic stochastic equation for a class of interface growth models predicted to exhibit universal asymptotic fluctuations. Thus began the study of the $1+1$ dimensional KPZ universality class, a broad collection of mathematical and physical models linked by their shared universal scaling behaviour. While the KPZ equation cannot serve as a universal fixed point for the class (due to its lack of scaling invariance), the field has grown in recent decades through deep connections to interacting particle systems, random planar geometry, and both classical and quantum integrable systems (see, e.g.\ \cite{corwin2011kardarparisizhangequationuniversalityclass, Quas11, hairer2012solvingkpzequation, borodin2016lecturesintegrableprobabilitystochastic, MQR17, Dauvergne_2022, aggarwal2024scalinglimitcoloredasep}).

 A recent advance in \cite{Quastel_2022} established that the KPZ fixed point -- the conjectural universal scale-invariant Markov process first constructed in \cite{MQR17} -- has distribution functions satisfying the Kadomtsev--Petviashvili (KP) equation. Shortly thereafter, it was shown that the classical Polynuclear Growth (PNG) model has distribution functions satisfying the two-dimensional Toda lattice (2DTL) equation \cite{matetski2024polynucleargrowthtodalattice}. 
These results suggest a broader question: do other integrable KPZ models admit closed equations tied to classical integrability theory? 

 In this article, we answer in the affirmative. Indeed, by studying several classical KPZ models (see Table \ref{tab:eqs_models}), we establish a novel collection of closed nonlinear equations satisfied by one-point distribution functions. Moreover, these equations admit a bilinear Hirota form and can be regarded as reparametrizations or as scaling limits of the Hirota bilinear difference equation~(HBDE), namely
 \begin{equation}\label{HBDE}
     \left[z_1e^{D_{1}} + z_2e^{D_2} + z_3e^{D_3} \right] f\cdot f = 0, 
 \end{equation}
where $z_i$ are arbitrary constants and $D_i$ are linear combinations of binary operators $D_{x_i}$ with $(e^{\gamma D_{x_i}}f\cdot g) (x_i) \defeq f(x_i+\gamma)g(x_i-\gamma)$ and other variables kept fixed.  The Taylor expansion $$(e^{\epsilon D_x}f\cdot g) (x) = \sum_{k=0}^{\infty} (D_x^k f\cdot g)(x) \frac{\epsilon^k}{k!} $$ defines the $k$-th Hirota derivative via the binary operator $(f, g) \mapsto D_x^k f\cdot g$. The HBDE, first introduced in 1981 \cite{doi:10.1143/JPSJ.50.3785}, has a deceptively simple form that conceals its rich and far-reaching structure. Indeed, the HBDE not only serves as a discretization for myriad classical integrable systems (e.g.\ KP, 2DTL) but also emerges in quantum integrable systems as the model-independent functional relations for eigenvalues of quantum transfer matrices \cite{Krichever_1997, Zabrodin_1997, zabrodin2012betheansatzhirotaequation}.
\subsection{Main Results}
\begin{table}[tbp]
\centering
\begingroup
  \setlength{\tabcolsep}{5pt}      
  \renewcommand{\arraystretch}{1.2}

  \begin{tabularx}{\textwidth}{@{}%
    >{\centering\arraybackslash}p{0.8cm}%
    @{\hspace{7pt}}%
    >{\centering\arraybackslash}p{0.54\textwidth}%
    >{\raggedright\arraybackslash}X%
    @{}}
    \toprule
    \textbf{\#} & \textbf{Bilinear Equation} & \textbf{Model(s)} \\
    \midrule

    (1) &
    {\hypersetup{hidelinks}\hyperref[RBM Kernel Thm]{$\displaystyle \bigl[D_t - \tfrac{1}{2}D_a^2\bigr]\,F_{t,a,n}\cdot F_{t,a,n-1}=0$}} &
    {\hypersetup{hidelinks}\hyperref[RBM Particle Dist Thm]{Reflected Brownian Motions (RBM)}}; {\hypersetup{hidelinks}\hyperref[RBM Particle b Dist Thm]{RBM with Moving Wall}} \\
    
    (2) &
    {\hypersetup{hidelinks}\hyperref[TASEP K THM]{$\displaystyle \left[D_t - (e^{-D_a}-1)\right]\,F_{t,a,n}\cdot F_{t,a,n-1}=0$}} &
    {\hypersetup{hidelinks}\hyperref[TASEP part thm]{The Totally Asymmetric Simple Exclusion Process (TASEP)}};                {\hypersetup{hidelinks}\hyperref[TASEP b part thm]{\, TASEP with Moving Wall}} \\
    
    (3) &
    {\hypersetup{hidelinks}\hyperref[PUSHTASEP K thm]{$\displaystyle \left[D_t - (e^{D_a}-1)\right]\,F_{t,a,n}\cdot F_{t,a+1,n-1}=0$}} &
    {\hypersetup{hidelinks}\hyperref[Push TASEP part thm]{Push-TASEP}} \\
    
    (4) &
    {\hypersetup{hidelinks}\hyperref[PTASEP K THM]{$\displaystyle \left[e^{D_t} - p e^{-D_a} - (1-p)\right]\,F_{t,a,n}\cdot F_{t,a,n-1}=0$}} &
    {\hypersetup{hidelinks}\hyperref[Parallel TASEP part thm]{Discrete-time Parallel TASEP}}  \\
    
    (5) &
    {\hypersetup{hidelinks}\hyperref[RBTASEP K THM]{$\displaystyle \left[e^{D_t} - p e^{-D_a} - (1-p)\right]\,F_{t,a,n}\cdot F_{t+1,a,n-1}=0$}} &
    {\hypersetup{hidelinks}\hyperref[RB Part Thm]{Bernoulli Jumps with Blocking}} \\
    
    (6) &
    {\hypersetup{hidelinks}\hyperref[LBTASEP K THM]{$\displaystyle \left[e^{D_t} - q e^{D_a} - (1-q)\right]\,F_{t,a,n}\cdot F_{t+1,a+1,n-1}=0$}} &
    {\hypersetup{hidelinks}\hyperref[LB part thm]{Bernoulli Jumps with Pushing}}  \\
    
    \bottomrule
  \end{tabularx}
\endgroup
\caption{Bilinear differential--difference equations and corresponding KPZ models.}
\label{tab:eqs_models}
\end{table}
We present a selection of our main results. First, we provide a general Fredholm determinant solution theory for the bilinear equation (\hyperref[tab:eqs_models]{1}). For ease of reference, we refer to each equation by the KPZ model(s) it corresponds to. Throughout, let $K_{t, a, n}$ be a family of trace-class integral operators acting on a separable Hilbert space $L^2(X, \mu)$ and suppose further that $I- K_{t, a, n}$ is invertible for all $(t, a, n) \in V \subseteq \R_+ \times \R \times \Z$. Define
    \begin{equation}\label{Intro FTAN def}
        F_{t, a, n} = \det(I-K_{t, a, n})_{L^2(X, \mu)}, \qquad (t, a, n) \in V.
    \end{equation} 
\begin{theorem}[\textbf{RBM Eq.\ General Solutions}]\label{intro RBM gen sol}
    Suppose $K_{t, a, n}$ is in the regularity class $C_r^{1, 2}(U, \mathcal{I}_1)$, with $U = \R_+ \times \R \times \Z$, such that the following three conditions hold: 
    \begin{enumerate}[(i)]
        \item  
       $\begin{aligned}[t]
            \partial_a K_{t, a, n} &= \psi_{t, a, n} \otimes \phi_{t, a, n}.  
        \end{aligned}$
        \item $\begin{aligned}
            \nabla_n^+ \psi_{t, a, n} = \partial_a \psi_{t, a, n}, \quad \nabla_n^- \phi_{t ,a, n} = \partial_a \phi_{t, a, n}.
        \end{aligned}$
        \item $\begin{aligned}
            \partial_t \psi_{t, a, n} &= \frac{1}{2}(\partial_a^2 + 2\partial_a + I)\psi_{t, a, n}, \medspace  \partial_t \phi_{t, a, n} = -\frac{1}{2}(\partial_a^2 - 2\partial_a + I)\phi_{t, a, n}.
        \end{aligned}$
    \end{enumerate}
Then for $V \subseteq U$,  $F_{t, a, n}$ satisfies
    \begin{align}
        \Bigl[D_t - \frac{1}{2}D_a^2 \Bigr]F_{t, a, n}\cdot F_{t, a, n-1} = 0. 
    \end{align}
    i.e.\
    \begin{align}
        F_{n-1} \partial_t F_{n} - F_{ n} \partial_t F_{ n-1}  - \frac{1}{2}F_{ n-1}\partial_a^2 F_{ n} + \partial_a F_{ n} \partial_a F_{n-1} - \frac{1}{2}F_{ n} \partial_a^2 F_{n-1} = 0. 
    \end{align}
\end{theorem}
We refer the reader to Section \ref{sec: RBM equation} for a definition of the regularity class $C^{1, 2}_r(U, \mathcal{I}_1)$. Next, let $\mathbf{y} = (y_0, y_1, y_2, \dots)$ with $y_0 \geq y_1 \geq y_2 \geq \dots$, and let $\mathbf{B}(t) = (b(t), B_1(t), B_2(t) \dots)$, where $B_k(t)$ are standard i.i.d.\ Brownian motions and $b(t) \in C(\R_+, \R)$ with $b(0) = 0$. We define the RBM process, with data $(\mathbf{y}, b(\cdot))$, recursively via the Skorokhod reflection map: 
\begin{equation}
    Y_0(t) = y_0 + b(t), \quad Y_n(t) = y_n + B_n(t) - \sup_{0 \leq s \leq t}[y_n + B_n(s) - Y_{n-1}(s)]^+
\end{equation}
 where $[\,\cdot\,]^+ = \max( \, \cdot \,, 0)$, $n \geq 1$ (see Fig.\ \ref{fig:rbm_panels1}--\ref{fig:rbm_panels}). As a consequence of Thm.\ \ref{intro RBM gen sol} and known Fredholm determinant formulas for RBM models, we have the following corollary (see Sec.\ \ref{sec: RBM/BLPP Models} for references). 

    \begin{corollary}[\textbf{RBM Solutions}]
     Let $Y_n^{\text{RBM}}(t)$ denote the $n$-th RBM particle, and let 
     \begin{equation}
         F_{t, a, n} =  \PP\bigl(Y_{n}^{\text{RBM}}(t) > a \, | \, \mathbf{Y}(0)= \mathbf{y}, \, Y_0(\cdot) = b(\cdot ) \bigr).
     \end{equation}
    \begin{enumerate}
        \item(\textbf{General One-Sided Initial Condition}): Take initial data $\mathbf{y}$ as above, and set $y_0 = \infty$ so that $b(t) \equiv \infty$. Notice then $Y_1^{\text{RBM}}(\cdot)$ is just a standard Brownian motion. Then, with $V = \R_+ \times \R \times \Z_{\geq 1}$, $F_{t, a, n}$ satisfies $\left[D_t  - \frac{1}{2}D_a^2\right] F_{t, a, n} \cdot F_{t, a, n-1} = 0$, with $ F_{0, a, n} = 1_{y_n > a}$ and $\bigl(\partial_t-\frac{1}{2}\partial_a^2\bigr) F_{t,a, 1} = 0.$
        \item(\textbf{RBM with Moving Wall}): Take initial data $\mathbf{y} \equiv \mathbf{0}$ and fix a moving wall $b(\cdot)$. Then, with  $V = \{(t, a, n) : t \in \R_+, a < b(t), n \in \Z_{\geq 1}\}$, $F_{t, a, n}$ satisfies $\left[D_t  - \frac{1}{2}D_a^2\right] F_{t, a, n} \cdot F_{t, a, n-1} = 0$, with $ F_{0, a, n} = 1_{0 > a}$, $\bigl(\partial_t - \frac{1}{2}\partial_a^2\bigr)F_{t, a, 1} = 0$  for  $a < b(t)$, and   
             $F_{t, a, 1} = 0$ for $a \geq b(t).$
    \end{enumerate}
\end{corollary}
\newlength{\panelH}
\setlength{\panelH}{0.20\textheight} 
\begin{figure}[t]
  \centering
  \begin{subfigure}[c]{0.36\textwidth}
    \centering
    \includegraphics[height=\panelH,width=\linewidth, keepaspectratio]{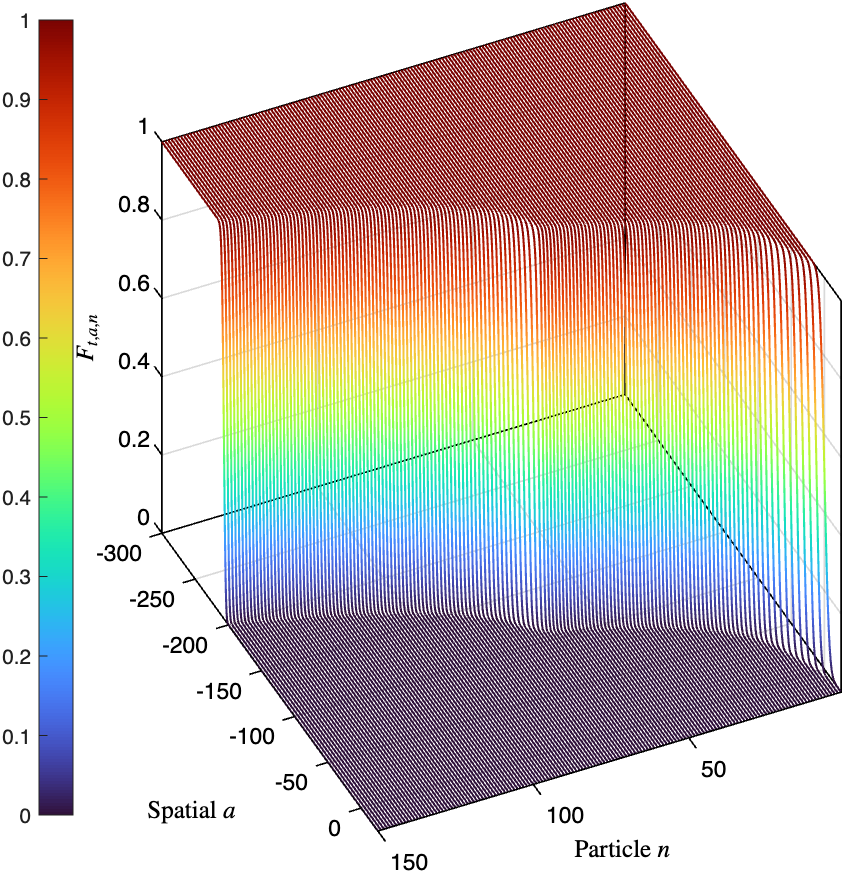}
  \end{subfigure}\hfill
  \begin{subfigure}[c]{0.36\textwidth}
    \centering
    \includegraphics[height=\panelH,width=\linewidth, keepaspectratio]{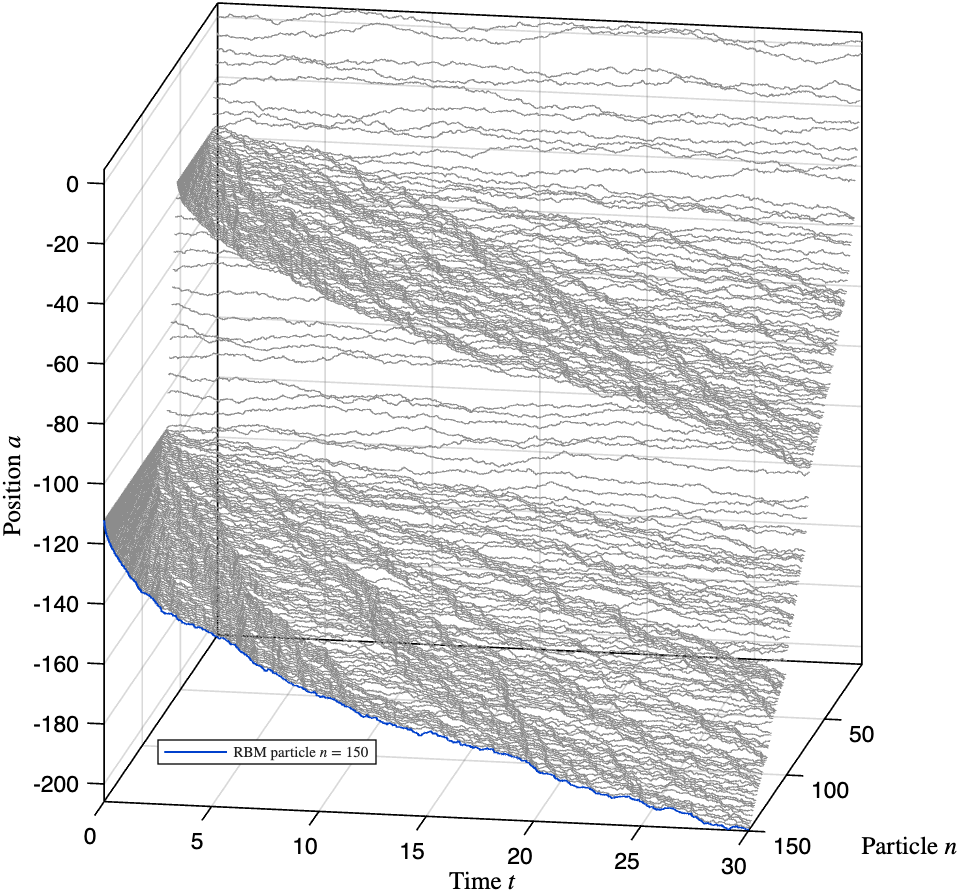}
  \end{subfigure}\hfill
  \begin{subfigure}[c]{0.26\textwidth}
    \centering
    \includegraphics[height=\panelH, width=\linewidth, keepaspectratio]{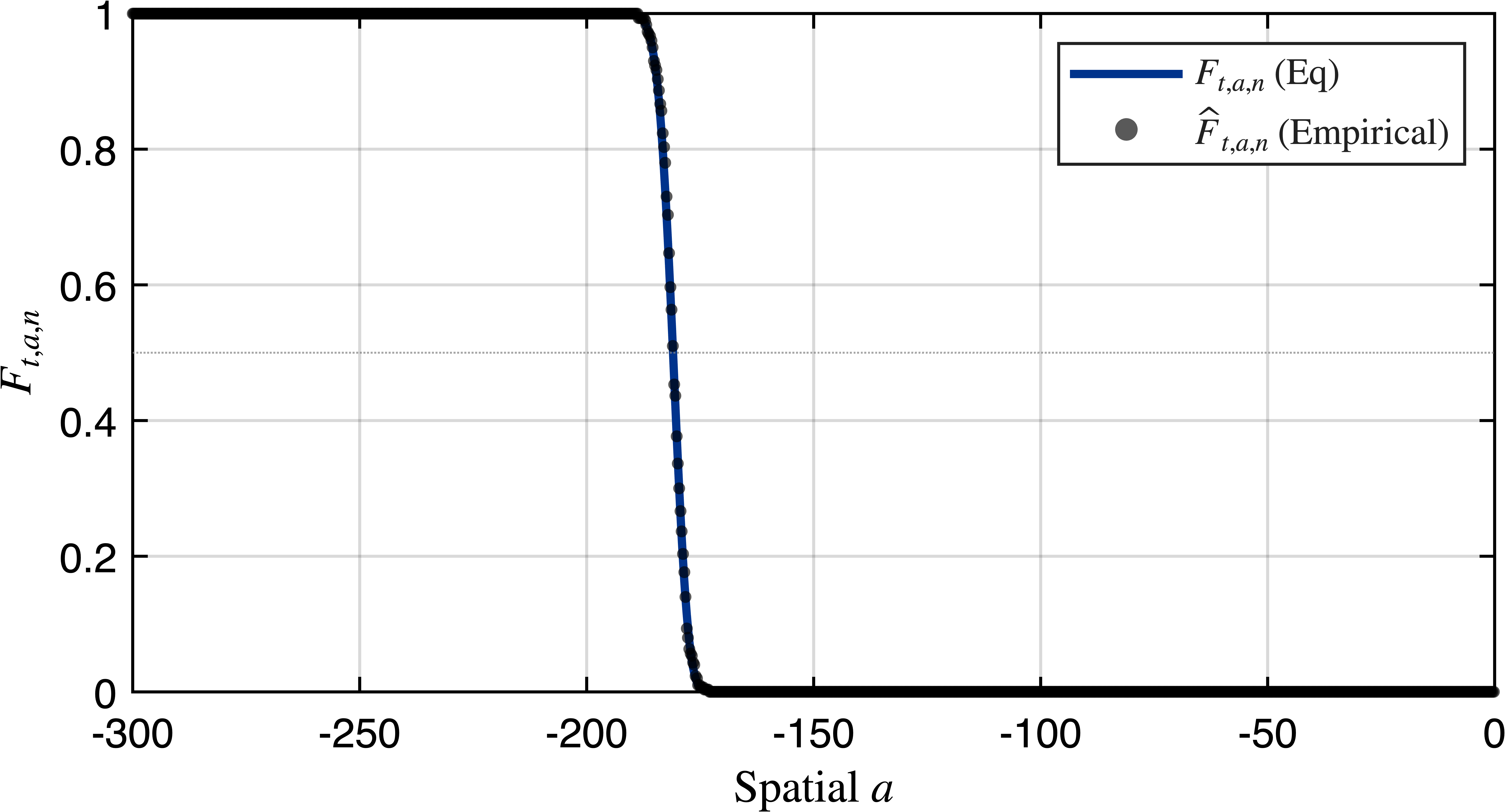}%
  \end{subfigure}
 \caption{Simulations of RBM with an alternating affine-packed initial condition. Left: PDE solution $F_{t,a,n}$ at $t=30$ (finite-difference scheme, initial data $F_{0, a, n} = 1_{y_n > a}$); Middle: RBM trajectories $Y_n(\cdot)$ (initial condition $Y_n(0) = y_n$); Right: $F_{t, a, N}$ vs.\ empirical $\hat{F}_{t, a, N}$ ($t=30, N=150$, and $300$ independent Monte Carlo runs).}
  \label{fig:rbm_panels1}
\end{figure}
Next, we connect equation (\hyperref[tab:eqs_models]{1}) to classical integrability theory by establishing a zero-curvature/Lax pair formulation.
We denote a backward shift operator by $a_ne^{-\partial_n}$ i.e.\ $(a_{\cdot}e^{-\partial_n}f)_n = a_nf_{n-1}$.
\begin{theorem}[\textbf{RBM Eq.\ Zero-Curvature Condition}]
Fix $M \in \Z$ and a collection of non-vanishing functions $\{F_{t, a, n}\}_{n\in \Z}$ with boundary condition $F_{t,a,m} \equiv 1$ for all $m \leq M$. Define 
\begin{align*}
    a_{n} &\defeq \frac{F_{n+1}F_{n-1}}{F_n^2}, \quad u_n \defeq \partial_a \log(F_{n}), \quad \nabla^s_n u_n \defeq \frac{u_{n+1}-u_{n-1}}{2}, \quad \mathcal{R} \defeq a_ne^{-\partial_n}.
\end{align*}
Consider the operators 
\begin{align}
    \mathcal{M} &\defeq \partial_t + (\nabla^s_n u_n)\mathcal{R} + \frac{1}{2}\mathcal{R}^2, \qquad 
    \bar{\mathcal{M}} \defeq \partial_a + \mathcal{R}, 
\end{align}
    acting on functions $f_n(t, a)$. Then 
    \begin{align}
       [\mathcal{M}, \bar{\mathcal{M}}] = 0 \Longleftrightarrow \left[D_t - \frac{1}{2}D_a^2 \right]F_{t, a, n}\cdot F_{t,a, n-1} = 0. 
    \end{align}
\end{theorem}
\begin{corollary}[\textbf{RBM Eq.\ Lax Pair}]
    In the same setting as above, define 
    \begin{align}
        L \defeq \partial_a + \mathcal{R}, \quad P \defeq (\nabla^s_n u_n) \mathcal{R} + \frac{1}{2}\mathcal{R}^2.
    \end{align}
    Then $(L, P)$ form a Lax pair for the RBM equation, i.e. 
    \begin{align}
        \partial_t L + [P, L] = 0 \Longleftrightarrow \left[D_t - \frac{1}{2}D_a^2 \right]F_{t, a, n}\cdot F_{t,a, n-1} = 0. 
    \end{align}
\end{corollary}
\begin{figure}[t]
  \centering
  \begin{subfigure}[t]{0.36\textwidth}
    \centering
    \includegraphics[height=\panelH, width=\linewidth, keepaspectratio]{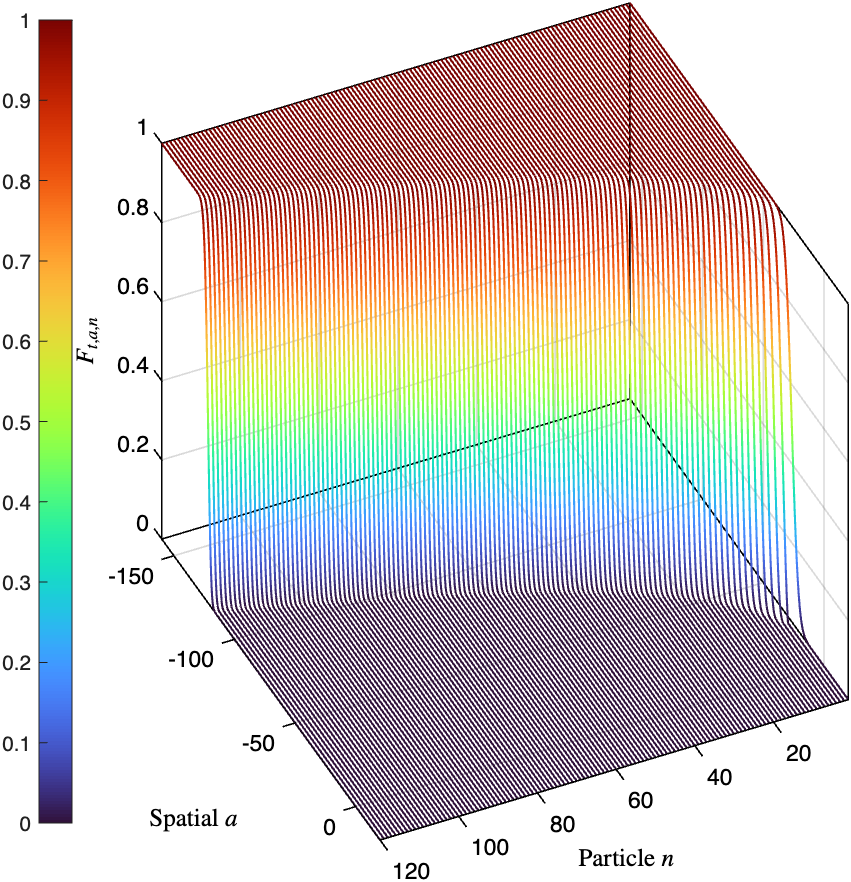} 
    \label{fig:rbm_cdf}
  \end{subfigure}\hfill
  \begin{subfigure}[t]{0.36\textwidth}
    \centering
    \includegraphics[height=\panelH, width=\linewidth, keepaspectratio]{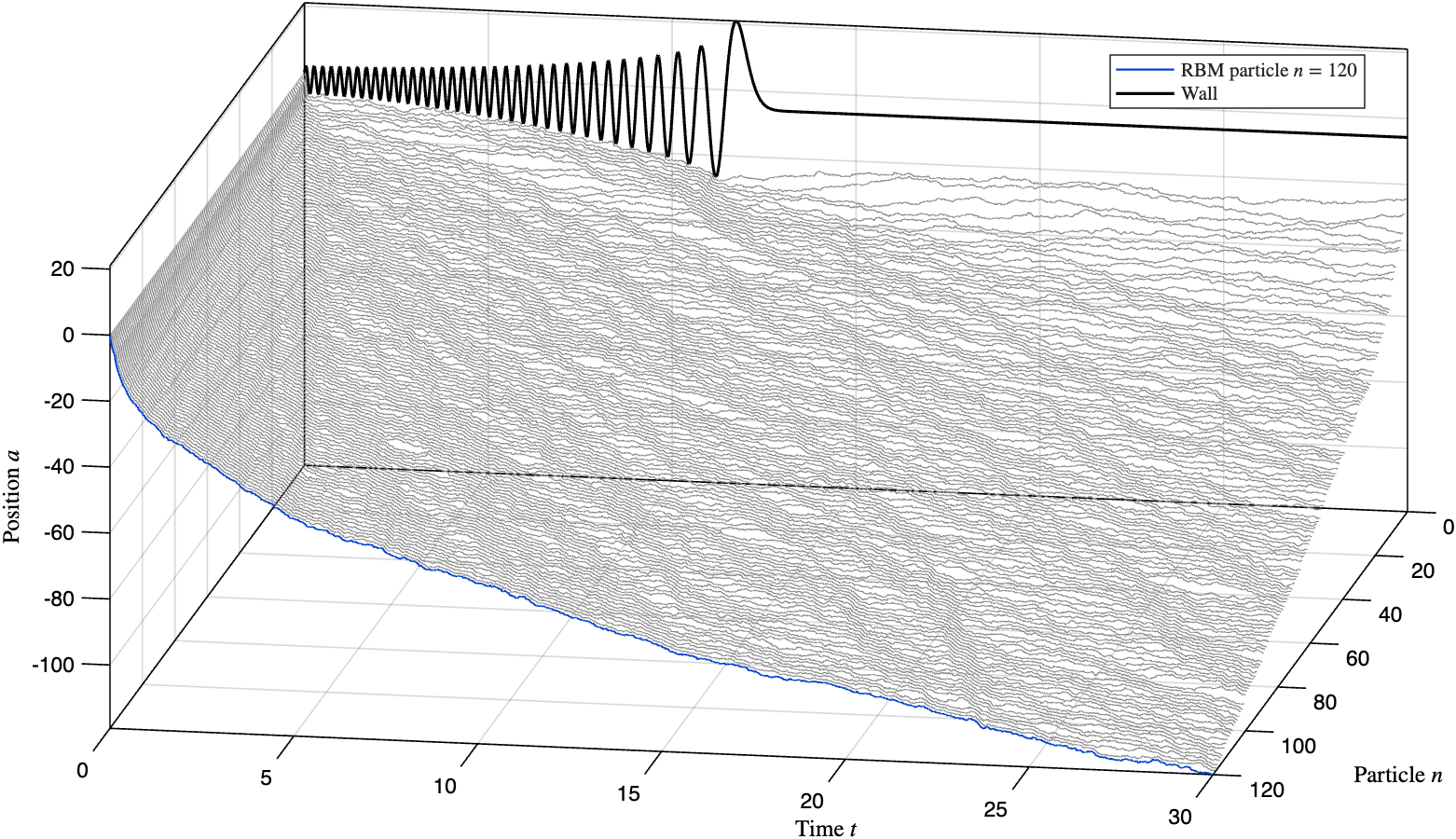}
    \label{fig:rbm_traj}
  \end{subfigure}\hfill
  \begin{subfigure}[b]{0.26\textwidth}
    \centering
    \includegraphics[height=\panelH,width=\linewidth, keepaspectratio]{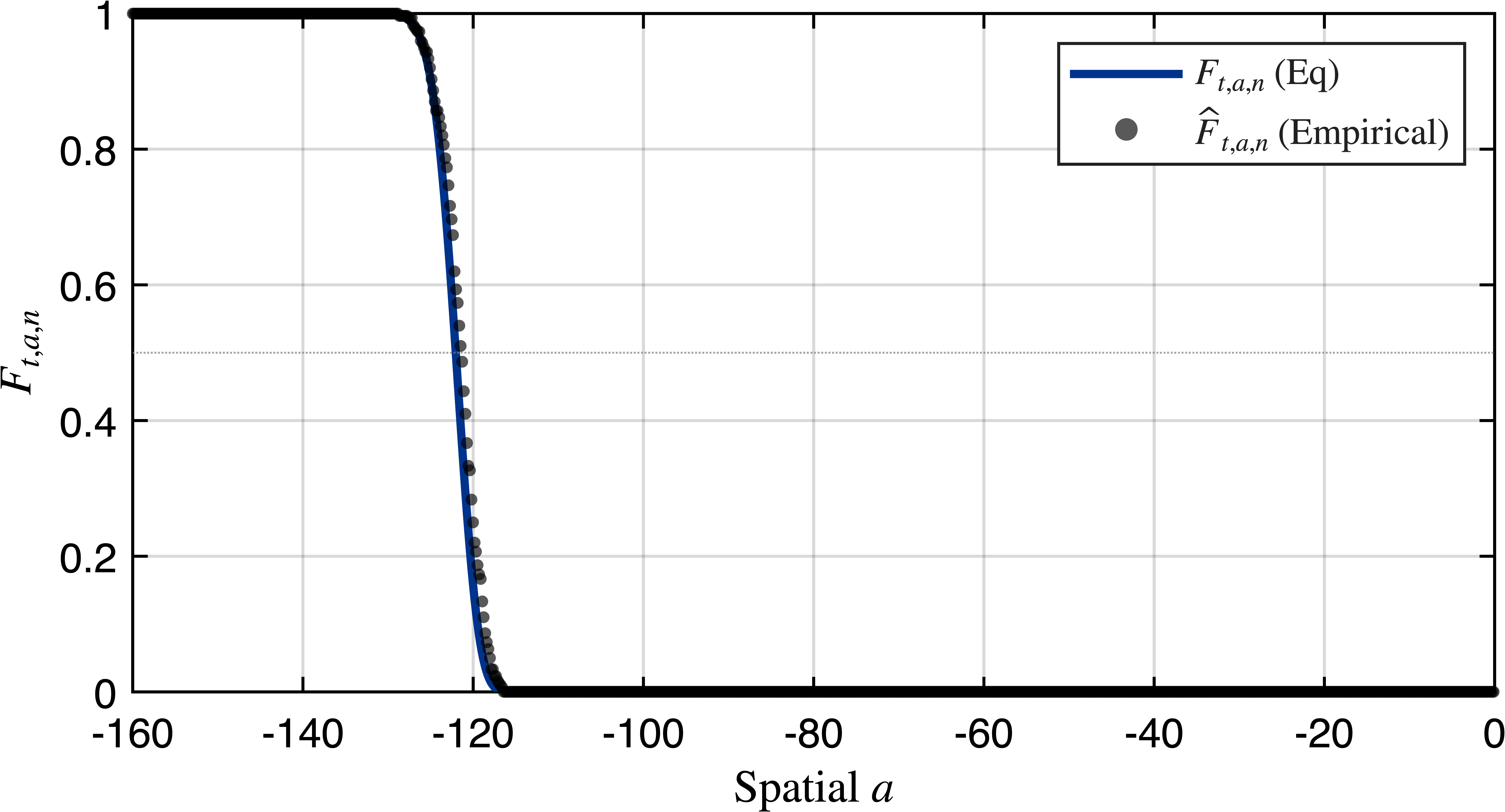}
    \label{fig:rbm_emp_theory}
  \end{subfigure}
\caption{Simulations of RBM with Moving Wall. Left: PDE solution $F_{t,a,n}$ at $t=30$ (finite-difference scheme, initial data $F_{0, a, n} = 1_{0 > a}$, and Dirichlet boundary condition $F_{t, b(t), 1} = 0$); Middle: RBM trajectories $Y_n(\cdot)$; Right: $F_{t, a, N}$ vs.\ empirical $\hat{F}_{t, a, N}$ ($t=30, N=120$, and $300$ independent Monte Carlo runs).} \label{fig:rbm_panels}
\end{figure}
Next, we present our results for the TASEP equation. In the same setting as \eqref{Intro FTAN def}, with $V \subseteq \R_+ \times \Z \times \Z$, we  provide an analogous Fredholm determinant solution theory for equation (\hyperref[tab:eqs_models]{2}). 
\begin{theorem}[\textbf{TASEP Eq.\ General Solutions}]\label{Intro TASEP Eq General SOl}
    Suppose $K_{t, a, n}$ is in the regularity class $C^1_r(U, \mathcal{I}_1)$, with $U = \R_+ \times \Z \times \Z$, such that the following three conditions hold: 
    \begin{enumerate}[(i)]
        \item  $\begin{aligned}[t] 
        \nabla_a^- K_{t, a, n} &= \psi_{t, a, n} \otimes \phi_{t, a, n}. 
        \end{aligned}$
        \item  $\begin{aligned}[t]
             \nabla_n^+ \psi_{t, a, n} = 2 \nabla_a^+ \psi_{t, a, n}, \medspace \nabla_n^- \phi_{t ,a, n} = 2 \nabla_a^- \phi_{t, a, n}.
        \end{aligned}$
        \item $\begin{aligned}[t]
            \partial_t \psi_{t, a, n} &= -\frac{1}{2}\nabla_a^- \psi_{t, a, n}, \medspace \partial_t \phi_{t, a, n} = -\frac{1}{2}\nabla_a^+ \phi_{t, a, n}.
        \end{aligned}$
    \end{enumerate}
  Then for $V \subseteq U$,  $F_{t, a, n}$ satisfies
    \begin{align}
        \left[D_t - (e^{-D_a} - 1) \right] F_{t, a, n} \cdot F_{t, a, n-1} = 0, 
    \end{align}
    i.e.
    \begin{align}
        F_{t, a, n-1}\partial_t F_{t, a, n} - F_{t, a, n} \partial_t F_{t, a, n-1} - F_{t, a-1, n}F_{t, a+1, n-1} + F_{t, a, n}F_{t, a, n-1} = 0.
    \end{align}
\end{theorem}
We refer the reader to Section \ref{sec: TASEP gensol} for the definition of the regularity class $C_r^{1}(U, \mathcal{I}_1)$. Next, we introduce the TASEP model, which is an interacting particle system on the integer lattice $\Z$ with at most one particle per site. Given a strictly decreasing initial configuration $\mathbf{y} = (y_0, y_1, y_2, \dotsc)$,  the dynamics run in continuous time as follows: Particle $Y_0(t)$ is deterministic and unaffected by other particles. Next,  particles with label $n \geq 1$ carry independent rate one exponential clocks, and when a particle's clock rings, it attempts to jump to the right by one unit. The jump is performed only if the destination site is empty; otherwise it is suppressed. After each particle's (attempted) jump its independent clock is instantaneously reset (see Fig.\ \ref{fig:rtasep_panels1}). As a consequence of Thm.\ \ref{Intro TASEP Eq General SOl} and known Fredholm determinant formulas for TASEP models, we have the following corollary (see Sec.\ \ref{sec: KPZ TASEP models}). 
\begin{figure}[t]
  \centering
  \begin{subfigure}[c]{0.36\textwidth}\centering
    \includegraphics[width=\linewidth,height=\panelH, keepaspectratio]{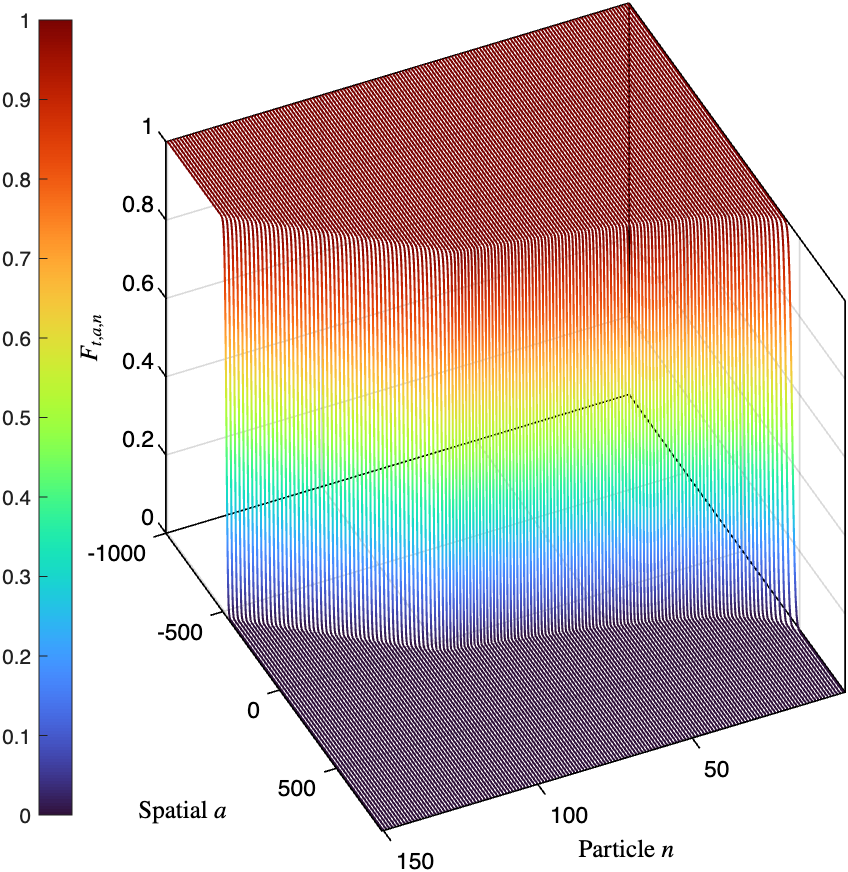}
  \end{subfigure}\hfill
  \begin{subfigure}[c]{0.36\textwidth}\centering
    \includegraphics[width=\linewidth,height=\panelH, keepaspectratio]{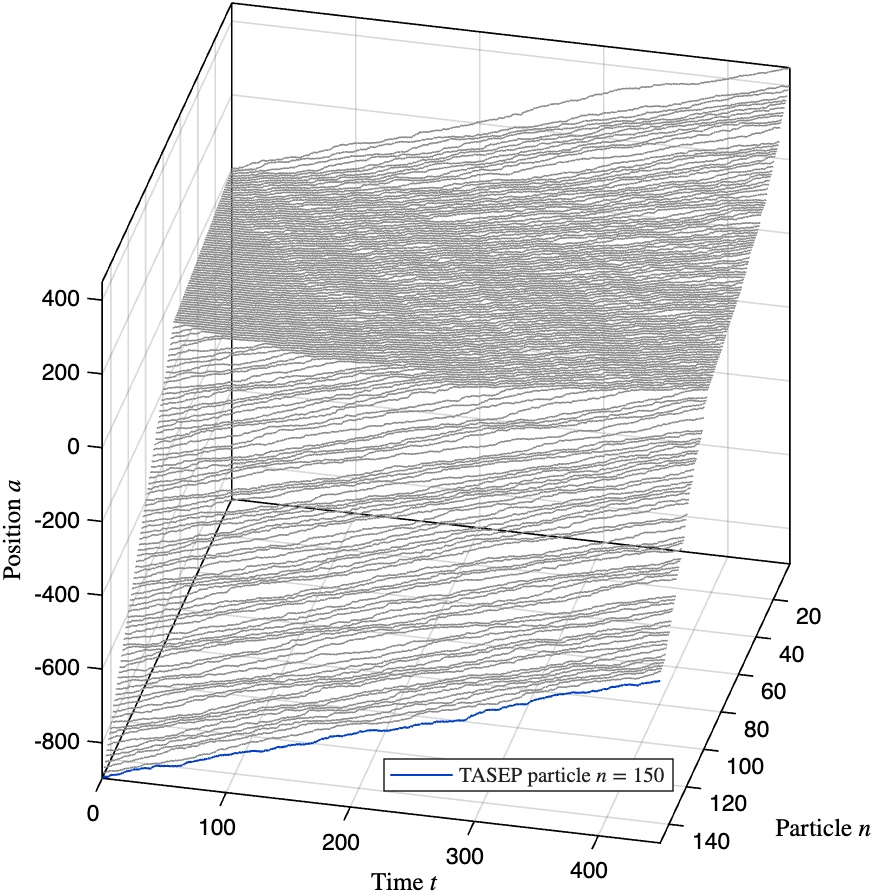}
  \end{subfigure}\hfill
  \begin{subfigure}[c]{0.26\textwidth}\centering
    \includegraphics[width=\linewidth, height=\panelH, keepaspectratio]{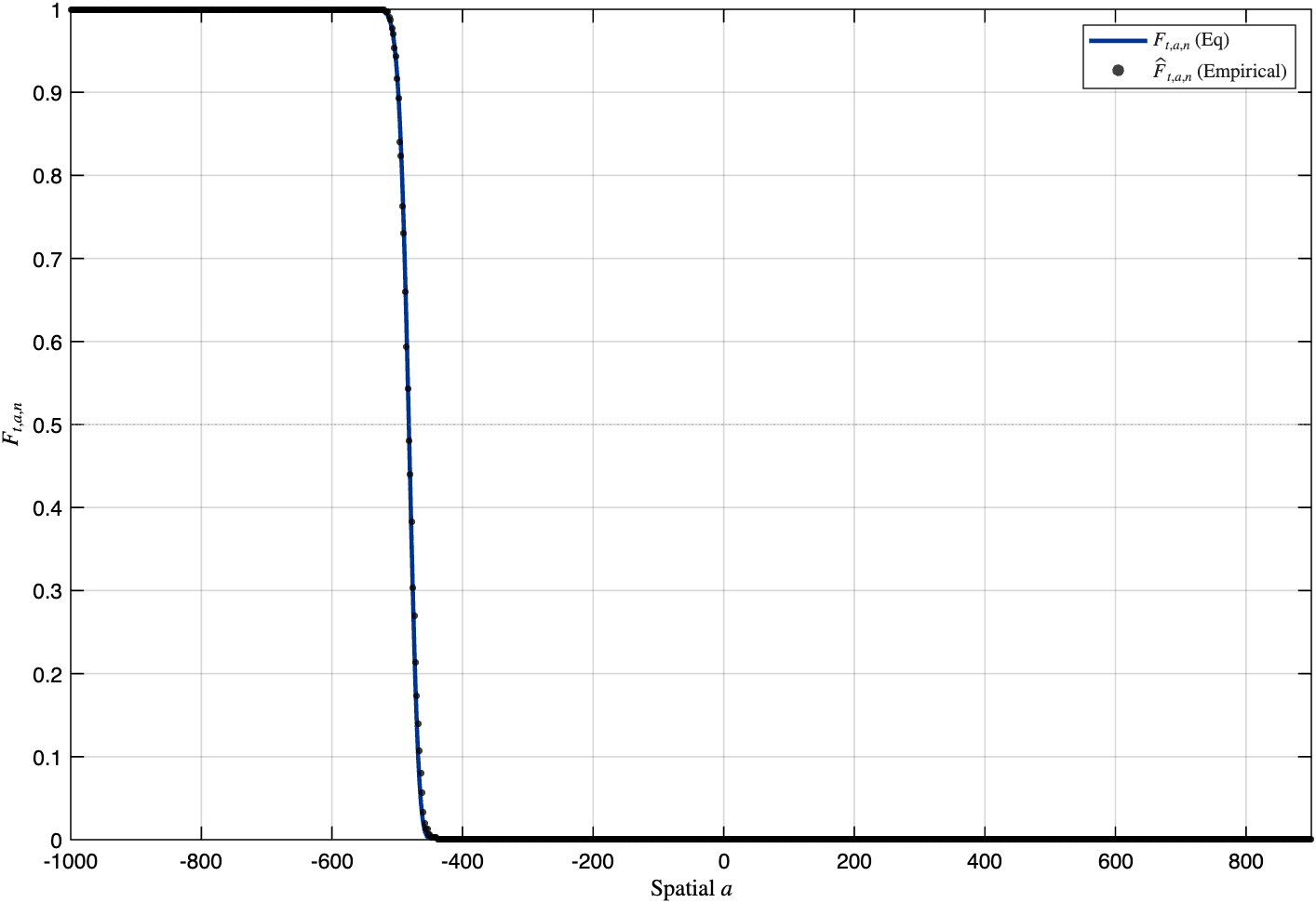}
  \end{subfigure}
  \caption{Simulations of TASEP with a shock initial condition. Left: PDE solution $F_{t,a,n}$ at $t=450$ (finite-difference scheme, initial data $F_{0, a, n} = 1_{y_n>a}$); Middle: TASEP trajectories $Y_n(\cdot)$ (initial condition $Y_n(0) = y_n)$; Right: $F_{t, a, N}$ vs.\ empirical $\hat{F}_{t, a, N}$ ($t = 450, N=150$, and $300$ independent Monte Carlo runs).}\label{fig:rtasep_panels1}
\end{figure}
\begin{corollary}[\textbf{TASEP Solutions}]
     Let $Y_n^{\text{TASEP}}(t)$ denote the $n$-th TASEP particle, and let 
     \begin{equation}
         F_{t, a, n} =  \PP\bigl(Y_{n}^{\text{TASEP}}(t) > a \, | \, \mathbf{Y}(0)= \mathbf{y}, Y_0(\cdot) = b(\cdot )\bigr).
     \end{equation}
    \begin{enumerate}
        \item(\textbf{General One-Sided Initial Condition}): Take initial data $\mathbf{y}$ as above, and set $y_0 = \infty$ so that $b(t) \equiv \infty$. Then, with $V = \R_+ \times \Z \times \Z_{\geq 1}$, $F_{t, a, n}$ satisfies $\left[D_t  - (e^{-D_a} -1)\right] F_{t, a, n} \cdot F_{t, a, n-1} = 0$, with $F_{0, a, n} = 1_{y_n > a} $ and $\bigl(\partial_t+\nabla_a^-\bigr) F_{t,a, 1}= 0$.
        \item(\textbf{TASEP with Moving Wall}): Take initial data $y_n = -n$, and suppose $Y_0(t) = b(t)$ evolves as follows: choose times $0 = s_0 < s_1 < s_2 < \dots$ with $s_k \rightarrow \infty$. At the times $s_k$, $Y_0$ jumps one unit rightward so that $Y_0(t) =  \max\{k \geq 0: s_k \leq t\}$, with $Y_0(0) =0$. Then, with  $V = \{(t, a, n) : t \in \R_+, a +n < b(t), n \in \Z_{\geq 1}\}$, $F_{t, a, n}$ satisfies $\left[D_t  - (e^{-D_a} -1)\right] F_{t, a, n} \cdot F_{t, a, n-1} = 0$, with $F_{0, a, n} = 1_{-n > a}$, $ (\partial_t + \nabla_a^-) F_{t, a, 1} = 0$ for $a < b(t)$, and $F_{t, a,1} = 0$ for $a \geq b(t)$.
    \end{enumerate}
\end{corollary}
We also present a zero-curvature/Lax pair formulation for the TASEP equation. 
\begin{theorem}[\textbf{TASEP Eq.\ Zero-Curvature Condition}]
    Fix $M \in \Z$ and a collection of non-vanishing functions $\{F_{t, a, n}\}_{a, n\in \Z}$ with boundary condition $F_{t,a,m} \equiv 1$ for all $m \leq M$. Define 
    \begin{align*}
    r_{a, n} \defeq \frac{F_{a-1, n+1}}{F_{a, n}}, 
    \end{align*} 
    and consider the operators 
    \begin{align}
        \mathcal{M} &\defeq \partial_t - \frac{r_{a, n}}{r_{a+1, n-1}}e^{\partial_a - \partial_n}, \quad 
        \bar{\mathcal{M}} \defeq e^{-\partial_a} + \frac{r_{a, n}}{r_{a, n-1}}e^{-\partial_n},
    \end{align}
    acting on functions $f_{a, n}(t)$. Then
    \begin{align}
       [\mathcal{M}, \bar{\mathcal{M}}] = 0 \Longleftrightarrow \left[D_t - (e^{-D_a} -1)\right]F_{t, a, n}\cdot F_{t, a ,n-1}= 0. 
    \end{align}
\end{theorem}
\begin{corollary}[\textbf{TASEP Eq.\ Lax Pair}] In the same setting as above, define 
\begin{align}
    L \defeq e^{-\partial_a} + \frac{r_{a, n}}{r_{a, n-1}}e^{-\partial_n}, \quad P \defeq -\frac{r_{a, n}}{r_{a+1, n-1}}e^{\partial_a -\partial_n}.
\end{align}
Then $(L, P)$ form a Lax pair for the TASEP equation, i.e. 
\begin{align*}
    \partial_t L + [P, L] = 0 \Longleftrightarrow \left[D_t - (e^{-D_a} -1)\right]F_{t, a, n}\cdot F_{t, a ,n-1}= 0.
\end{align*}
\end{corollary}
Finally, we present our results for the Parallel TASEP equation (\hyperref[tab:eqs_models]{4}). Here, we make direct contact with the HBDE as there is a change of variables which transforms equation (\hyperref[tab:eqs_models]{4}) into the HBDE form \eqref{HBDE} (see Sec.\ \ref{sec: Parallel-TASEP-Scal}). It appears, however, that our Fredholm determinant solution framework is novel. We again work in the setting of \eqref{Intro FTAN def} with now $V \subseteq \Z_{> 0} \times \Z \times \Z$.
\begin{theorem}[\textbf{Parallel TASEP Eq.\ General Solutions}]\label{Intro Parallel TASEP Eq. General Solutions}
    Suppose $K_{t, a, n}$ is in the regularity class $B_r(U, \mathcal{I}_1)$, with $U = \Z_{> 0} \times \Z \times \Z$, such that the following three conditions hold: 
    \begin{enumerate}[(i)]
        \item  $\begin{aligned}[t] 
        \nabla_a^- K_{t, a, n} &= \psi_{t, a, n} \otimes \phi_{t, a, n} .
        \end{aligned}$
        \item  $\begin{aligned}[t]
             \nabla_n^+ \psi_{t, a, n} &= \delta \psi_{t, a+1, n} - \psi_{t, a, n} - \gamma \psi_{t-1, a, n}, \medspace
             \nabla_n^- \phi_{t ,a, n} = -(\delta \phi_{t, a-1, n} -\phi_{t, a, n }  - \gamma \phi_{t+1, a, n}).
        \end{aligned}$
        \item $\begin{aligned}[t]
            \nabla_t^+ \psi_{t, a, n} &= \beta \nabla_a^- \psi_{t, a, n}, \medspace \nabla_t^- \phi_{t, a, n} = \beta\nabla_a^+ \phi_{t, a, n},
        \end{aligned}$
    \end{enumerate}
    for some arbitrary constants $\beta, \delta, \gamma$ with $\gamma \neq 0$. Then for $V \subseteq U$,  $F_{t, a, n}$ satisfies
    \begin{align}
        \left[e^{D_t} - p e^{-D_a} -(1-p) \right] F_{t, a, n} \cdot F_{t, a, n-1} = 0, \qquad p = -\frac{\beta \delta}{\gamma}, 
    \end{align}
    i.e. 
    \begin{align}
       F_{t+1,a,  n}F_{t-1, a, n-1}-pF_{t, a-1, n}F_{t, a+1, n-1} -(1-p)F_{t, a, n}F_{t, a, n-1} = 0. 
    \end{align}
\end{theorem}
We refer the reader to Section \ref{sec: Discrete TASEP gensol} for a definition of the regularity class $B_r(U, \mathcal{I}_1)$. Next, we introduce the Parallel TASEP model, which is a discrete-time variant of the TASEP model (here, we take $Y_0 \equiv \infty)$. In discrete time, at each update $t \mapsto t+1$ every particle independently attempts to jump one unit to the right with probability $p$. Each particle's jump is performed only if the destination site was empty at time $t$, otherwise it is suppressed. As a consequence of Thm.\ \ref{Intro Parallel TASEP Eq. General Solutions} and known Fredholm determinant formulas for Parallel TASEP, we have the following corollary (see Sec.\ \ref{sec: KPZ discrete tasep}). 
\begin{corollary}[\textbf{Parallel TASEP Solutions}]
    Let $Y_n^{\text{Parallel-TASEP}}(t)$ denote the $n$-th Parallel-TASEP particle, and fix one-sided initial data $\mathbf{y} = (y_1, y_2, \dots)$ with 
   \begin{align}
       F_{t,a, n} = \PP(Y_n^{\text{Parallel-TASEP}}(t) > a \, | \, \mathbf{Y}(0) = \mathbf{y}).
   \end{align}
   Then, with $V = \{(t, a, n) : t \in \Z_{>0}, a < y_n + t, n \in \Z_{\geq 1} \}$, $F_{t, a, n}$ satisfies $ \left[e^{D_t} - pe^{-D_a} - (1-p)\right] F_{t,a,n}\cdot F_{t, a, n-1} = 0$, with $F_{0, a, n} = 1_{y_n > a}$ and $(\nabla_t^+ + p\nabla_a^- )F_{t, a, 1} = 0$.
\end{corollary}
We also present a zero-curvature formulation for the Parallel TASEP equation. Alternatively, this may be derived from the well-known zero-curvature conditions of the HBDE (see \cite{Zabrodin_19972}). 
\begin{theorem}[\textbf{Parallel TASEP Eq.\ Zero-Curvature Condition}]
    Fix $M \in \Z$ and a collection of non-vanishing functions $\{F_{t, a, n}\}_{t, a, n\in \Z}$ with boundary condition $F_{t,a,m} \equiv 1$ for all $m \leq M$. Define 
\begin{align*}
    r_{t, a, n} &\defeq \frac{F_{t+1, a-1, n+1}}{F_{t, a, n}},
\end{align*}
    and consider the operators 
    \begin{align}
        \mathcal{M} &\defeq e^{\partial_t} - c \frac{r_{t, a, n}}{r_{t, a+1, n-1}}e^{\partial_a - \partial_n}, \quad 
        \bar{\mathcal{M}} \defeq - \bar{c}e^{-\partial_a} + \frac{r_{t, a, n}}{r_{t-1, a, n-1}}e^{-\partial_t - \partial_n},
    \end{align}
    acting on functions $f_{t, a, n}$, where $c, \bar{c}$ are arbitrary constants such that $c\bar{c}= p$. Then 
    \begin{align}
        [\mathcal{M}, \bar{\mathcal{M}}] = 0 \Longleftrightarrow \left[e^{D_t} -pe^{-D_a} -(1-p)\right] F_{t, a, n}\cdot F_{t, a, n-1} = 0.
    \end{align}
\end{theorem}
\begin{remark}
    Notice the Parallel TASEP equation is invariant under the ``gauge" transformation $F_{t, a, n} \mapsto g_0(n)g_1(a+n)g_2(-t+n)g_3(-t+a+n) F_{t, a, n}$, where $g_i$ are arbitrary functions, which motivates the choice of operators and coefficients in $\mathcal{M}, \bar{\mathcal{M}}$. Analogous statements can be shown to hold for the bilinear TASEP and RBM equations. This suggests a gauge reformulation of our bilinear equations in the spirit of \cite{gauge}, which we leave for future work.
\end{remark}

\subsection{Outline}
The paper is organized as follows. In the remainder of Section \ref{sec: 1} we state our notational conventions used throughout the paper. In
Section \ref{sec: 2}, we develop a Fredholm determinant solution theory for the six bilinear equations in Table \ref{tab:eqs_models}. In Section \ref{sec: 3}, we verify that the corresponding KPZ models satisfy the conditions of the solution theory presented in the previous section. In Section \ref{sec: 4}, we
derive several formal scaling limits of the bilinear equations, including the KP limit under 1:2:3 KPZ scaling. In Section \ref{sec: 5}, we present Lax pairs and zero-curvature conditions. Appendix \ref{Fredholm Determinants and Elemantary Lemmas} provides a list of Fredholm determinant identities used throughout the
paper, while Appendix \ref{sec: B} rederives the KP equation for the one-point distributions of the KPZ fixed point using the methods presented in this paper.

\subsection{Preliminaries and Conventions}
\paragraph{Trace-Class Operators} 
Let $\mathcal{H} = L^2(X, \mu)$ be a separable Hilbert space equipped with its standard inner product $\langle  \cdot , \cdot  \rangle$. Denote  by $\mathcal{I}_1(\mathcal{H})$ the  set of compact linear operators $A$ such that 
\begin{align}
    \norm{A}_{1} \defeq   \sum_n \langle \sqrt{A^*A} e_n, e_n \rangle < \infty, 
\end{align}
where $\{e_n\}_{n}$ is an orthonormal basis of $\mathcal{H}$. Of particular importance to us are the trace-class \emph{integral operators} $K$, acting on $\mathcal{H}$ via a kernel $K(x, y)$ as  
\begin{equation}
    Kf(x) = \int_X K(x, y)f(y)d\mu(y).
\end{equation}
For such $K$, the Fredholm determinant, 
\begin{align}
    \det(I+K) = \sum_{n=0}^{\infty} \frac{1}{n!} \int_{X^n} \det \bigl[ K(x_i, x_j) \bigr]_{i,j=1}^n \, d\mu(x_1)\cdots d\mu(x_n).
\end{align}
is well-defined and finite (see Appendix~\ref{Fredholm Determinants and Elemantary Lemmas}). If, in addition, $K$ is rank one, we will use the notation 
\begin{equation}
    K  = \psi \otimes \phi, \quad \text{ i.e.\ } \quad Kf(x) = \psi (x) \int_X\phi(y) f(y) d\mu(y).
\end{equation}
 More generally, given a parameter space $(Y, \nu)$, we will consider sums of rank one operators over $Y$, denoted as $A = \int_Y \psi_r \otimes \phi_r \, \nu (dr)$ and acting by $ Af(x) = \int_Y \psi_r(x) \left(\int_{X} \phi_r(y)f(y) d\mu(y) \right) d\nu(r).$
When defined, we will denote the resolvent of an operator $K$ as 
\begin{equation}
    R = (I-K)^{-1}. 
\end{equation}
\paragraph{Difference Operators and Parameter Indices}
For a function $f$ depending on a parameter $z$, denote
\begin{align*}
    \nabla_z^+ f(z) = f(z+1) - f(z), \quad \nabla_z^- f(z) = f(z) - f(z-1)
\end{align*}
and also the shift operators $e^{\partial_z}f(z) = f(z+1), e^{-\partial_z}f = f(z-1) $. For maps depending on parameters such as $F_{t, a, n}$, we will often suppress indices and simply write $F_n = F_{t, a, n}, R_n = (I-K_{t,a, n})^{-1}$, when it is clear from context the other variables are fixed. \paragraph{Topological Conventions} Common topological spaces will be assumed to have their standard topologies unless otherwise stated. When we consider product spaces with the product topology, we will refer to ``open" sets to mean that continuous spaces are open in their standard topology, and, in order to avoid overly pedantic theorem statements, that  discrete sets contain the required lattice points for the theorem (e.g. contain $n, n-1$). We denote $\bar{\R}= \R\cup\{\pm \infty\}$ and $\R_+ = (0, \infty)$. 
\subsection*{Acknowledgments}
I thank my advisor, Jeremy Quastel, for generous support, patient guidance, and many helpful discussions throughout the development of this work.
\section{Fredholm Determinant General Solutions}\label{sec: 2}
In this section we present general Fredholm determinant solutions to the bilinear equations listed in Table ~\ref{tab:eqs_models}. For ease of reference, we label each equation by the probabilistic model it corresponds to, though the solutions we present may apply more generally. Let $\mathcal{H} = L^2(X, \mu)$ denote a separable Hilbert space equipped with its standard inner product $\langle  \cdot , \cdot  \rangle$, and let $\mathcal{I}_1$ be the space  of trace-class operators on $\mathcal{H}$. 

\subsection{RBM Equation}\label{sec: RBM equation}
Let $U = I \times \R \times \{M, \dots, N\}$, where $M < N \in \Z \cup \{\pm \infty\}$, and $I \subseteq \R$ is open. We say a family of trace-class integral operators $A_{t, a, n} \in C_r^{1, 2}(U, \mathcal{I}_1)$ if for all $(t, a, n) \in U$, $A_{t,a, n}$ is trace-class, the map $t \mapsto A_{t,\cdot,\cdot}$ is $C^1$ in trace norm, the map $a \mapsto A_{\cdot, a, \cdot}$ is $C^2$ in trace norm, and for each $(t, a, n) \in U$ and a.e.\ $(x, y) \in X \times X$,  $\lim_{r\rightarrow -\infty}  A_{t,r, n}(x, y) = 0$ with
\begin{equation}
    \int_{-\infty}^{a}  \abs{\partial_t A_{t,r, n}(x, y)}  + \abs{\partial_r A_{t, r, n}(x, y)} + \abs{\partial_r^2 A_{t, r, n}(x, y)}dr < \infty.
\end{equation}

\begin{theorem}[\textbf{RBM Eq.\ General Solutions}]\label{RBM Kernel Thm}
    Let $K_{t, a, n} \in C^{1, 2}_r \bigl(U,    \mathcal{I}_1\bigr)$ be a family of trace-class integral operators acting on $ L^2(X, \mu)$ such that the following three conditions hold: 
    \begin{enumerate}
        \item\label{BLPP DECOMP} ($a$--flows): 
       $\begin{aligned}[t]
            \partial_a K_{t, a, n} &= \psi_{t, a, n} \otimes \phi_{t, a, n}.  
        \end{aligned}$
        \item ($n$--flows): $\begin{aligned}
            \nabla_n^+ \psi_{t, a, n} = \partial_a \psi_{t, a, n}, \quad \nabla_n^- \phi_{t ,a, n} = \partial_a \phi_{t, a, n}.
        \end{aligned}$
        \item\label{RBM t flows}($t$--flows): $\begin{aligned}
            \partial_t \psi_{t, a, n} &= \frac{1}{2}(\partial_a^2 + 2\partial_a + I)\psi_{t, a, n}, \medspace  \partial_t \phi_{t, a, n} = -\frac{1}{2}(\partial_a^2 - 2\partial_a + I)\phi_{t, a, n}.
        \end{aligned}$
    \end{enumerate}
Suppose further $I- K_{t, a, n}$ is invertible for all $(t, a, n) \in V$, with $V \subseteq U$ open and $n>M$. Then 
\begin{equation}
    F_{t, a, n} = \det(I-K_{t, a, n})_{L^2(X, \mu)}, \qquad (t, a, n) \in V
\end{equation} satisfies
    \begin{align}\label{RBM Formula}
        \Bigl[D_t - \frac{1}{2}D_a^2 \Bigr]F_{t, a, n}\cdot F_{t, a, n-1} = 0. 
    \end{align}
\end{theorem}
\begin{remark}
    Note that \eqref{RBM Formula} can be written out explicitly as
    \begin{align}
        F_{t, a, n-1} \partial_t F_{t,a ,n} - F_{t, a, n} \partial_t F_{t, a, n-1}  - \frac{1}{2}F_{t, a, n-1}\partial_a^2 F_{t,a, n} + \partial_a F_{t, a, n} \partial_a F_{t,a ,n-1} - \frac{1}{2}F_{t, a, n} \partial_a^2 F_{t, a, n-1} = 0. 
    \end{align}
    Moreover, making the substitution $G_{t, a, n} = \log F_{t, a, n}$, we can rewrite the equation as 
    \begin{align}
        \partial_t (\nabla_n^- G_{t,a, n})  - \frac{1}{2}(\partial_a^2 G_{t, a, n} + \partial_a^2 G_{t,a, n-1}) - \frac{1}{2}(\partial_a \nabla_n^- G_{t, a,n})^2 = 0 .
    \end{align}
    Indeed, this was the original form the equation was discovered in before it was realized it could be put into bilinear form. 
\end{remark}

\begin{proof}[Proof of Thm.\ \ref{RBM Kernel Thm}]
First, notice that due to \eqref{BLPP DECOMP}--\eqref{RBM t flows} and since  $K_{t, a, n} \in C_r^{1, 2}(U, \mathcal{I}_1)$, we have
\begin{align}
    K_{t, a, n} &= \int_{-\infty}^a \psi_{t, r, n} \otimes \phi_{t, r, n}\, dr, \label{BLPP a cond}\\  
        \nabla_n^- K_{t, a, n} &= \psi_{t, a, n-1} \otimes \phi_{t, a, n}\label{BLPP n cond},\\
        \partial_t K_{t, a, n} &= \frac{1}{2}( \psi_{t, a, n+1} \otimes \phi_{t, a, n} + \psi_{t, a, n} \otimes \phi_{t, a, n-1}). \label{BLPP t cond}
\end{align}
To see \eqref{BLPP n cond}, we compute
    \begin{align*}
        \nabla_n^- K_{t,a, n} 
        &= \int_{-\infty}^a \nabla_n^- \psi_{t, r, n} \otimes \phi_{t, r, n} +  \psi_{t, r, n-1} \otimes \nabla_n^- \phi_{t, r, n}\, dr = \int_{-\infty}^a \partial_r( \psi_{t, r, n-1} \otimes \phi_{t, r, n}) dr \\
        &= \psi_{t, a, n-1} \otimes \phi_{t, a, n}.
    \end{align*}
    To see \eqref{BLPP t cond}, we compute
    \begin{align*}
        \partial_t K_{t, a, n} &= \int_{-\infty}^a \partial_t (\psi_{t, r, n} \otimes \phi_{t, r, n})  \\
        &= \frac{1}{2}\int_{-\infty}^a \partial_r \bigl[\partial_r \psi_{t, r, n} \otimes \phi_{t, r,n}  + 2 \psi_{t, r, n} \otimes \phi_{t, r, n}- \psi_{t, r, n} \otimes \partial_r \phi_{t, r, n}\bigr] dr\\
        &= \frac{1}{2}(\partial_a \psi_{t, a, n} \otimes \phi_{t, a, n} + 2\psi_{t, a, n} \otimes \phi_{t, a, n} - \psi_{t, a, n} \otimes \partial_a \phi_{t, a, n}) \\ 
        &= \frac{1}{2}( \psi_{t, a, n+1} \otimes \phi_{t, a, n} + \psi_{t, a, n} \otimes \phi_{t, a, n-1})
    \end{align*}
    as desired. With the help of Lem.\ ~\ref{Fredholm det der}–\ref{lem:resolvent-derivative}, we can now compute
    \begin{align*}
        \partial_t F_n &= -\frac{1}{2}F_n \left( \langle R_n \psi_{n+1}, \phi_n \rangle + \langle R_n \psi_n, \phi_{n-1} \rangle \right),  \\
        \partial_a F_n &= -F_n \langle R_n \psi_n, \phi_n \rangle,  \\
        \partial_a^2 F_n &= -F_n \left(\langle R_n \psi_{n+1}, \phi_n \rangle - \langle R_n \psi_n, \phi_{n-1}\rangle \right).
    \end{align*}
    The first two identities follow directly, and the last one follows from the computation $ \partial_a^2 F_n = -F_n ( -\langle R_n \psi_n, \phi_n \rangle^2 + \langle R_n \psi_n, \phi_n \rangle^2 + \langle R_n \psi_{n+1}, \phi_n \rangle -\langle R_n \psi_n, \phi_n \rangle + \langle R_n \psi_n, \phi_n \rangle - \langle R_n \psi_n, \phi_{n-1}\rangle)$. 
    Now, notice that 
    \begin{align*}
        (\partial_t - \frac{1}{2}\partial_a^2)F_{n} &= -F_{ n}\langle R_n \psi_n, \phi_{n-1}\rangle, \\
        -(\partial_t + \frac{1}{2}\partial_a^2)F_{ n-1} &= F_{n-1}\langle R_{n-1}\psi_{n}, \phi_{n-1}\rangle.
    \end{align*}
Therefore, we have 
\begin{align*}
    &\frac{1}{F_{t, a, n} F_{t, a, n-1}}(D_t F_{t, a , n} \cdot F_{t, a, n-1} - \frac{1}{2}D_a^2 F_{t, a, n} \cdot F_{t,a,n-1}) \\
    &= - \langle R_n \psi_n, \phi_{n-1}\rangle + \langle R_{n-1}\psi_{n}, \phi_{n-1}\rangle + \langle R_n \psi_n, \phi_n\rangle \langle R_{n-1}\psi_{n-1}, \phi_{n-1}\rangle \\
    &= -\frac{F_n}{F_{n-1}}\langle R_n \psi_{n-1}, \phi_{n-1}\rangle \langle R_n \psi_n, \phi_n\rangle +  \langle R_n \psi_n, \phi_n\rangle \langle R_{n-1}\psi_{n-1}, \phi_{n-1}\rangle \\
    &= - \langle R_{n-1} \psi_{n-1}, \phi_{n-1}\rangle \langle R_n \psi_n, \phi_n\rangle +  \langle R_n \psi_n, \phi_n\rangle \langle R_{n-1}\psi_{n-1}, \phi_{n-1}\rangle  \\
    &= 0, 
\end{align*}
where on the third and fourth lines we used Lem.\ ~\ref{Fredholm fg} with $f = \psi_n, g = \phi_{n-1}$ and $\psi= \psi_{n-1}, g = \phi_{n-1}$, respectively. 
\end{proof}

\subsection{TASEP Equations}\label{sec: TASEP gensol}
Let $U = I \times \Z \times \{M, \dots, N\}$, where $M < N \in \Z \cup \{\pm \infty\}$, and $I \subseteq \R$ is open. We say a family of trace-class integral operators $A_{t, a, n} \in C_r^1(U, \mathcal{I}_1)$ if for all $(t, a, n) \in U$, $A_{t,a,n}$ is trace-class, the map $t \mapsto A_{t, \cdot, \cdot}$ is $C^1$ in trace norm, and for each $(t, a, n) \in U$ and a.e.\ $(x, y) \in X \times X$, $\lim_{r\rightarrow -\infty} A_{t,r, n}(x, y) = 0$ with   
\begin{align}
    \sum_{r=-\infty}^{a} \abs{\partial_t A_{t,r, n}(x, y)} + \abs{\nabla_r^+ A_{t, r, n}(x, y)} < \infty. 
\end{align}
\begin{theorem}[\textbf{TASEP Eq.\ General Solutions}]\label{TASEP K THM}
    Let $K_{t, a, n} \in C^1_r(U, \mathcal{I}_1)$ be a family of trace-class integral operators acting on $L^2(X, \mu)$ such that the following three conditions hold: 
    \begin{enumerate}
        \item ($a$--flows): $\begin{aligned}[t] 
        \nabla_a^- K_{t, a, n} &= \psi_{t, a, n} \otimes \phi_{t, a, n}. \label{TASEP grad a decomp}
        \end{aligned}$
        \item\label{TASEP n flows} ($n$--flows): $\begin{aligned}[t]
             \nabla_n^+ \psi_{t, a, n} = 2 \nabla_a^+ \psi_{t, a, n}, \medspace \nabla_n^- \phi_{t ,a, n} = 2 \nabla_a^- \phi_{t, a, n}.
        \end{aligned}$
        \item\label{TASEP t flows} ($t$--flows): $\begin{aligned}[t]
            \partial_t \psi_{t, a, n} &= -\frac{1}{2}\nabla_a^- \psi_{t, a, n}, \medspace \partial_t \phi_{t, a, n} = -\frac{1}{2}\nabla_a^+ \phi_{t, a, n}.
        \end{aligned}$
    \end{enumerate}
   Suppose further $I- K_{t, a, n}$ is invertible for all $(t, a, n)\in V$, with $V \subseteq U$ open and $n > M$. Then 
\begin{equation}
    F_{t, a, n} = \det(I-K_{t, a, n})_{L^2(X, \mu)}, \qquad (t, a, n) \in V
\end{equation} satisfies
    \begin{align}
        \left[D_t - (e^{-D_a} - 1) \right] F_{t, a, n} \cdot F_{t, a, n-1} = 0. \label{TASEP FORM}
    \end{align}
\end{theorem}
\begin{remark}
    Note that \eqref{TASEP FORM} can be written out explicitly as 
    \begin{align*}
        F_{t, a, n-1}\partial_t F_{t, a, n} - F_{t, a, n} \partial_t F_{t, a, n-1} - F_{t, a-1, n}F_{t, a+1, n-1} + F_{t, a, n}F_{t, a, n-1} = 0.
    \end{align*}
\end{remark}
\begin{proof}[Proof of Thm.\ \ref{TASEP K THM}]
    First, notice that due to \eqref{TASEP grad a decomp}--\eqref{TASEP t flows} and since $K_{t ,a, n} \in C_r^{1}(U, \mathcal{I}_1)$, we have 
    \begin{align}
        K_{t, a, n} &= \sum_{r = -\infty}^a \psi_{t, r, n} \otimes \phi_{t, r, n}, \\
        \nabla_n^- K_{t, a, n} &= 2 \psi_{t, a+1, n-1} \otimes  \phi_{t, a, n} \label{TASEP grad n on K},\\
        \partial_t K_{t, a, n} &= -\frac{1}{2}\psi_{t, a, n} \otimes \phi_{t, a+1, n}. \label{TASEP partial t}
    \end{align}
    To see \eqref{TASEP grad n on K}, we compute 
    \begin{align*}
        K_{t, a, n} - K_{t, a, n-1} 
        &= \sum_{r= -\infty}^a \nabla_n^- \psi_{t, r, n} \otimes \phi_{t, r, n} + \psi_{t, r, n-1}\otimes \nabla_n^- \phi_{t, r, n} =  2 \sum_{r=-\infty}^a \nabla_r^+ (\psi_{t, r, n-1} \otimes \phi_{t, r-1, n}) \\
        &= 2 \psi_{t, a+1, n-1} \otimes \phi_{t, a, n},
    \end{align*}
    To see \eqref{TASEP partial t}, we compute
    \begin{align*}
        \partial_t K_{t, a, n} &= \sum_{r=-\infty}^a \partial_t \psi_{t, r, n} \otimes \phi_{t, r, n} + \psi_{t,r, n} \otimes \partial_t \phi_{t, r, n} = -\frac{1}{2}\sum_{r=-\infty}^a  \nabla_r^+ (\psi_{t, r-1, n} \otimes \phi_{t, r, n}) \\
        &= -\frac{1}{2}\psi_{t, a, n} \otimes \phi_{t, a+1, n}.
    \end{align*}
    With the help of Lem.\ \ref{Fredholm det der}--\ref{lem:resolvent-derivative}, we can now compute
    \begin{align*}
        \frac{1}{F_nF_{n-1}}D_t F_n \cdot F_{n-1} &= \frac{1}{2}\langle R_n \psi_{a, n}, \phi_{a+1, n} \rangle - \frac{1}{2}\langle R_{n-1} \psi_{a, n-1}, \phi_{a+1, n-1} \rangle.
    \end{align*}
    In addition, using Lem.\ \ref{Fredholm rank one} and the $n$--flows, we have 
    \begin{align*}
        \frac{F_{a-1, n}F_{a+1, n-1}}{F_n F_{n-1}} &= \bigl(1+\langle R_n \psi_{a, n}, \phi_{a, n}\rangle \bigr)\bigl(1-\langle R_{n-1}\psi_{a+1, n-1}, \phi_{a+1, n-1}\rangle \bigr) \\
        &= 1 + \langle R_n \psi_{a, n}, \phi_{a, n}\rangle-\langle R_{n-1}\psi_{a+1, n-1}, \phi_{a+1, n-1}\rangle \\
        &\qquad - \langle R_n \psi_{a, n}, \phi_{a, n}\rangle\langle R_{n-1}\psi_{a+1, n-1}, \phi_{a+1, n-1}\rangle \\
        &= 1 + \tfrac{1}{2}\langle R_n \psi_{a, n}, \phi_{a+1, n}\rangle + \tfrac{1}{2}\langle R_n\psi_{a, n}, \phi_{a+1, n-1}\rangle \\
        &\qquad - \tfrac{1}{2}\langle R_{n-1}\psi_{a, n}, \phi_{a+1, n-1}\rangle - \tfrac{1}{2}\langle R_{n-1}\psi_{a, n-1}, \phi_{a+1, n-1}\rangle \\
        &\qquad  - \langle R_n \psi_{a, n}, \phi_{a, n}\rangle\langle R_{n-1}\psi_{a+1, n-1}, \phi_{a+1, n-1}\rangle.
    \end{align*}
    Therefore, we have
    \begin{align*}
        & \frac{1}{F_{t, a, n} F_{t, a, n-1}}\left[D_t F_{t,a, n} \cdot F_{t,a, n-1} - F_{t, a-1, n}F_{t, a+1, n-1} +F_{t,a,n}F_{t, a, n-1}\right] \\
        &= -\frac{1}{2}\langle R_n \psi_{a, n}, \phi_{a+1, n-1}\rangle + \frac{1}{2}\langle R_{n-1} \psi_{a, n}, \phi_{a+1, n-1}\rangle + \langle R_n \psi_{a, n}, \phi_{a, n}\rangle\langle R_{n-1}\psi_{a+1, n-1}, \phi_{a+1, n-1}\rangle \\
        &= 0. 
    \end{align*}
    The last line follows since by using Lem.\ \ref{Fredholm fg} with $\psi = 2\psi_{a+1, n-1}, \phi = \phi_{a, n}$, we have
    \begin{align*}
        &-\frac{1}{2}\langle R_n \psi_{a, n}, \phi_{a+1, n-1}\rangle + \frac{1}{2}\langle R_{n-1} \psi_{a, n}, \phi_{a+1, n-1}\rangle \\
        &= -\frac{F_n}{F_{n-1}}\langle R_n \psi_{a+1, n-1}, \phi_{a+1, n-1}\rangle\langle R_n \psi_{a, n}, \phi_{a, n}\rangle \\
        &= - \langle R_{n-1}\psi_{a+1, n-1}, \phi_{a+1, n-1}\rangle \langle R_n \psi_{a, n}, \phi_{a, n}\rangle.
    \end{align*}
\end{proof}
\begin{theorem}[\textbf{Push-TASEP Eq.\ General Solutions}]\label{PUSHTASEP K thm}
   Let $K_{t, a, n} \in C^1_r(U, \mathcal{I}_1)$ be a family of trace-class integral operators acting on $L^2(X, \mu)$ such that the following three conditions hold: 
       \begin{enumerate}
        \item\label{Push-Tasep a flows} ($a$--flows): $\begin{aligned}[t]
            \nabla_a^- K_{t, a, n} &= \psi_{t, a, n} \otimes \phi_{t, a,n}.
        \end{aligned}$ 
    \item\label{Push-TASEP n flows} ($n$--flows): $
    \begin{aligned}[t]
        \nabla_n^+ \psi_{t, a, n } &= 2 \nabla_a^+ \psi_{t, a, n}, \medspace \nabla_n^- \phi_{t, a, n} = 2 \nabla_a^- \phi_{t, a, n}.
    \end{aligned} $
    \item\label{Push-TASEP t flows} ($t$--flows): $
    \begin{aligned}[t]
        \partial_t \psi_{t, a, n} =  2  \nabla_a^+ \psi_{t, a, n}, \medspace \partial_t \phi_{t, a, n} =   2 \nabla_a^- \phi_{t, a, n}.
    \end{aligned}$
    \end{enumerate}
   Suppose further $I- K_{t, a, n}$ is invertible for all $(t, a, n)\in V$, with $V \subseteq U$ open and $n > M$. Then 
\begin{equation}
    F_{t, a, n} = \det(I-K_{t, a, n})_{L^2(X, \mu)}, \qquad (t, a, n) \in V
\end{equation} satisfies
    \begin{align}
        \left[D_t -  (e^{D_a} - 1)\right]F_{t, a, n} \cdot F_{t, a+1, n-1} &= 0. \label{Push Tasep FOrm} 
    \end{align}
\end{theorem}
\begin{remark}
    Note that \eqref{Push Tasep FOrm} can be written out explicitly as 
    \begin{align*}
        F_{t, a+1, n-1}\partial_t F_{t,a, n} - F_{t, a, n}\partial_t F_{t, a+1, n-1} - F_{t, a+1, n}F_{t, a, n-1} + F_{t, a, n}F_{t, a+1, n-1} = 0. 
    \end{align*}
\end{remark}
\begin{proof}[Proof of Thm.\ \ref{PUSHTASEP K thm}]
    First, notice that due to \eqref{Push-Tasep a flows}--\eqref{Push-TASEP t flows} and since $K_{t, a, n} \in C_r^1(U, \mathcal{I}_1)$, we have
    \begin{align}
        K_{t, a, n} &= \sum_{r= -\infty}^a \psi_{t, r, n}\otimes \phi_{t, r,n}, \\
        \nabla_n^- K_{t, a, n} &= 2 \psi_{t, a+1, n-1} \otimes \phi_{t, a, n}, \\
        \partial_t K_{t, a, n} &=  2 \psi_{t, a+1, n} \otimes \phi_{t, a, n}, \label{PushTASEp dev on t}
    \end{align}
    where \eqref{PushTASEp dev on t} follows from 
    \begin{align*}
    \partial_t K_{t, a, n} &=2 \sum_{r=-\infty}^a \psi_{t, r+1, n} \otimes \phi_{t, r, n} - \psi_{t, r, n}\otimes \phi_{t, r-1, n} =  2 \psi_{t, a+1, n}\otimes \phi_{t, a, n}.
\end{align*}
 With the help of Lem.\ \ref{Fredholm det der}--\ref{lem:resolvent-derivative}, we can now compute
\begin{align*}
    \frac{1}{F_{a, n}F_{a+1, n-1}}D_t F_{a, n} \cdot F_{a+1, n-1} = -2 \langle R_{a, n} \psi_{a+1, n}, \phi_{a, n}\rangle + 2 \langle R_{a+1, n-1}\psi_{a+2, n-1}, \phi_{a+1, n-1}\rangle.
\end{align*}
In addition, using Lem.\ \ref{Fredholm rank one} and the $n$--flows, we have 
\begin{align*}
    \frac{F_{a+1, n}F_{a, n-1}}{F_{a, n}F_{a+1, n-1}} &= \bigl(1- \langle R_n \psi_{a+1, n}, \phi_{a+1, n}\rangle \bigr)\bigl(1+ \langle R_{a+1, n-1}\psi_{a+1, n-1}, \phi_{a+1, n-1}\rangle \bigr) \\
    &= 1 - \langle R_n \psi_{a+1, n}, \phi_{a+1, n}\rangle+ \langle R_{a+1, n-1}\psi_{a+1, n-1}, \phi_{a+1, n-1}\rangle \\
    &\qquad - \langle R_n \psi_{a+1, n}, \phi_{a+1, n}\rangle\langle R_{a+1, n-1}\psi_{a+1, n-1}, \phi_{a+1, n-1}\rangle \\
    &= 1 - 2\langle R_n \psi_{a+1, n}, \phi_{a, n}\rangle + \langle R_n \psi_{a+1, n}, \phi_{a+1, n-1}\rangle \\
    &\qquad + 2\langle R_{a+1, n-1}\psi_{a+2, n-1}, \phi_{a+1, n-1}\rangle - \langle R_{a+1, n-1}\psi_{a+1, n}, \phi_{a+1, n-1}\rangle \\
    &\qquad - \langle R_n \psi_{a+1, n}, \phi_{a+1, n}\rangle\langle R_{a+1, n-1}\psi_{a+1, n-1}, \phi_{a+1, n-1}\rangle
\end{align*}
Therefore, we have 
\begin{align*}
    &\frac{1}{F_{t, a, n} F_{t, a+1, n-1}} \left[ D_t F_{ t, a, n}\cdot F_{t, a+1, n-1} -  F_{t, a+1, n}F_{t, a, n-1} +  F_{t, a, n}F_{t, a+1, n-1} \right] \\
    &= - \langle R_{n} \psi_{a+1, n}, \phi_{a+1, n-1}\rangle +  \langle R_{a+1, n-1}\psi_{a+1, n}, \phi_{a+1, n-1}\rangle \\
    &\quad +  \langle R_n \psi_{a+1, n}, \phi_{a+1, n}\rangle\langle R_{a+1, n-1}\psi_{a+1, n-1}, \phi_{a+1, n-1}\rangle \\
    &= 0,
\end{align*}
where the last line follows from  the fact 
$K_{t, a, n} - K_{t, a+1, n-1} = \psi_{t, a+1, n-1} \otimes \phi_{t, a+1, n}$, and so applying Lem.\ \ref{Fredholm fg} we have 
\begin{align*}
    &- \langle R_{n} \psi_{a+1, n}, \phi_{a+1, n-1}\rangle +  \langle R_{a+1, n-1}\psi_{a+1, n}, \phi_{a+1, n-1}\rangle \\
    &= -  \frac{F_{n}}{F_{a+1, n-1}}\langle R_n \psi_{a+1, n}, \phi_{a+1, n}\rangle \langle R_{n} \psi_{a+1, n-1}, \phi_{a+1, n-1}\rangle \\
    &= - \langle R_n \psi_{a+1, n}, \phi_{a+1, n}\rangle\langle R_{a+1, n-1}\psi_{a+1, n-1}, \phi_{a+1, n-1}\rangle.
\end{align*}

\end{proof}

\subsection{Discrete-Time TASEP Equations}\label{sec: Discrete TASEP gensol}
Let $U = I \times \Z \times \{M, \dots, N\}$, where $M < N \in \Z \cup \{\pm \infty\}$, and $I \subseteq \Z$. We say a family of trace-class integral operators $A_{t, a, n} \in B_r(U, \mathcal{I}_1)$ if for each $(t,a , n) \in U$, $A_{t,a,n}$ is trace-class and for a.e.\ $(x, y) \in X \times X$, $\lim_{r \rightarrow -\infty} A_{t,r, n}(x, y) = 0$ with 
\begin{align}
    \sum_{r=-\infty}^{a} \abs{\nabla_t^+ A_{t, r, n}(x, y)}+ \abs{\nabla_r^+ A_{t, r, n}(x, y)}  < \infty. 
\end{align}
\begin{theorem}[\textbf{Parallel TASEP Eq.\ General Solutions}]\label{PTASEP K THM}
    Let $K_{t, a, n} \in B_r(U, \mathcal{I}_1)$ be a family of trace-class integral operators acting on $L^2(X, \mu)$ such that the following three conditions hold: 
    \begin{enumerate}
        \item \label{PTASEP grad a decomp}($a$--flows): $\begin{aligned}[t] 
        \nabla_a^- K_{t, a, n} &= \psi_{t, a, n} \otimes \phi_{t, a, n} .
        \end{aligned}$
        \item\label{Parallel TASEP n flows} ($n$--flows): $\begin{aligned}[t]
             \nabla_n^+ \psi_{t, a, n} &= \delta \psi_{t, a+1, n} - \psi_{t, a, n} - \gamma \psi_{t-1, a, n}, \\
             \nabla_n^- \phi_{t ,a, n} &= -(\delta \phi_{t, a-1, n} -\phi_{t, a, n }  - \gamma \phi_{t+1, a, n}).
        \end{aligned}$
        \item\label{Parallel TASEP t flows}($t$--flows): $\begin{aligned}[t]
            \nabla_t^+ \psi_{t, a, n} &= \beta \nabla_a^- \psi_{t, a, n}, \medspace \nabla_t^- \phi_{t, a, n} = \beta\nabla_a^+ \phi_{t, a, n},
        \end{aligned}$
    \end{enumerate}
    for some arbitrary constants $\beta, \delta, \gamma$ with $\gamma \neq 0$. Suppose further $I-K_{t, a, n}$ is invertible for all $(t, a, n) \in V \subseteq U$ open with $n > M$.  Then 
\begin{equation}
    F_{t, a, n} = \det(I-K_{t, a, n})_{L^2(X, \mu)}, \qquad (t, a, n) \in V
\end{equation} satisfies
    \begin{align}
        \left[e^{D_t} - p e^{-D_a} -(1-p) \right] F_{t, a, n} \cdot F_{t, a, n-1} = 0, \qquad p = -\frac{\beta \delta}{\gamma}. \label{PTASEP FORM}
    \end{align}
\end{theorem}
\begin{remark}
    Note that \eqref{PTASEP FORM} can be written explicitly as 
    \begin{align}
       F_{t+1,a,  n}F_{t-1, a, n-1}-pF_{t, a-1, n}F_{t, a+1, n-1} -(1-p)F_{t, a, n}F_{t, a, n-1} = 0. 
    \end{align}
\end{remark}
\begin{proof}[Proof of Thm.\ \ref{PTASEP K THM}]
    First, notice that due to \eqref{PTASEP grad a decomp}--\eqref{Parallel TASEP t flows} and since $K_{t, a, n} \in B_r(U, \mathcal{I}_1)$, we have
    \begin{align}
        K_{t, a, n} &= \sum_{r= -\infty}^a \psi_{t, r, n} \otimes \phi_{t, r, n} ,\\
        (\nabla_n^--\nabla_a^-) K_{t,a, n} &= (1+\beta)\gamma \psi_{t-1, a, n-1} \otimes \phi_{t+1, a, n}, \label{PTASEP grad n K}\\
        \nabla_t^+ K_{t, a, n} &= \beta \psi_{t,a, n} \otimes \phi_{t+1, a+1, n} \label{PTASEP grad t K}.
    \end{align}
    To see \eqref{PTASEP grad t K}, note we have 
\begin{align*}
    \nabla_t^+ K_{t, a, n} &= \sum_{r= -\infty}^a \nabla_t^+ \psi_{t, r, n} \otimes \phi_{t+1, r, n} + \psi_{t, r, n} \otimes \nabla_t^+ \phi_{t, r, n} \\
    &= \beta \sum_{r=-\infty}^a \nabla_r^+ (\psi_{t, r-1, n}\otimes \phi_{t+1, r, n})\\
    &= \beta \psi_{t, a, n} \otimes \phi_{t+1, a+1, n}.
\end{align*}
   Next, from the discrete flows, we can derive the identities
\begin{alignat}{2}
     \delta \psi_{t, a+1, n-1} - \gamma \psi_{t-1, a, n-1}  &= \psi_{t, a, n}   &&=  (1+\beta)\psi_{t-1, a, n} - \beta \psi_{t-1, a-1, n}, \label{p-flows identites 1}\\
     \delta \phi_{t, a-1, n} - \gamma \phi_{t+1, a, n}  &= \phi_{t, a, n-1}  &&= (1+\beta)\phi_{t+1, a, n-1} - \beta\phi_{t+1, a+1, n-1}. \label{p-flows identites 2}
\end{alignat}
    To see \eqref{PTASEP grad n K}, we compute 
    \begin{align*}
        \nabla_n^- K_{t, a, n} &= \sum_{r=-\infty}^a \nabla_n^- \psi_{t, r, n} \otimes \phi_{t,r, n} + \psi_{t,r,  n-1} \otimes \nabla_n^- \phi_{t, r, n} \\
        &= \sum_{r= -\infty}^a (\delta \psi_{t, r+1, n-1} -\gamma \psi_{t-1, r, n-1})\otimes \phi_{t, r, n} -  \psi_{t, r, n-1}\otimes (\delta \phi_{t, r-1, n}  -\gamma \phi_{t+1, r, n}) \\
          &= \delta \sum_{r=-\infty}^a \nabla_r^+ (\psi_{t, r, n-1} \otimes \phi_{t, r-1, n}) - \gamma\sum_{r= -\infty}^a \left(\psi_{t-1, r, n-1} \otimes \phi_{t, r, n} - \psi_{t, r, n-1} \otimes \phi_{t+1, r, n} \right).
    \end{align*}
    Labeling the $\delta$ summation by $(\mathrm{I})$ and noticing it telescopes, gives 
    \begin{align*}
        (\mathrm{I}) &=  \delta \psi_{t, a+1, n-1} \otimes \phi_{t, a, n}.
    \end{align*}
    Similarly, labeling the $\gamma$ term by $(\mathrm{II})$, and using \eqref{p-flows identites 1}--\eqref{p-flows identites 2}, gives 
    \begin{align*}
        ( \mathrm{II})&= \gamma \sum_{r= -\infty}^a \Bigl[-\left(\psi_{t-1, r, n-1} \otimes ((1+\beta)\phi_{t+1, r, n} - \beta \phi_{t+1, r+1, n})\right)  \\
        &\qquad \quad \quad \, \, \, +\left((1+\beta)\psi_{t-1, r, n-1} -\beta \psi_{t-1, r-1, n-1}) \otimes \phi_{t+1, r, n} \right)\Bigr]\\
        &=  \beta\gamma \sum_{r= -\infty}^a \nabla_r^+ (\psi_{t-1, r-1, n-1} \otimes \phi_{t+1, r, n})\\
        &= \beta \gamma \psi_{t-1, a, n-1} \otimes \phi_{t+1, a+1, n}.
    \end{align*}
    Therefore, combining these and once again using \eqref{p-flows identites 1}--\eqref{p-flows identites 2}, 
    \begin{align*}
      ( \mathrm{I}) +  ( \mathrm{II}) &= \delta \psi_{t, a+1, n-1} \otimes \phi_{t, a, n} + \beta \gamma \psi_{t-1, a, n-1} \otimes \phi_{t+1, a+1, n}\\
      &= (\delta \psi_{t, a+1, n-1} - \gamma \psi_{t-1, a, n-1}) \otimes \phi_{t, a, n} + (1+\beta)\gamma \psi_{t-1, a, n-1} \otimes \phi_{t+1, a, n} \\
        &= \psi_{t, a, n} \otimes \phi_{t,a, n} + (1+\beta)\gamma \psi_{t-1, a, n-1} \otimes \phi_{t+1, a, n}, 
\end{align*}
and so  
\begin{align*}
    (\nabla_n^- - \nabla_a^-)K_{t,a, n}= K_{t, a-1, n} - K_{t, a, n-1} &= (1+\beta)\gamma \psi_{t-1, a, n-1}\otimes \phi_{t+1, a,n}. 
\end{align*}
We will also use 
\begin{align}
    K_{t+1, a, n} - K_{t, a, n-1} &= (1+ \beta)\delta \psi_{t, a+1, n-1} \otimes \phi_{t+1, a, n} \label{parallel TASEP K t n-1},
\end{align}
which follows from
\begin{align*}
    K_{t+1, a, n} - K_{t, a, n-1} &= K_{t+1, a, n} - K_{t+1, a-1, n} + K_{t+1, a-1, n} - K_{t, a-1, n} +K_{t, a-1, n} - K_{t, a, n-1}\\
    &= (\psi_{t+1, a, n}  + \beta \psi_{t, a-1, n} + (1+\beta) \gamma \psi_{t-1, a, n-1}) \otimes \phi_{t+1, a, n} \\
    &= (1+ \beta)\delta \psi_{t, a+1, n-1} \otimes \phi_{t+1, a, n}.
\end{align*}
Next, we calculate  
\begin{align*}
    \frac{F_{t+1, a, n}}{F_{t, a, n}} &= 1 - \beta \langle R_n \psi_{n}, \phi_{t+1, a+1}\rangle \\
    &= 1-\beta \delta \gamma^{-1}\langle R_n \psi_n, \phi_n\rangle + \beta \gamma^{-1}\langle R_n \psi_n, \phi_{a+1, n-1}\rangle \\
    &= 1-\beta \delta \gamma^{-1}\langle R_n \psi_n, \phi_n\rangle + \beta \gamma^{-1} \frac{F_{t, a-1, n}}{F_{t, a, n}}\langle R_{a-1, n} \psi_n, \phi_{a+1, n-1}\rangle \\
    &= 1-\beta \delta \gamma^{-1}\langle R_n \psi_n, \phi_n\rangle - \beta \frac{F_{t, a-1, n}}{F_{t, a, n}}\langle R_{a-1,n} \psi_{t-1, n-1}, \phi_{a+1, n-1}\rangle \\
    &\quad +\beta \delta \gamma^{-1}\frac{F_{t, a-1, n}}{F_{t, a, n}}\langle R_{a-1,n} \psi_{a+1, n-1}, \phi_{a+1, n-1}\rangle, \\
    \intertext{ where the second equality follows from the LHS of \eqref{p-flows identites 2}, the third equality from Lem.\ ~\ref{Fredholm rank one} and $a$--flows, the fourth equality follows from the RHS of \eqref{p-flows identites 1}. Now, using Lem.\ ~\ref{Fredholm rank one}--Lem.\ ~\ref{Fredholm fg},  and identifying terms, we have }
     &= 1-\beta \delta \gamma^{-1}\langle R_n \psi_n, \phi_n\rangle - \beta \frac{F_{t, a-1, n}}{F_{t, a, n}}\frac{F_{t, a, n-1}}{F_{t, a-1, n}}\langle R_{n-1} \psi_{t-1, n-1}, \phi_{a+1, n-1}\rangle \\
    &\quad +\beta \delta \gamma^{-1}\frac{F_{t, a-1, n}}{F_{t, a, n}}\langle R_{n-1} \psi_{a+1, n-1}, \phi_{a+1, n-1}\rangle  \\
    &\quad +\beta \delta (1+\beta)\frac{F_{t, a, n-1}}{F_{t, a, n}}\langle R_{n-1} \psi_{t-1, n-1}, \phi_{a+1, n-1}\rangle \langle R_{n-1} \psi_{a+1, n-1}, \phi_{t+1}\rangle.\\
    &= 1+\beta \delta \gamma^{-1}\left(1-\frac{F_{a-1, n}}{F_{a, n}}\right) +  \frac{F_{a-1, n}F_{a, n-1}}{F_{a, n}F_{a-1, n}}\left(1-\frac{F_{t-1, n-1}}{F_{t, n-1}}\right) \\
    &\quad +\beta \delta \gamma^{-1}\frac{F_{a-1, n}}{F_{a, n}}\left(1-\frac{F_{a+1, n-1}}{F_{a, n-1}}\right) - \frac{F_{n-1}}{F_n}\left(1-\frac{F_{t-1, n-1}}{F_{t, n-1}}\right) \left(1-\frac{F_{t+1, n}}{F_{t, n-1}}\right).
\end{align*}
Multiplying through by $F_nF_{n-1}$, we obtain 
\begin{align*}
    F_{t+1, a, n}F_{t-1, a, n-1}+\beta\delta\gamma^{-1}F_{t, a-1, n}F_{t, a+1, n-1} -(1+ \beta \delta \gamma^{-1})F_{t, a, n}F_{t, a, n-1} = 0.
\end{align*}
\end{proof}

\begin{theorem}[\textbf{Bernoulli Jumps with Blocking Eq.\ General Solutions}]\label{RBTASEP K THM}
    Let $K_{t, a, n} \in B_r(U, \mathcal{I}_1)$ be a family of trace-class integral operators acting on $L^2(X, \mu)$ such that the following three conditions hold: 
    \begin{enumerate}
        \item ($a$--flows): $\begin{aligned}[t] 
        \nabla_a^- K_{t, a, n} &= \psi_{t, a, n} \otimes \phi_{t, a, n}. \label{RBTASEP grad a decomp}
        \end{aligned}$
        \item ($n$--flows): $\begin{aligned}[t]
             \nabla_n^+ \psi_{t, a, n} = \aaa \nabla_a^+ \psi_{t, a, n}, \medspace \nabla_n^- \phi_{t ,a, n} = \aaa \nabla_a^- \phi_{t, a, n}.
        \end{aligned}$
        \item ($t$--flows): $ \begin{aligned}[t]
            \nabla_t^+ \psi_{t, a, n} &= \beta \nabla_a^- \psi_{t, a, n}, \medspace \nabla_t^- \phi_{t, a, n} = \beta\nabla_a^+ \phi_{t, a, n},\label{RBTASEP grad t decomp}
        \end{aligned}$
    \end{enumerate}
    for some arbitrary constants $\aaa, \beta$  with $1+\beta-\aaa \neq 0$. Suppose further $I-K_{t,a, n}$ is invertible for all $(t, a, n) \in V \subseteq U$ open with $n > M$. Then
    \begin{equation}
        F_{t,a, n} = \det(I-K_{t,a, n})_{L^2(X, \mu)}, \qquad (t,a,n) \in V
    \end{equation}
    satisfies 
    \begin{align}
        \left[e^{D_t} - p e^{-D_a} - \left(1-p\right) \right] F_{t, a, n}\cdot F_{t+1, a, n-1} = 0, \qquad p = \frac{\aaa\beta}{1+\beta -\aaa} .\label{RBPTASEP FORM}
    \end{align}
\end{theorem}
\begin{remark}
    Note that \eqref{RBPTASEP FORM} can be written out explicitly as 
    \begin{align}
    F_{t+1, a, n}F_{t, a, n-1} - pF_{t, a-1, n}F_{t+1, a+1, n-1}-(1-p)F_{t, a, n}F_{t+1, a, n-1} &= 0.
\end{align} 
\end{remark}
\begin{proof}[Proof of Thm.\ \ref{RBTASEP K THM}]
     First, notice that due to \eqref{RBTASEP grad a decomp}--\eqref{RBTASEP grad t decomp} and since $K_{t, a, n} \in B_r(U, \mathcal{I}_1)$, we have 
    \begin{align}
        K_{t, a, n} &= \sum_{r=-\infty}^a \psi_{t, r, n} \otimes \phi_{t, r, n},\\
        \nabla_n^- K_{t, a, n} &= \aaa \psi_{t, a+1, n-1} \otimes \phi_{t, a, n}, \\
        \nabla_t^+ K_{t, a, n} &= \beta \psi_{t, a, n} \otimes \phi_{t+1, a+1, n}. \label{RB timn K}
    \end{align}
    To see \eqref{RB timn K}, notice
    \begin{align*}
        \nabla_t^+K_{t, a, n} &= \sum_{r=-\infty}^a \nabla_t^+  \psi_{t, r, n} \otimes \phi_{t+1, r, n} + \psi_{t, r, n} \otimes \nabla_t^+ \phi_{t, r, n} = \beta\sum_{r=-\infty}^a \nabla_r^+ (\psi_{t, r-1, n} \otimes \phi_{t+1, r, n}) \\
        &= \beta \psi_{t, a, n} \otimes \phi_{t+1, a+1, n}.
    \end{align*}
    Next, from the discrete flows it is not hard to see the identities
    \begin{align}
        \psi_{t, a, n} 
        &= c_1 \psi_{t+1, a+1, n-1} + c_2 \psi_{t, a, n-1}, \qquad 
        \phi_{t+1, a+1, n}
        = c_3 \phi_{t+1,a+1, n-1}  -c_4 \phi_{t, a, n}, \label{RBTASEP identities}
    \end{align}
    with 
    \begin{align*}
        c_1 = \frac{\aaa}{1+\beta}, \quad c_2 = \frac{1+\beta -\aaa}{1+\beta}, \quad c_3 = \frac{1+\beta}{1+\beta-\aaa}, \quad  c_4 = \frac{\aaa}{1+\beta-\aaa}.
    \end{align*}
    With the help of Lem.\ ~\ref{Fredholm rank one}--~\ref{Fredholm fg} and \eqref{RBTASEP identities}, we compute
    \begin{align*}
        \frac{F_{t, a, n}}{F_{t+1, a, n}} &= 1+\beta \langle R_{t+1,a, n} \psi_{t, a, n}, \phi_{t+1, a+1, n}\rangle \\
        &= 1 - \beta c_4 \langle R_{t+1, n} \psi_n, \phi_{n}\rangle + \beta c_3 \langle R_{t+1, n} \psi_n, \phi_{t+1, a+1, n-1}\rangle \\ 
        &= 1 - \beta c_4 \langle R_{t+1, n} \psi_n, \phi_{n}\rangle + \beta c_3c_1 \langle R_{t+1, n} \psi_{t+1, a+1, n-1}, \phi_{t+1, a+1, n-1}\rangle \\
        &\quad + \beta c_3c_2 \langle R_{t+1, n} \psi_{n-1}, \phi_{t+1, a+1, n-1}\rangle \\
        &= 1 - \beta c_4 \langle R_{t+1, n} \psi_n, \phi_{n}\rangle + \beta c_3c_1  \frac{F_{t+1, n-1}}{F_{t+1}} \langle R_{t+1, n-1} \psi_{t+1, a+1, n-1}, \phi_{t+1, a+1, n-1}\rangle \\
        &\quad + \beta c_3c_2 \langle R_{t+1, n-1} \psi_{n-1}, \phi_{t+1, a+1, n-1}\rangle \\
        &\quad + \beta \aaa c_3c_2 \frac{F_{t+1, n-1}}{F_{t+1, n}}\langle R_{t+1, n-1}\psi_{n-1}, \phi_{t+1}\rangle \langle R_{t+1, n-1}\psi_{t+1, n-1, a+1}, \phi_{t+1, a+1, n-1}\rangle \\
     &= 1 - \beta c_4 \frac{F_{n}}{F_{t+1}}\left(\frac{F_{a-1, n}}{F_n} -1\right) +  \beta c_3c_1  \frac{F_{t+1, n-1}}{F_{t+1}} \left(1-\frac{F_{t+1, a+1, n-1}}{F_{t+1, n-1}}\right) \\
        &\quad +  c_3c_2 \left(\frac{F_{n-1}}{F_{t+1, n-1}}-1\right) + \frac{\beta \aaa c_3c_2}{1-\aaa + \beta} \frac{F_{t+1, n-1}}{F_{t+1, n}}\left(\frac{F_{a-1, n}}{F_{t+1 ,n-1}}-1\right) \left(1-\frac{F_{t+1, a+1, n-1}}{F_{t+1, n-1}}\right) 
\end{align*} 
where in the last line  we used Lem.\ ~\ref{Fredholm rank one} and noticed
\begin{align*}
    K_{t+1, a, n-1}- K_{t, a-1, n} &= (1+\beta-\aaa)\psi_{t, a, n-1}\otimes \phi_{t+1, a, n}, 
\end{align*}
by expanding the left-hand side as $K_{t+1, a, n-1} - K_{t+1, a, n} + K_{t+1, a, n} - K_{t+1, a-1, n} + K_{t+1, a-1, n}- K_{t, a-1, n}$. Collecting terms, we arrive at 
\begin{align*}
    F_{t+1, a, n}F_{t, a, n-1} - \frac{\aaa \beta}{1+\beta-\aaa}F_{t, a-1, n}F_{t+1, a+1, n-1}-\left(1-\frac{\aaa \beta}{1+\beta -\aaa}\right)F_{t, a, n}F_{t+1, a, n-1} &= 0.
\end{align*} 
\end{proof}

\begin{theorem}[\textbf{Bernoulli Jumps with Pushing Eq.\ General Solutions}]\label{LBTASEP K THM}
     Let $K_{t, a, n} \in B_r(U, \mathcal{I}_1)$ be a family of trace-class integral operators acting on $L^2(X, \mu)$ such that the following three conditions hold: 
    \begin{enumerate}
        \item ($a$--flows): $\begin{aligned}[t] 
        \nabla_a^- K_{t, a, n} &= \psi_{t, a, n} \otimes \phi_{t, a, n}. \label{LBTASEP grad a decomp}
        \end{aligned}$
        \item ($n$--flows): $\begin{aligned}[t]
             \nabla_n^+ \psi_{t, a, n} = \aaa \nabla_a^+ \psi_{t, a, n}, \medspace \nabla_n^- \phi_{t ,a, n} = \aaa \nabla_a^- \phi_{t, a, n}.
        \end{aligned}$
        \item \label{LBTASEP grad t decomp} ($t$--flows): $\begin{aligned}[t]
            \nabla_t^+ \psi_{t, a, n} &= \beta \nabla_a^+ \psi_{t, a, n}, \medspace \nabla_t^- \phi_{t, a, n} = \beta\nabla_a^- \phi_{t, a, n},
        \end{aligned}$
    \end{enumerate}
    for some arbitrary constants $\aaa, \beta$ with $\beta-\aaa \neq 0$. Suppose further $I-K_{t,a, n}$ is invertible for all $(t, a, n) \in V \subseteq U$ open with $n > M$. Then  
    \begin{equation}
        F_{t,a, n} = \det(I-K_{t,a, n})_{L^2(X, \mu)}, \qquad (t,a,n) \in V
    \end{equation}
    satisfies 
    \begin{align}
        \left[e^{D_t} - q e^{D_a} -\left(1-q\right) \right] F_{t, a, n} \cdot F_{t+1, a+1, n-1} = 0, \quad q =  \frac{\beta(1-\aaa)}{\beta -\aaa}. \label{LBPTASEP FORM}
    \end{align}
\end{theorem}
\begin{remark}
    Note that \eqref{LBPTASEP FORM} can be written out explicitly as 
    \begin{align}
        F_{t+1, a, n}F_{t, a+1, n-1} - qF_{t, a+1, n}F_{t+1, a, n-1} -(1-q)F_{t, a, n}F_{t+1, a+1, n-1} = 0.
    \end{align}
\end{remark}
\begin{proof}[Proof of Thm.\ \ref{LBTASEP K THM}]
First, note that due to \eqref{LBTASEP grad a decomp}--\eqref{LBTASEP grad t decomp} and since $K_{t,a, n} \in B_r(U, \mathcal{I}_1)$, we have
\begin{align}
    K_{t, a, n} &= \sum_{r=-\infty}^a \psi_{t, r, n} \otimes \phi_{t, r, n}, \\
    \nabla_n^- K_{t, a, n} &= \aaa \psi_{t,a+1, n-1} \otimes \phi_{t, a, n}, \\
    \nabla_t^+ K_{t, a, n} &= \beta \psi_{t, a+1, n} \otimes \phi_{t+1, a, n} \label{LBTASEP T}. 
\end{align}
 To see \eqref{LBTASEP T}, we compute
    \begin{align*}
        \nabla_t^+ K_{t, a, n} &= \sum_{r=-\infty}^a \nabla_t^+ \psi_{t, r, n} \otimes \phi_{t+1, r, n} + \psi_{t, r, n} \otimes \nabla_t^+ \phi_{t, r, n} = \beta \sum_{r=-\infty}^a \nabla_r^+ (\psi_{t, r, n} \otimes \phi_{t+1, r-1, n})\\
        &= \beta \psi_{t, a+1, n} \otimes \phi_{t+1, a, n}
    \end{align*}
    Next, from the discrete flows it is not hard to see the identities
    \begin{align}
        (\aaa- \beta)\psi_{t, a+1, n-1} &= \aaa \psi_{t+1, a+1, n-1} - \beta \psi_{t, a+1, n}\label{LB tasep id}, \\
         (\beta-\aaa)\phi_{t+1, a, n} &= (1-\aaa)\phi_{t, a+1, n} - (1-\beta)\phi_{t+1, a+1, n-1}. \label{LB tasep id2}
    \end{align}
    Therefore, using Lem.\ ~\ref{Fredholm rank one}--\ref{Fredholm fg} and \eqref{LB tasep id}--\eqref{LB tasep id2}, we have
    \begin{align*}
        \frac{F_{t+1, a, n}}{F_{t, a, n}} &= 1-\beta\langle R_{t, a, n} \psi_{t, a+1, n}, \phi_{t+1, a, n}\rangle \\
        &= 1 - \frac{\beta(1-\aaa)}{\beta-\aaa}\langle R_n \psi_{a+1}, \phi_{a+1}\rangle + \frac{\beta(1-\beta)}{\beta-\aaa}\langle R_n \psi_{a+1}, \phi_{t+1, a+1, n-1}\rangle \\
        &= 1 - \frac{\beta(1-\aaa)}{\beta-\aaa}\langle R_n \psi_{a+1}, \phi_{a+1}\rangle + \frac{\aaa(1-\beta)}{\beta-\aaa}\langle R_n \psi_{t+1, a+1, n-1}, \phi_{t+1, a+1, n-1}\rangle  \\
        &\quad + (1-\beta) \langle R_n \psi_{ a+1, n-1}, \phi_{t+1, a+1, n-1}\rangle \\
        &=  1 - \frac{\beta(1-\aaa)}{\beta-\aaa}\langle R_n \psi_{a+1}, \phi_{a+1}\rangle + \frac{\aaa(1-\beta)}{\beta-\aaa}\langle R_{t+1, n-1} \psi_{t+1, a+1, n-1}, \phi_{t+1, a+1, n-1}\rangle \\
        &\quad - \aaa(1-\beta)  \langle R_{t+1, n-1}\psi_{t+1, a+1, n-1}, \phi_{t+1}\rangle \langle R_{n}\psi_{a+1, n-1}, \phi_{t+1, a+1, n-1}\rangle \\
        &\quad  + (1-\beta) \langle R_n \psi_{ a+1, n-1}, \phi_{t+1, a+1, n-1}\rangle, 
    \end{align*}
    where we used 
    \begin{align*}
        K_{t, a, n} - K_{t+1,a, n-1} &= -(\beta - \aaa) \psi_{t, a+1, n-1}\otimes \phi_{t+1,a,n}.
    \end{align*}
    Now, notice 
    \begin{align*}
        (1-\beta)\langle R_n \psi_{a+1, n-1}, \phi_{t+1, a+1, n-1}\rangle &= \langle R_n \psi_{a+1, n-1}, \phi_{a+1, n-1}\rangle - \beta \langle R_n \psi_{a+1, n-1}, \phi_{t+1, n-1}\rangle \\
        &= \frac{F_{n-1}}{F_n}\left(1-\frac{F_{a+1, n-1}}{F_{n-1}}\right) + \frac{F_{n-1}}{F_n}\left(\frac{F_{t+1, n-1}}{F_{n-1}}-1\right).
    \end{align*}
    This allows us to write 
    \begin{align*}
        \frac{F_{t+1, n}}{F_n} &=  1 - \frac{\beta(1-\aaa)}{\beta-\aaa}\left(1-\frac{F_{a+1, n}}{F_n}\right) + \frac{\aaa(1-\beta)}{\beta-\aaa}\left(1-\frac{F_{t+1, a+1, n-1}}{F_{t+1, n-1}}\right)\\
        &\quad \frac{F_{t+1, n}}{F_{t+1, n-1}}\left(\frac{F_{n-1}}{F_n}\left(1-\frac{F_{a+1, n-1}}{F_{n-1}}\right) + \frac{F_{n-1}}{F_n}\left(\frac{F_{t+1, n-1}}{F_{n-1}}-1\right)\right)
    \end{align*}
    Multiplying through by $F_{n}F_{t+1, n-1}$ and collecting terms yields
    \begin{align*}
        F_{t+1, a, n}F_{t, a+1, n-1} - \frac{\beta(1-\aaa)}{\beta-\aaa}F_{t, a+1, n}F_{t+1, a, n-1} + \frac{\aaa(1-\beta)}{\beta-\aaa}F_{t, a, n}F_{t+1, a+1, n-1} = 0.
    \end{align*}
\end{proof}
\section{KPZ Models and One-Point Distributions}\label{sec: 3}
\subsection{RBM/BLPP Models}\label{sec: RBM/BLPP Models} 
Let $\mathbf{y} = (y_0, y_1, y_2, \dots)$ with $y_0 \geq y_1 \geq y_2 \geq \dots$, and $y_n \in \bar{\R}$. Given $\mathbf{f}(t) = (f_0(t), f_1(t), \dots) \in C([0, \infty), \R^{\Z_{\geq 0}})$, we define the (negative) reflection process recursively by the formula 
\begin{align}
    f_0^{\Lambda}(t; \mathbf{y}) = y_0 + f_0(t), \quad f_n^{\Lambda}(t; \mathbf{y}) = y_n + f_n(t) - \sup_{0 \leq s \leq t}[y_n + f_n(s) - f_{n-1}^{\Lambda}(s; \mathbf{y})]^+
\end{align}
for $n \geq 1$, with $[\,\cdot\,]^+ = \max( \, \cdot \,, 0)$. For our probabilistic models, we fix a filtered complete probability space $(\Omega, \mathcal{F}, (\mathcal{F}_t)_{t \geq 0}, \PP)$ that supports infinitely many i.i.d.\ standard $(\mathcal{F}_{t})_{t\geq 0}$-Brownian motions $\{B_{n}\}_{n \geq 1}$, and take 
 $f_n(t)(\omega) = B_n(t)(\omega), \omega \in \Omega, n \geq 1$ to be the sample paths. The resulting process will be referred to as Reflected Brownian Motions (RBM), with data $\bigl(\mathbf{y}, f_0(\cdot)\bigr)$.
\begin{remark}
    For $\mathbf{f}(0) = \mathbf{0}$, let 
    \begin{align}
        \mathbf{f}[(0, m) \rightarrow (t, n)] = \sup_{0\leq t_m \leq \cdots \leq t_{n} = t} \left[f_m(t_m) + \sum_{k=m+1}^{n} \bigl(f_{k}(t_k) - f_k(t_{k-1})\bigr)\right]. \label{LPP formula}
    \end{align}
    Through a simple induction argument, one can show
    \begin{align}
         -f_n^{\Lambda}(t; \mathbf{y}) = \max_{0\leq m \leq n} \{   \mathbf{(-f)}[(0, m) \rightarrow (t, n)] - y_m\}. \label{reflection and lpp}
    \end{align}
    The above formulas have a geometric interpretation: a given $t_m \leq \dots \leq t_{n}$ may be thought of as specifying a sequence of jump times for a nondecreasing càdlàg path $\pi: [0, t] \rightarrow \{m, \dots, n\}$, with $\pi(0) \geq m, \pi(t) =n$, whose path length is defined by the term inside the $\sup$ of \eqref{LPP formula}. Therefore, $\mathbf{f}[(0, m) \rightarrow (t, n)]$ is the value of the maximal path length from $(0,m)$ to $(t, n)$ through the field $(f_m, \dotsc, f_n)$. When $f_k(t)(\omega) = B_k(t)(\omega)$, the model is referred to as Brownian Last Passage Percolation (BLPP), which is variationally dual to RBM by \eqref{reflection and lpp} (and since $\mathbf{B} \overset{d}{=}(\mathbf{-B})$). 
\end{remark}
\subsubsection{Reflected Brownian Motions with General Initial Condition}
Here, we take a general one-sided initial data $\mathbf{y}$ with $y_0 = \infty$ and  $f_0(t) \equiv \infty$ (so we may think $\mathbf{y} = (y_1, y_2, \dots)$). Let us denote the $n$-th particle as $Y_n^{\text{RBM}}(t)$, and note that $Y^{\text{RBM}}_1(t)$ is a standard Brownian motion started at $y_1$.
\begin{theorem}\label{RBM Particle Dist Thm}
     Fix initial data $\mathbf{y}=(y_1, y_2, \dots)$ with $y_1 \geq y_2 \geq \dots$, and let $V = \R_+ \times \R \times \Z_{\geq 1}$, 
    \begin{equation}
        F_{t, a, n} = \PP( Y_n^{\text{RBM}}(t) > a \, | \, \mathbf{Y}(0) = \mathbf{y}), \qquad (t, a, n) \in V. 
    \end{equation}
    Then $F_{t,a, n}$ satisfies 
    \begin{align}
        \left[D_t  - \frac{1}{2}D_a^2\right] F_{t, a, n} \cdot F_{t, a, n-1} &= 0, \qquad \bigl(\partial_t-\frac{1}{2}\partial_a^2\bigr) F_{t,a, 1}= 0, \qquad  F_{0, a, n} = 1_{y_n > a}. \label{RBM PART form}
    \end{align}
\end{theorem}
\begin{remark}
    By extending the initial configuration by setting $y_m = + \infty$ if $m < 1$ (so $F_{t, a, m} \equiv 1$), the bilinear equation already implies the $n=1$ heat equation. 
\end{remark}
In order to prove Thm.\ \ref{RBM Particle Dist Thm} we start from the one-point distribution formula for RBM in \cite{MQR16} and, after a minor transformation, demonstrate that the kernel satisfies the conditions of Thm.\ \ref{RBM Kernel Thm}. To this end, let $H_n$ be the $n$-th Hermite polynomial, defined by
\begin{align*}
    H_n(x) = (-1)^n e^{x^2/2}\frac{d^n}{dx^n} e^{-x^2/2}.
\end{align*}
 We also define
\begin{align} \label{varphi def}
    \varphi_n(t, x) &= t^{-n/2}\frac{1}{\sqrt{2\pi t}}e^{-x^2/2t}H_n(x/\sqrt{t}), \quad
        \bar{\varphi}_n(t, x) = \frac{1}{n!}t^{n/2}H_n(x/\sqrt{t}).
\end{align}
These satisfy the raising and lowering identities 
\begin{equation}\label{Hermite calc}
\partial_x\varphi_n(t,x)=-\varphi_{n+1}(t,x),\qquad \partial_x\bar\varphi_n(t,x)=\bar\varphi_{n-1}(t,x),
\end{equation}
\begin{equation}\label{Hermite calc bar}
\partial_t\varphi_n(t,x)=\tfrac12\,\varphi_{n+2}(t,x),\qquad
\partial_t\bar\varphi_n(t,x)=-\tfrac12\,\bar\varphi_{n-2}(t,x).
\end{equation}
Next, suppose we are given a vector of initial data $\mathbf{y}$ as in Thm.\ \ref{RBM Particle Dist Thm}. Let $(B_k)_{k \geq 0}$ be a discrete time random walk taking $\text{Exp}(1)$ steps downwards and define the epigraph hitting time $$ \tau = \inf\{ k \geq 0: B_k \geq y_{k+1}\}.$$  Define
\begin{align}
    \psi_{t, a, n}(u) &= e^{u-a}\varphi_n(t, u-a), \quad 
    \phi_{t, a, n}^{\mathbf{y}}(v) = \E_{B_0 = v} \left[e^{a-B_{\tau}}\bar{\varphi}_{n-\tau -1}(t, B_{\tau} -a)1_{\tau < n} \right].
\end{align}
\begin{lemma}[{\cite[Thm.~2.1]{MQR16}}]
    With $F_{t, a ,n}$ as above,   
    \begin{align}
        F_{t, a, n} &= \det(I-K_{t, a, n}^{\mathbf{y}})_{L^2(\R)},  \quad \text{ with } \quad K_{t, a, n}^{\mathbf{y}} = \int_{-\infty}^a \psi_{t, r, n} \otimes \phi_{t, r, n}^{\mathbf{y}}\, dr. \label{mqr det form}
    \end{align}
\end{lemma}
\begin{proof}
    Let
    \begin{align*}
        \mathcal{S}_{-t, -n}(u, v) &= e^{u-v}\varphi_n(t, u-v), \quad 
        \bar{\mathcal{S}}_{-t, n}(u, v) = e^{v-u}\bar{\varphi}_{n-1}(t, u-v),
    \end{align*}
    and
    \begin{align*}
        \bar{\mathcal{S}}_{t, n}^{\text{epi}(\mathbf{y})}(u, v) = \E_{B_0 = u}\left[\bar{\mathcal{S}}_{-t, n-\tau}(B_{\tau}, v)1_{\tau < n}\right].
    \end{align*}
   A special case of \cite[Thm.~2.1]{MQR16} gives the following trace-class integral operator on $L^2(\R)$ 
    \begin{align*}
        K_{t, a, n}(u, v) &= 1_{\{u\leq a\}} \left(\mathcal{S}^*_{-t, -n}\bar{\mathcal{S}}^{\text{epi}(\mathbf{y})}_{t, n} \right)(u, v) \, 1_{\{v\leq a\}}.
    \end{align*}
    Let $A(u, z) = 1_{\{u \leq a\}} \mathcal{S}_{-t, -n}(z, u)$, $B(z, v) = \bar{\mathcal{S}}_{t, n}^{\text{epi}(\mathbf{y})}(z, v)1_{\{v \leq a\}}$,  then Lem.\ \ref{AB BA} gives $\det(I-AB) = \det(I-BA)$ and hence \eqref{mqr det form}. 
    \end{proof}
\begin{proof}[Proof of Thm.\ \ref{RBM Particle Dist Thm}]
   We show the kernel given in \eqref{mqr det form} satisfies the conditions of Thm.\ \ref{RBM Kernel Thm}. The $a$--flow is immediate from the integral representation. For $\psi$, with $x=u-a$, we have $$\partial_a \psi_{t,a,n}(u) = -e^{u-a}(\varphi_n(t, x) + \partial_x \varphi_n(t, x)) = \nabla_n^+ \psi_{t, a, n}(u),$$ using $\partial_x \varphi_n = -\varphi_{n+1}$.
   For $\phi$, we have
    \begin{align*}
        \partial_a \phi_{t, a, n}^{\mathbf{y}}(v) 
        &=  \E_{B_0 = v} \left[e^{a-B_{\tau}}\bar{\varphi}_{n-\tau -1}(t, B_{\tau} -a)1_{\tau < n} \right] - \E_{B_0 = v} \left[e^{a-B_{\tau}}\bar{\varphi}_{n-\tau -2}(t, B_{\tau} -a)1_{\tau < n-1} \right] \\
        &= \nabla_n^- \phi_{t, a, n}^{\mathbf{y}} (v), 
    \end{align*}
    where we used $\partial_x \bar{\varphi}_n = \bar{\varphi}_{n-1}$ and that when $\tau = n-1$, the derivative of $\bar{\varphi}_{n-\tau - 1}$ vanishes. The derivatives with respect to time follow similarly using $\partial_t\varphi_n = \frac{1}{2}\varphi_{n+2}$ and $\partial_t 
    \bar{\varphi}_{n} = -\frac{1}{2}\bar{\varphi}_{n-2}$, together with the $n$--flow identities. Finally, by the standard Gaussian-polynomial estimates for $\varphi_n, \bar{\varphi}_n$, the fact that $\phi_{t, a, n}^{\mathbf{y}}(v)$ is supported on $\{ v \geq y_n\}$, and noting $B_{\tau} = v$ when $v \geq y_1$, we have that $K_{t, a, n}^{\mathbf{y}} \in C_r^{1, 2}$, as required.
\end{proof}

\subsubsection{Reflected Brownian Motions with Moving Wall}
In this section, we take $\mathbf{y} \equiv 0$, $f_0(t) = b(t) $ for some continuous function with $b(0) = 0$. Let us denote the $n$-th particle as $Y_n^{\text{RBM}, b}(t)$. We note that due to the variational formula \eqref{reflection and lpp}, $-Y_n^{\text{RBM}, b}(t)$ is the value of BLPP from $(0, 0)$ to $(t, n)$ with the lower boundary $f_0(\cdot) = -b(\cdot)$. 
\begin{theorem}\label{RBM Particle b Dist Thm}
    Fix $\mathbf{y}=0$ and $b(t) \in C(\R_+, \R)$ with $b(0) = 0$. Let $V = \{(t, a, n) : t \in \R_+, a < b(t), n \in \Z_{\geq 1}\}$ with
    \begin{align}
        F_{t, a, n} &= \PP( Y_n^{\text{RBM}, b}(t) > a \, | \, \mathbf{Y}(0) = 0, \,  Y_0(\cdot) = b(\cdot)), \qquad (t,a, n) \in V. 
    \end{align}
    Then $F_{t, a, n}$ satisfies
    \begin{gather}
        \left[D_t  - \frac{1}{2}D_a^2\right] F_{t, a, n} \cdot F_{t, a, n-1} = 0, \qquad  F_{0, a, n} = 1_{0 > a},\label{RBM PART form b} \\
        \bigl(\partial_t - \frac{1}{2}\partial_a^2\bigr)F_{t, a, 1} = 0 \quad  \text{for } a < b(t), \qquad  
             F_{t, a, 1} = 0 \quad \text{ for } a \geq b(t). 
        \label{RBM with b ic}
    \end{gather}
\end{theorem}
In order to prove Thm.\ \ref{RBM Particle b Dist Thm} we start from the one-point distribution formula for RBM with moving wall from \cite{rahman2025}, and, after a sequence of minor modifications, demonstrate that the resulting kernel satisfies the conditions of Thm.\ \ref{RBM Kernel Thm}. To this end, recall $\varphi_n, \bar{\varphi}_n$ defined in \eqref{varphi def}, and for a standard Brownian motion $B(t)$, define the hitting time $\tau = \inf\{ s \geq 0 \, | \, B(s) \geq b(s)\}$. Fix $T> 0$ and assume $t \in (0, T)$. Define
\begin{align*}
    \psi_{t, a, n}^b(u) &= \E_{B_{0}=u}\left[e^{-a + \frac{u^2}{4T}}\varphi_n(t-\tau, B(\tau) - a)1_{\tau \leq t}\right], \quad 
    \phi_{t, a, n}(v) = e^{a-\frac{v^2}{4T}}\bar{\varphi}_{n-1}(t, v-a).
\end{align*}

\begin{lemma}[{\cite[Cor.~1.1]{rahman2025}}] 
        With $F_{t, a, n}$ as above and $t \in (0, T)$ for some $T > 0$, 
    \begin{align}
        F_{t,a, n} = \det(I-K_{t,a, n}^b)_{L^2(\R)}, \quad \text{with} \quad K_{t, a, n}^{b} = \int_{-\infty}^a \psi_{t, r, n}^b \otimes \phi_{t, r, n} dr. \label{r det form}
    \end{align}
\end{lemma}
\begin{proof}
Let
\begin{align*}
    \mathcal{S}_{t, n}(u, v) &= \bar{\varphi}_{n-1}(t, u-v), \quad 
    \mathcal{S}^{\text{epi}(b)}_{t, n}(u, v) = (-1)^n\E_{B_0 = u}[\varphi_n(t-\tau, -B(\tau) - v)1_{\tau \leq t}].
\end{align*}
A special case of \cite[Cor.~1.1]{rahman2025} gives the following trace-class integral operator on $L^2(\R)$
\begin{align*}
    K_{t, a, n} &= 1_{\{u \geq -a\}}\left(\mathcal{S}_{t, n}\mathcal{S}^{\text{epi}(b)}_{t, n}\right)(u, v)1_{\{v \geq -a\}}.
\end{align*} 
Now we apply some simple manipulations to transform our kernel into the desired form. Define
\begin{align*}
    A(u, z) &= e^{-u} \mathcal{S}_{t, n}(u, z)e^{-\frac{z^2}{4T}}, \qquad 
    B(z, v) = e^{\frac{z^2}{4T}} \mathcal{S}_{t, n}^{\text{epi}(b)}(z, v)e^{v}.
\end{align*}
Now, using the estimates in \cite[Prop.~5.1]{rahman2025} there exists a constant $C_{n,T}$ such that
\begin{align}
    \abs{A_{t, n}(u, z)} &\leq C_{n, T}e^{-u}(\abs{u-z}^{n-1} + 1)e^{-\frac{z^2}{4T}}, \label{ra est 1}\\
    \abs{B_{t, n}(z, v)} &\leq C_{n, T} t^{-1/2} e^{\frac{z^2}{4T}}e^{v-\frac{(v-z)^2}{4t}} (\abs{z}^n+\abs{v}^n + 1).  \label{ra est 2}
\end{align}
Using Lem.\ \ref{AB BA} with $\tilde{A}_{t, n}(u, z) =1_{\{u \geq - a\}}A_{t, n}(u, z), \tilde{B}_{t, n}(z, v) = B_{t, n}(z,v)1_{\{v \geq - a\}}$, and a simple change of variables, we therefore get 
\begin{align*}
    K_{t, a, n} &\rightarrow \int_{-\infty}^a B_{t, n}(u, -r)  A_{t, n}(-r, v)dr.
\end{align*}
 Now notice that 
 \begin{align*}
     B_{t, n}(u, -r) &= e^{-r+ \frac{u^2}{4T}}(-1)^n\E_{B_0 = u}[\varphi_n(t-\tau, -B(\tau) + r)1_{\tau \leq t}] \\
     &= e^{-r+ \frac{u^2}{4T}}\E_{B_0 = u}[\varphi_n(t-\tau, B(\tau) - r)1_{\tau \leq t}].
 \end{align*}
 Finally, we can replace $A_{t, n}(-r, v) \rightarrow A_{t, n}(-r, -v)$ (by passing the change of variables onto the $L^2$ function being integrated on), yielding the desired form. 
\end{proof}
\begin{proof}[Proof of Thm.\ \ref{RBM Particle b Dist Thm}]
For a fixed $T> 0$ with $t \in (0, T)$, the proof is completely analogous to the proof of Thm.\ \ref{RBM Particle Dist Thm}, noting the estimates \eqref{ra est 1}--\eqref{ra est 2} and that only $\varphi$ contributes to the time derivative on $\psi^b$ since as $t-\tau \downarrow 0$, we have $\abs{B(\tau) - a}> 0$ by the continuity of $b(\cdot)$, so the small time Gaussian contribution of $\varphi_n(t-\tau, B(\tau)-a)$ kills the contribution of the indicator. 
\end{proof}
\subsection{Continuous-Time Particle Exclusion Models}\label{sec: KPZ TASEP models}
\subsubsection{TASEP}
The totally asymmetric simple exclusion process (TASEP) is an interacting particle system on the one-dimensional integer lattice $\Z$ with at most one particle per site. Given a strictly decreasing initial configuration $\mathbf{y} = (y_1, y_2, y_3, \dotsc)$, the dynamics run in continuous time as follows: each particle carries an independent rate one exponential clock, and when particle $n$'s clock rings, it attempts to jump to the right by one unit. The jump is performed only if the destination site is empty, otherwise it is suppressed. After each particle's (attempted) jump its independent clock is instantaneously reset.  

\begin{theorem}\label{TASEP part thm}
    Fix one-sided initial data $\mathbf{y} = (y_1, y_2, \dots)$ with $y_1 > y_2 > \dots$, and let $V = \R_+ \times \Z \times \Z_{\geq 1}$ with  
   \begin{align}
       F_{t,a, n} = \PP(Y_n^{\text{TASEP}}(t) > a \, | \, \mathbf{Y}(0) = \mathbf{y}), \qquad (t, a ,n) \in V.
   \end{align}
   Then $F_{t, a, n}$ satisfies
   \begin{align}
       \left[D_t -\left(e^{-D_a} - 1\right)\right] F_{t,a,n}\cdot F_{t, a, n-1} &= 0, \qquad (\partial_t + \nabla_a^-) F_{t, a, 1} = 0, \qquad  F_{0, a, n} = 1_{y_n > a}. 
   \end{align}
\end{theorem}
\begin{remark}
    By extending the initial configuration by setting $y_m = +\infty$ if $m<1$ (so $F_{t,a, m} \equiv 1$), the bilinear equation already implies the $n=1$ forward equation.
\end{remark}
In order to prove Thm.\ \ref{TASEP part thm} we start from the one-point distribution formula in \cite{MQR17} and, after a sequence of minor transformations, demonstrate that the kernel satisfies the conditions of Thm.\ \ref{TASEP K THM}. To this end, define
\begin{align}
    \varphi_{t, a, n}(u) &= \frac{1}{2\pi i} \oint_{\Gamma_0} \frac{(1-w)^n}{2^{a-u}w^{n+1+a-u}}e^{t(w-\frac{1}{2})}dw, \quad 
    \bar{\varphi}_{t, a, n}(v) = \frac{1}{2\pi i}\oint_{\Gamma_0} \frac{(1-w)^{a-v+n-1}}{2^{v-a}w^n}e^{t(w-\frac{1}{2})}dw
\end{align}
where $\Gamma_0$ is a simple counter-clockwise loop around $0$ but excluding $1$. It is elementary to see these satisfy 
\begin{alignat}{2}
    \nabla_n^+ \varphi_{t, a, n}(u) &= 2\nabla_a^+ \varphi_{t,a, n}(u), \quad \nabla_n^- \bar{\varphi}_{t,a, n}(v) &&= 2\nabla_a^- \bar{\varphi}_{t, a, n}, \label{Tasep psi n flow}\\
    \partial_t \varphi_{t, a, n}(u) &= -\frac{1}{2} \nabla_a^- \varphi_{t,a, n}(u), \quad \partial_t \bar{\varphi}_{t,a, n}(v) &&= -\frac{1}{2}\nabla_a^+ \bar{\varphi}_{t, a, n}. \label{TASEP psi t flows}
\end{alignat}
Next, for given initial data $\mathbf{y}$, let $B_m$ be a discrete-time Geom(1/2) random walk with jumps strictly in the negative direction and let $\tau = \min \{m \geq 0 : B_m > y_{m+1}\}$. Define
\begin{align}
    \psi_{t,a, n}(u) &= \varphi_{t, a, n}(u), \quad \phi_{t, a, n}^{\mathbf{y}}(v) = \E_{B_0 = v}\left[\bar{\varphi}_{t, a, n-\tau}(B_{\tau})1_{\tau < n}\right].
\end{align}
\begin{lemma}[{\cite[Thm.\ 2.6]{MQR17}}]
    With $F_{t,a ,n}$ as above,
    \begin{align}
        F_{t, a, n} &= \det(I-K_{t,a, n}^{\mathbf{y}})_{\ell^2(\Z)}, \quad \text{ with } \quad K_{t, a, n}^{\mathbf{y}} = \sum_{r=-\infty}^a \psi_{t, r, n} \otimes \phi_{t,r, n}^{\mathbf{y}} \label{TASEP FP kernel}.
    \end{align}
\end{lemma}
\begin{proof}
     Let
\begin{align*}
     \mathcal{S}_{-t, -n}(u, v) &= \frac{1}{2\pi i} \oint_{\Gamma_0} \frac{(1-w)^n}{2^{v-u}w^{n+1+v-u}}e^{t(w-1/2)}dw \\
     \bar{\mathcal{S}}_{-t, n}(u, v) &=  \frac{1}{2\pi i} \oint_{\Gamma_0} \frac{(1-w)^{v-u+n-1}}{2^{u-v}w^{n}}e^{t(w-1/2)}dw 
 \end{align*}
 and
 \begin{align*}
     \bar{\mathcal{S}}_{-t, n}^{\text{epi}(\mathbf{y})}(u, v) = \E_{B_0 = u} [\bar{\mathcal{S}}_{-t, n-\tau}(B_{\tau}, v)1_{\tau < n}].  
 \end{align*}
 A special case of \cite[Thm.\ 2.6]{MQR17} gives the following trace-class integral operator on $\ell^2(\Z)$
 \begin{align*}
     K_{t,a, n}(u, v) &= 1_{\{u \leq a\}}\Bigl(\sum_{r\in\Z} \mathcal{S}_{-t,-n}(r, u)\bar{\mathcal{S}}^{\text{epi}(\mathbf{y})}_{-t, n}(r, v) \Bigr)1_{\{v \leq a\}}.
 \end{align*}
 Using Lem.\ \ref{AB BA}--\ref{lem:det-transpose}, we can therefore write 
    \begin{align*}
     K_{t, a, n}(u, v) &\rightarrow  \sum_{r=-\infty}^a \bar{\mathcal{S}}_{-t, n}^{\text{epi}({\mathbf{y}})}(u, r)\mathcal{S}_{-t, -n}(v, r) \rightarrow  \sum_{r=-\infty}^a \mathcal{S}_{-t, -n}(u, r)\bar{\mathcal{S}}_{-t, n}^{\text{epi}({\mathbf{y}})}(v, r),
 \end{align*}
 without changing the value of the Fredholm determinant, yielding the stated representation. 
 \end{proof}
\begin{proof}[Proof of Thm.\ \ref{TASEP part thm}]
We show that the kernel given in \eqref{TASEP FP kernel} satisfies the conditions of Thm.\ \ref{TASEP K THM}. Note the $a$-flow is immediate from the kernel representation. The $n$-flows and $t$-flows for $\psi_{t, a, n}, \phi_{t, a, n}^{\mathbf{y}}$ follow from \eqref{Tasep psi n flow}--\eqref{TASEP psi t flows} and noticing that $\bar{\varphi}_{t, a, n-\tau-1} = 0$, when $\tau = n-1$. Moreover, since $\psi_{t, a, n}(u)$ has support on $a \geq u-n$, and $\phi_{t, a, n}^{\mathbf{y}}(v)$ has support on $v > y_n$, it is clear $K_{t,a, n}^{\mathbf{y}} \in C^1_r$, as required. 
\end{proof}
\subsubsection{TASEP with Moving Wall}
Consider again the TASEP particle model, and suppose we introduce a particle $Y_0(t)$ whose trajectory is deterministic in the following way: choose times $0 = s_0 < s_1 < s_2 < \dots$ with $s_k \rightarrow \infty$. The new particle jumps one unit rightward at the times $s_k$, so that $Y_0(t) = Y_0(0) + \max\{k \geq 0: s_k \leq t\}$, with $Y_0(0) > Y_1(0)$. The other TASEP particles evolve as before subject to the same exclusion rule. Denote the $n$-th particle as $Y_n^{\text{TASEP}, b}$ with the moving wall $Y_0(t) = b(t)$. 
\begin{theorem}\label{TASEP b part thm}
    Fix a moving wall $b(t)$ with $b(0) = 0$, initial data $\mathbf{y} = (-1, -2, \dots)$, and let $V=\{(t, a, n) : t \in \R_+, a+n < b(t), n \in \Z_{\geq 1} \}$ with 
   \begin{align}
       F_{t,a, n} = \PP(Y_n^{\text{TASEP}, b}(t) > a \, | \, \mathbf{Y}(0) = \mathbf{y}, Y_0(t) = b(t)), \qquad (t, a, n) \in V.
   \end{align}
   Then $F_{t, a, n}$ satisfies 
   \begin{gather}
       \left[D_t -\left(e^{-D_a} - 1\right)\right] F_{t,a,n}\cdot F_{t, a, n-1} = 0, \qquad  F_{0, a, n} = 1_{-n > a}, \\
       (\partial_t + \nabla_a^-) F_{t, a, 1} = 0 \quad \text{for }a < b(t), \qquad  F_{t, a,1} = 0 \quad \text{for } a \geq b(t).\label{tasep b ic}
    \end{gather}
\end{theorem}
In order to prove Thm.\ \ref{TASEP b part thm} we start from the one-point distribution formula for TASEP with moving wall from \cite{rahman2025}, and, after a sequence of minor modifications, demonstrate that the resulting kernel satisfies the conditions of Thm.\ \ref{TASEP K THM}. To this end, for $u, v \in \R$, define
\begin{align}
    \varphi_{t, a, n}(u) &= \frac{e^{t/2}}{2\pi i} \oint_{\abs{w} < 1} \frac{(1-w)^n}{2^a w^{a+n+1}}e^{(w-1)(t-u)}dw, \\
    \bar{\varphi}_{t, a, n}(v) &= -\frac{e^{-t/2}}{2\pi i} \int_{\Re(w) = \sigma} \frac{w^{a+n}}{2^{-a}(1-w)^n}e^{(w-1)(v-t)}dw
\end{align}
for some fixed $\sigma \in (0, 1)$. It is elementary to see these satisfy 
\begin{alignat}{2}
    \nabla_n^+ \varphi_{t, a, n}(u) &= 2\nabla_a^+ \varphi_{t,a, n}(u), \quad \nabla_n^- \bar{\varphi}_{t,a, n}(v) &&= 2\nabla_a^- \bar{\varphi}_{t, a, n}(v), \label{Tasep b psi n flow}\\
    \partial_t \varphi_{t, a, n}(u) &= -\frac{1}{2} \nabla_a^- \varphi_{t,a, n}(u), \quad \partial_t \bar{\varphi}_{t,a, n}(v) &&= -\frac{1}{2}\nabla_a^+ \bar{\varphi}_{t, a, n}(v). \label{TASEP b psi t flows}
\end{alignat}
Next, let $B_m$ be a random walk with $\text{Exp}(1)$ step distribution, so that the $m$-step distribution is given by $Q^m(u, v)= \frac{(v-u)^{m-1}}{(m-1)!}e^{u-v}1_{\{v \geq u\}}$. Let $\tau = \inf \{ m \geq 0 : B_m \leq s_{m+1} \}$, and define
\begin{align}
    \psi_{t, a, n}^b(u) = \E_{B_0 = u}[\varphi_{t, a-\tau, n}(B_{\tau})1_{\tau \leq a+n}], \quad \phi_{t, a, n}(v) = \bar{\varphi}_{t,a, n}(v). 
\end{align}
\begin{lemma}[{\cite[Thm.~2]{rahman2025}}]
    With $F_{t,a ,n}$ as above,
    \begin{align}
        F_{t, a, n} &= \det(I-K_{t,a, n}^{b})_{L^2(\R)}, \quad \text{with} \quad K_{t, a, n}^{b} = \sum_{r=-\infty}^a \psi_{t, r, n}^{b} \otimes \phi_{t,r, n} \label{TASEP b FP kernel}, 
    \end{align}
\end{lemma}
\begin{proof}
Let
\begin{align*}
    \mathcal{S}_{a, n}(u, v) &= -\frac{1}{2\pi i} \oint_{|w-1|<1} \frac{w^{a+n+1}}{(1-w)^n}e^{(w-1)(v-u)}dw 1_{\{v\leq u\}},\\
    \bar{\mathcal{S}}_{a, n}(u, v) &= \frac{1}{2\pi i} \oint_{|w| <1} \frac{(1-w)^n}{w^{a+n+1}}e^{(w-1)(v-u)}dw,
\end{align*}
and 
\begin{align*}
    \mathcal{S}^{b}_{a, n}(u, v) = \E_{B_0 = u}\left[\bar{\mathcal{S}}_{a-\tau, n}(B_\tau, v)1_{\tau \leq a+n}\right].
\end{align*}
A special case of \cite[Thm.~2]{rahman2025} gives the following trace-class integral operator on $L^2(\R)$
\begin{align*}
    K_{t, a, n}(u, v) = 1_{\{u \geq t\}}\left(\mathcal{S}_{a, n}\mathcal{S}^b_{a, n}\right)(u, v)1_{\{v \geq t\}}.
\end{align*}
It will be convenient for us to remove the $1_{\{v\leq u\}}$ indicator on $\mathcal{S}_{a, n}(u, v)$ while keeping the algebraic structure. To this end, notice that when $a+n+1 < 0$, we have $\bar{\mathcal{S}}_{a, n} \equiv 0$. Therefore, we may assume $a+n+1 \geq 0$, and we have a.e.\ 
\begin{align*}
    \tilde{\mathcal{S}}_{a, n}(u, v) \defeq -\frac{1}{2\pi i} \int_{\Re(w) = \sigma} \frac{w^{a+n+1}}{(1-w)^n}e^{(w-1)(v-u)}dw = \mathcal{S}_{a, n}(u, v). 
\end{align*}
Using Lem.\ \ref{AB BA}, we can take $ K_{t, a, n}(u, v) \rightarrow  \int_t^{\infty} \mathcal{S}_{a, n}^b(u, z) \tilde{\mathcal{S}}_{a, n}(z, v)dz$. Next, a straightforward calculation yields
\begin{alignat*}{2}
    \nabla_a^- \tilde{\mathcal{S}}_{a, n}(u, v)  &= - \tilde{\mathcal{S}}_{a, n-1}(u, v), \qquad 
    \nabla_a^+ \bar{\mathcal{S}}_{a, n}(u, v) &&= \bar{\mathcal{S}}_{a, n+1}(u, v), \\
    \partial_u \tilde{\mathcal{S}}_{a, n}(u, v) &= \tilde{\mathcal{S}}_{a+1, n-1}(u, v), \qquad 
    \partial_v \bar{\mathcal{S}}_{a, n}(u, v) &&= - \bar{\mathcal{S}}_{a-1, n+1}(u, v).
\end{alignat*}
Note that the corresponding identities hold for $\mathcal{S}^b_{a, n}$, noting that when $\tau = a+ 1+ n$, then $\bar{\mathcal{S}}_{a-\tau, n} = \bar{\mathcal{S}}_{-(n+1), n} = 0$. 
Using this, we compute
\begin{align*}
    &\int_t^{\infty} \mathcal{S}_{a, n}^b(u, z) \tilde{\mathcal{S}}_{a, n}(z, v)dz  = \int_{t}^{\infty}\mathcal{S}_{a, n}^b(u, z)\bigl(\tilde{\mathcal{S}}_{a-1, n}(z, v) - \partial_z \tilde{\mathcal{S}}_{a-1, n}(z, v)\bigr)dz \\
    &= \mathcal{S}^b_{a, n}(u, t) \tilde{\mathcal{S}}_{a-1, n}(t, v) + \int_{t}^{\infty} \bigl(\mathcal{S}_{a,n}^b(u,z) + \partial_z \mathcal{S}_{a, n}^b(u, z)\bigr)\tilde{\mathcal{S}}_{a-1, n}(z, v) dz \\
    &= \mathcal{S}^b_{a, n}(u, t) \tilde{\mathcal{S}}_{a-1, n}(t, v) + \int_{t}^{\infty} \mathcal{S}_{a-1,n}^b(u, z)\tilde{\mathcal{S}}_{a-1, n}(z, v) dz  = \sum_{r=-\infty}^a \mathcal{S}^b_{r, n}(u, t) \tilde{\mathcal{S}}_{r-1, n}(t, v),
\end{align*}
which, once conjugating by the appropriate factors, yields the stated representation.
\end{proof}
\begin{proof}[Proof of Thm.\ \ref{TASEP b part thm}]
We check the conditions of Thm.\ \ref{TASEP K THM}. Note the flow identities follow from \eqref{Tasep b psi n flow}--\eqref{TASEP b psi t flows}, and noting that when $\tau = a+ 1+ n$, then $\psi_{t, a-\tau, n}^b = \psi_{t, -(n+1), n}^b = 0$. Regularity follows from noticing $\psi_{t,a,n}^b(u)$ is supported on $a +n + 1 \geq 0, u \leq s_{a+n+1}$, and $\phi_{t, a, n}(v)$ is supported on $v \leq t$.
\end{proof}
\subsubsection{Push-TASEP} 
Like TASEP, Push-TASEP is an interacting particle system on the one-dimensional integer lattice $\Z$ with at most one particle per site. Given a strictly decreasing initial configuration $\mathbf{y} = (y_1, y_2, y_3, \dotsc)$, the dynamics run in continuous time as follows: each particle carries an independent rate one exponential clock, and when particle $n$'s clock rings, it jumps to the left by one unit; with any left neighbours also being pushed leftward so as to preserve exclusion. After each particle's jump its independent clock is instantaneously reset.  
\begin{theorem}\label{Push TASEP part thm}
    Fix one-sided initial data $\mathbf{y} = (y_1, y_2, \dots)$ with $y_1 > y_2 > \dots$ and let $V=\{(t, a, n) : t \in \R_+, a < y_n, n \in \Z_{\geq 1}\}$ with
   \begin{align}
       F_{t,a, n} = \PP(Y_n^{\text{Push-TASEP}}(t) > a \, | \, \mathbf{Y}(0) = \mathbf{y}).
   \end{align}
   Then $F_{t,a , n}$ satisfies 
   \begin{align}
       \left[D_t -\left(e^{D_a} - 1\right)\right] F_{t,a,n}\cdot F_{t, a+1, n-1} &= 0, \qquad (\partial_t -\nabla_a^+)F_{t, a, 1}  = 0, \qquad F_{0, a, n} = 1_{y_n > a}.
   \end{align}
\end{theorem}
\begin{remark}
    By extending the initial configuration by setting $y_m = +\infty$ if $m<1$ (so $F_{t,a, m} \equiv 1$), the bilinear equation already implies the $n=1$ forward equation.
\end{remark}
 Define 
    \begin{align*}
        \varphi_{t, a, n}(u) &= \frac{1}{2\pi i} \oint_{\Gamma_0} \frac{(1-w)^n}{2^{a-u}w^{n+1 +a-u}}e^{t(\frac{1}{w} -2)}dw, \quad
        \bar{\varphi}_{t, a, n}(v) = \frac{1}{2 \pi i} \oint_{\Gamma_{0}} \frac{(1-w)^{a-v +n-1}}{2^{v-a}w^n}e^{ t(2-\frac{1}{1-w})}dw
    \end{align*}
    where $\Gamma_0$ is a simple counter-clockwise loop around $0$ but excluding $1$. It is elementary to see these satisfy 
    \begin{alignat}{2}
        \nabla_n^+ \varphi_{t,a, n} &= 2 \nabla_a^+ \varphi_{t,a, n}, \qquad \nabla_n^- \bar{\varphi}_{t, a, n} &&= 2 \nabla_a^- \bar{\varphi}_
        {t,a, n}, \label{pushtasep n id}\\
        \partial_t \varphi_{t, a, n} &= 2\nabla_a^+ \varphi_{t,a, n}, \qquad \partial_t \bar{\varphi}_{t, a, n} &&= 2 \nabla_a^- \bar{\varphi}_
        {t,a, n}.\label{pushtasep t id}
    \end{alignat}
    Next, for given initial data $\mathbf{y}$, let $B_m$ be a discrete-time Geom(1/2) random walk with jumps strictly in the negative direction, and let $\tau = \min \{m \geq 0 : B_m > y_{m+1}\}$. Define
\begin{align}
    \psi_{t,a, n}(u) &= \varphi_{t, a, n}(u), \quad \phi_{t, a, n}^{\mathbf{y}}(v) = \E_{B_0 = v}\left[\bar{\varphi}_{t, a, n-\tau}(B_{\tau})1_{\tau < n}\right].
\end{align}
\begin{lemma}[{\cite[Thm.~4.1]{Nica_2020}}]
With $F_{t,a ,n}$ as above,
    \begin{align}
        F_{t, a, n} &= \det(I-K_{t,a, n}^{\mathbf{y}})_{\ell^2(\Z)}, \quad \text{ with } \quad K_{t, a, n}^{\mathbf{y}} = \sum_{r=-\infty}^a \psi_{t, r, n} \otimes \phi_{t,r, n}^{\mathbf{y}} \label{PUSHTASEP FP kernel}. 
    \end{align}
\end{lemma}
\begin{proof}
Define     
 \begin{align*}
        \mathcal{S}_{-t, -n}(u, v) &= \frac{1}{2\pi i} \oint_{\Gamma_0} dw \frac{(1-w)^n}{2^{v-u}w^{v-u+n+1}}e^{t(\frac{1}{w} -2)}, \\
        \bar{\mathcal{S}}_{-t, n}(u, v) &= \frac{1}{2 \pi i} \oint_{\Gamma_{0}}dw \frac{(1-w)^{v-u+n-1}}{2^{u-v}w^n}e^{ t(2-\frac{1}{1-w})}, 
    \end{align*}
    and 
    \begin{align*}
        \bar{\mathcal{S}}_{-t, n}^{\text{epi}(\mathbf{y})}(u, v) = \E_{B_0 = u} [\bar{\mathcal{S}}_{-t, n-\tau}(B_{\tau}, v)1_{\tau < n}]. 
    \end{align*}
    A special case of \cite[Thm.~4.1]{Nica_2020} gives the following trace-class integral operator on $\ell^2(\Z)$
 \begin{align*}
     K_{t,a, n}(u, v) &= 1_{\{u \leq a\}}\Bigl(\sum_{r\in\Z} \mathcal{S}_{-t,-n}(r, u)\bar{\mathcal{S}}^{\text{epi}(\mathbf{y})}_{-t, n}(r, v) \Bigr)1_{\{v \leq a\}}.
 \end{align*}
 Using Lem.\ \ref{AB BA}--\ref{lem:det-transpose}, we can therefore write 
    \begin{align*}
     K_{t, a, n}(u, v) &\rightarrow  \sum_{r=-\infty}^a \bar{\mathcal{S}}_{-t, n}^{\text{epi}({\mathbf{y}})}(u, r)\mathcal{S}_{-t, -n}(v, r) \rightarrow  \sum_{r=-\infty}^a \mathcal{S}_{-t, -n}(u, r)\bar{\mathcal{S}}_{-t, n}^{\text{epi}({\mathbf{y}})}(v, r),
 \end{align*}
 without changing the value of the Fredholm determinant, yielding the stated representation.
 \end{proof}
 \begin{proof}[Proof of Thm.\ \ref{Push TASEP part thm}]
     Using \eqref{PUSHTASEP FP kernel} and Thm.\ \ref{PUSHTASEP K thm}, the proof is completely analogous to the proof of Thm.\ \ref{TASEP part thm}. 
 \end{proof}

\subsection{Discrete-Time Particle Exclusion Models}\label{sec: KPZ discrete tasep}
\subsubsection{Parallel TASEP}
Parallel TASEP is a discrete-time variant of TASEP. At each  update $t \mapsto t+1$, every particle independently attempts to jump one unit to the right with probability $p$. Each particle's jump is performed only if the destination site was empty at time $t$, otherwise it is suppressed.
\begin{theorem}\label{Parallel TASEP part thm}
    Fix one-sided initial data $\mathbf{y} = (y_1, y_2, \dots)$ with $y_1 > y_2 > \dots$ and let $V = \{(t, a, n) : t \in \Z_{>0}, a < y_n + t, n \in \Z_{\geq 1} \}$ with
   \begin{align}
       F_{t,a, n} = \PP(Y_n^{\text{Parallel-TASEP}}(t) > a \, | \, \mathbf{Y}(0) = \mathbf{y}), \qquad (t, a, n) \in V. 
   \end{align}
   Then $F_{t, a, n}$ satisfies
   \begin{align}
       \left[e^{D_t} - pe^{-D_a} - (1-p)\right] F_{t,a,n}\cdot F_{t, a, n-1} &= 0, \quad (\nabla_t^+ + p\nabla_a^- )F_{t, a, 1} = 0, \quad F_{0, a, n} = 1_{y_n > a}.
   \end{align}
\end{theorem}
\begin{remark}
    By extending the initial condition by setting $y_m \equiv \infty$ if $m < 1$ (so $F_{t, a, m} \equiv 1$), the bilinear equation already implies the $n=1$ forward equation.
\end{remark}
In order to prove Thm.\ \ref{Parallel TASEP part thm} we start from the one-point distribution formula in \cite{Matetski_2022}, and, after some minor transformations, demonstrate that the kernel satisfies the conditions of Thm.\ \ref{PTASEP K THM}. To this end, let $q=1-p$, and define 
\begin{align*}
    \varphi_{t, a, n}(u) &= \frac{q^{n-1}}{2\pi i} \oint_{\gamma_r}  \frac{1}{2^{a-u}w^{a-u+n+1}}(1-w)^n(q+pw)^{t-(n-1)}(q+p/2)^{-t}dw  \\
    \bar{\varphi}_{t, a, n}(v) &= \frac{q^{-(n-1)}}{2\pi i} \oint_{\gamma_{\delta}} \frac{(1-w)^{a-v+n-1}}{2^{v-a}w^n} (1-pw)^{-t+n-1}(q+p/2)^t dw
\end{align*}
where $\gamma_r$, $\gamma_{\delta}$ are simple counter-clockwise loops around the origin with radius $r\in (0, 1)$ and $\delta \in (0, p^{-1})$, respectively. It is elementary to see these satisfy 
\begin{align}
             \nabla_n^+ \varphi_{t, a, n}(u) &= 2 \varphi_{t, a+1, n} - \varphi_{t, a, n} - \tfrac{1}{q+p/2} \varphi_{t-1, a, n}, \label{PTASEP n flows part}\\
             \nabla_n^- \bar{\varphi}_{t ,a, n} &= -(2 \bar{\varphi}_{t, a-1, n} -\bar{\varphi}_{t, a, n }  - \tfrac{1}{q+p/2} \bar{\varphi}_{t+1, a, n}),
\end{align}
and 
\begin{align}
            \nabla_t^+ \varphi_{t, a, n} &= -\tfrac{p}{2(q+p/2)} \nabla_a^- \varphi_{t, a, n}, \quad 
            \nabla_t^- \bar{\varphi}_{t, a, n} = -\tfrac{p}{2(q+p/2)}\nabla_a^+ \bar{\varphi}_{t, a, n}.\label{PTASEP t flows part}
\end{align}
Next, for given initial data $\mathbf{y}$, let $B_m$ be a discrete-time random walk with jumps with transition matrix $Q(x, y) = \tfrac{1}{2(q+p/2)}(\tfrac{1}{2})^{x-y-1}q^{\eta(x-y)}1_{x >y}$, where $\eta(z) = 1_{z=1}$, and let $\tau = \min \{m \geq 0 : B_m > y_{m+1}\}$. Define
\begin{align}
    \psi_{t,a, n}(u) &= \varphi_{t, a, n}(u), \quad \phi_{t, a, n}^{\mathbf{y}}(v) = \E_{B_0 = v}\left[\bar{\varphi}_{t, a, n-\tau}(B_{\tau})1_{\tau < n}\right].
\end{align}
\begin{lemma}[{\cite[Prop.~2.3]{Matetski_2022}}]
    With $F_{t,a ,n}$ as above,
    \begin{align}
        F_{t, a, n} &= \det(I-K_{t,a, n}^{\mathbf{y}})_{\ell^2(\Z)}, \quad \text{ with } \quad K_{t, a, n}^{\mathbf{y}} = \sum_{r=-\infty}^a \psi_{t, r, n} \otimes \phi_{t,r, n}^{\mathbf{y}} \label{PTASEP FP kernel}.
    \end{align}
\end{lemma}
\begin{proof}
     Define
\begin{align*}
     \mathcal{S}_{-t, -n}(u, v) &=  \frac{q^{n-1}}{2\pi i} \oint_{\gamma_r}  \frac{1}{2^{v-u}w^{v-u+n+1}}(1-w)^n(q+pw)^{t-(n-1)}dw\\
     \bar{\mathcal{S}}_{-t, n}(u, v) &= \frac{q^{-(n-1)}}{2\pi i} \oint_{\gamma_{\delta}} \frac{(1-w)^{u-v+n-1}}{2^{v-u}w^n} (1-pw)^{-t+n-1}dw 
 \end{align*}
 and
 \begin{align*}
     \bar{\mathcal{S}}_{-t, n}^{\text{epi}(\mathbf{y})}(u, v) = \E_{B_0 = u} [\bar{\mathcal{S}}_{-t, n-\tau}(B_{\tau}, v)1_{\tau < n}].  
 \end{align*}
 A special case of \cite[Prop.~2.3]{Matetski_2022} gives the following trace-class integral operator on $\ell^2(\Z)$
 \begin{align*}
     K_{t,a, n}(u, v) &= 1_{\{u \leq a\}}\Bigl(\sum_{r\in\Z} \mathcal{S}_{-t,-n}(r, u)\bar{\mathcal{S}}^{\text{epi}(\mathbf{y})}_{-t, n}(r, v) \Bigr)1_{\{v \leq a\}}.
 \end{align*}
 Using Lem.\ \ref{AB BA}--\ref{lem:det-transpose}, we can therefore write 
    \begin{align*}
     K_{t, a, n}(u, v) &\rightarrow  \sum_{r=-\infty}^a \bar{\mathcal{S}}_{-t, n}^{\text{epi}({\mathbf{y}})}(u, r)\mathcal{S}_{-t, -n}(v, r) \rightarrow  \sum_{r=-\infty}^a \mathcal{S}_{-t, -n}(u, r)\bar{\mathcal{S}}_{-t, n}^{\text{epi}({\mathbf{y}})}(v, r),
 \end{align*}
 without changing the value of the Fredholm determinant.  Conjugating by the appropriate terms yields the stated representation. 
 \end{proof}
\begin{proof}[Proof of Thm.\ \ref{Parallel TASEP part thm}]
We show that the kernel given in \eqref{PTASEP FP kernel} satisfies the conditions of Thm.\ \ref{PTASEP K THM}. Note the $a$--flow is immediate from the integral representation. The $n$--flows and $t$--flows for $\psi_{t, a, n}, \phi_{t, a, n}^{\mathbf{y}}$ follows from \eqref{PTASEP n flows part}--\eqref{PTASEP t flows part} and noticing that $\bar{\varphi}_{t, a, n-\tau-1} = 0$ when $\tau = n-1$. Moreover, since $\psi_{t, a, n}(u)$ has support on $a \geq u-n$, and $\phi_{t, a, n}^{\mathbf{y}}(v)$ has support on $v > y_n$, it is clear $K_{t,a, n}^{\mathbf{y}} \in B_r$, as required. 
\end{proof}
\subsubsection{Bernoulli Jumps with Blocking}
Bernoulli Jumps with Blocking is another discrete-time variant of TASEP, except now the dynamics are updated sequentially from right to left as follows: at each update $t \mapsto t+1$, the $k$-th particle attempts to jump one unit to the right with probability $p$. Each particle's jump is performed only if the destination site is empty at time $t+1$, otherwise it is suppressed. 
\begin{theorem}\label{RB Part Thm}
    Fix one-sided initial data $\mathbf{y} = (y_1, y_2, \dots, y_N)$ with $y_1 > y_2 > \dots > y_N$ and let $V = \{(t, a, n) : t \in \Z_{>0}, a < y_n + t, 1 \leq n \leq N \}$ with
   \begin{align}
       F_{t,a, n} = \PP(Y_n^{\text{Bernoulli-Blocking}}(t) > a \, | \, \mathbf{Y}(0) = \mathbf{y}), \qquad (t, a, n) \in V. 
   \end{align}
   Then $F_{t, a, n}$ satisfies
   \begin{align}
       \left[e^{D_t} - pe^{-D_a} - (1-p)\right] F_{t,a,n}\cdot F_{t+1, a, n-1} &= 0, \quad (\nabla_t^+ + p\nabla_a^- )F_{t, a, 1} = 0,\quad F_{0, a, n} = 1_{y_n > a}.
   \end{align}
\end{theorem}
\begin{remark}
    By extending the initial condition by setting $y_m \equiv \infty$ if $m < 1$ (so $F_{t, a, m} \equiv 1$), the bilinear equation already implies the $n=1$ forward equation.
\end{remark}
Let
\begin{align*}
    \varphi_{t, a, n}(u) &= \frac{1}{2\pi i} \oint_{\gamma_r}  \frac{1}{2^{a-u}w^{a-u+n+1}}(1-w)^n (q+pw)^{t}(q+p/2)^{-t} dw \\
    \bar{\varphi}_{t, a, n}(v) &= \frac{1}{2 \pi i} \oint_{\gamma_{\delta}} \frac{(1-w)^{a-v +n-1}}{2^{v-a}w^n}(1-pw)^{- t}(q+p/2)^{t} dw, 
\end{align*}
where $q = 1-p$, and $\gamma_r$, $\gamma_{\delta}$ are simple counter-clockwise loops around the origin with radius $r\in (0, 1)$ and $\delta \in (0, p^{-1})$, respectively. It is elementary to see that these satisfy 
\begin{alignat}{2}
    \nabla_n^+ \varphi_{t, a, n} &= 2\nabla_a^+ \varphi_{t,a, n}, &\qquad \nabla_n^- \bar{\varphi}_{t, a, n} &= 2\nabla_a^-\bar{\varphi}_{t, a, n}, \\
    \nabla_t^+ \varphi_{t, a, n} &= -\tfrac{p/2}{q+p/2}\nabla_a^- \varphi_{t,a, n}, &\qquad \nabla_t^- \bar{\varphi}_{t, a, n} &= -\tfrac{p/2}{q+p/2} \nabla_a^+\bar{\varphi}_{t, a, n}.
\end{alignat}
Next, for given initial data $\mathbf{y}$, let $B_m$ be a discrete-time Geom(1/2) random walk with jumps strictly in the negative direction, and let $\tau = \min \{m \geq 0 : B_m > y_{m+1}\}$. Define
\begin{align}
    \psi_{t,a, n}(u) &= \varphi_{t, a, n}(u), \quad \phi_{t, a, n}^{\mathbf{y}}(v) = \E_{B_0 = v}\left[\bar{\varphi}_{t, a, n-\tau}(B_{\tau})1_{\tau < n}\right].
\end{align}
\begin{lemma}[{\cite[Prop.~2.3]{Matetski_2022}}]
    With $F_{t,a ,n}$ as above,
    \begin{align}
        F_{t, a, n} &= \det(I-K_{t,a, n}^{\mathbf{y}})_{\ell^2(\Z)}, \quad \text{ with } \quad K_{t, a, n}^{\mathbf{y}} = \sum_{r=-\infty}^a \psi_{t, r, n} \otimes \phi_{t,r, n}^{\mathbf{y}} \label{RBTASEP FP kernel}.
    \end{align}
\end{lemma}
The proofs of the above lemma and Thm.\ \ref{RB Part Thm} are completely analogous to the proof of Thm.\ \ref{Parallel TASEP part thm}, and so omitted.  
\subsubsection{Bernoulli Jumps with Pushing}
Bernoulli Jumps with Pushing is a discrete-time variant of Push-TASEP, with dynamics updated sequentially from right to left as follows: for each update $t \mapsto t+1$, the $k$-th particle independently jumps leftward by one unit with probability $q$, and stays put with probability $p=1-q$, except that particle $k$ is forced to jump if particle $k-1$ arrives on top of it in the update to time $t+1$. 
\begin{theorem}\label{LB part thm}
    Fix one-sided initial data $\mathbf{y} = (y_1, y_2, \dots, y_N)$ with $y_1 > y_2 > \dots > y_N$ and let $V = \{(t, a, n) : t \in \Z_{>0}, a < y_n, 1 \leq n \leq N\}$ with
   \begin{align}
       F_{t,a, n} = \PP(Y_n^{\text{Bernoulli-Pushing}}(t) > a \, | \, \mathbf{Y}(0) = \mathbf{y}), \qquad (t, a, n) \in V. 
   \end{align}
   Then $F_{t, a, n}$ satisfies
   \begin{align}
       \left[e^{D_t} - qe^{D_a} - (1-q)\right] F_{t,a,n}\cdot F_{t+1, a+1, n-1} &= 0, \quad (\nabla_t^+ - q\nabla_a^+ )F_{t, a, 1} = 0, \quad F_{0, a, n} = 1_{y_n > a}.
   \end{align}
\end{theorem}
\begin{remark}
    By extending the initial condition by setting $y_m \equiv \infty$ if $m < 1$ (so $F_{t, a, m} \equiv 1$), the bilinear equation already implies the $n=1$ forward equation.
\end{remark}
Let
\begin{align*}
    \varphi_{t, a, n}(u) &= \frac{1}{2\pi i} \oint_{\gamma_r}  \frac{1}{ 2^{a-u}w^{a-u+n+1}}(1-w)^n (p + \frac{q}{w})^t (p+2q)^{-t} dw\\
    \bar{\varphi}_{t, a, n}(v) &= \frac{1}{2 \pi i} \oint_{\gamma_{\delta}} \frac{(1-w)^{a-v +n-1}}{2^{v-a}w^n}(p+\frac{q}{1-w})^{-t}(p+2q)^t dw, 
\end{align*}
where $p = 1-q$, and $\gamma_r$, $\gamma_{\delta}$ are simple counter-clockwise loops around the origin with radius $r, \delta \in (0, 1)$. It is elementary to see that these satisfy 
\begin{alignat}{2}
    \nabla_n^+ \varphi_{t, a, n} &= 2\nabla_a^+ \varphi_{t,a, n}, &\qquad \nabla_n^- \bar{\varphi}_{t, a, n} &= 2\nabla_a^-\bar{\varphi}_{t, a, n}, \\
    \nabla_t^+ \varphi_{t, a, n} &= \tfrac{q}{q+p/2}\nabla_a^+ \varphi_{t,a, n}, &\qquad \nabla_t^- \bar{\varphi}_{t, a, n} &=  \tfrac{q}{q+p/2} \nabla_a^-\bar{\varphi}_{t, a, n}.
\end{alignat}
Next, for given initial data $\mathbf{y}$, let $B_m$ be a discrete-time Geom(1/2) random walk with jumps strictly in the negative direction, and let $\tau = \min \{m \geq 0 : B_m > y_{m+1}\}$. Define
\begin{align}
    \psi_{t,a, n}(u) &= \varphi_{t, a, n}(u), \quad \phi_{t, a, n}^{\mathbf{y}}(v) = \E_{B_0 = v}\left[\bar{\varphi}_{t, a, n-\tau}(B_{\tau})1_{\tau < n}\right].
\end{align}
\begin{lemma}[{\cite[Prop.~2.7]{Matetski_2022}}]
    With $F_{t,a ,n}$ as above,
    \begin{align}
        F_{t, a, n} &= \det(I-K_{t,a, n}^{\mathbf{y}})_{\ell^2(\Z)}, \quad \text{ with } \quad K_{t, a, n}^{\mathbf{y}} = \sum_{r=-\infty}^a \psi_{t, r, n} \otimes \phi_{t,r, n}^{\mathbf{y}} \label{LBTASEP FP kernel}.
    \end{align}
\end{lemma}
The proofs of the above lemma and Thm.\ \ref{LB part thm} are completely analogous to the proof of Thm.\ \ref{Parallel TASEP part thm}, and so omitted. 
\section{Scaling of Bilinear Equations}\label{sec: 4}
In this section we present a selection of scaling limits for our bilinear equations (see Fig.\ \ref{fig:scaling-limits}). These examples correspond to known or conjectured scaling limits of the associated KPZ models. An important disclaimer: in this section, our method is purely formal. By Taylor expanding the bilinear operators, we obtain our desired asymptotic equations at leading order, but we do not attempt a rigorous justification of convergence at the level of solutions. 
\begin{figure}[t]
\centering
\begin{tikzpicture}[scale=0.75,
  >=Latex,
  node distance=10mm and 18mm,
  every node/.style={transform shape},
  eqnode/.style={rounded rectangle, draw, thick, minimum width=30mm, align=center, inner sep=2pt},
  shown/.style={-Latex, semithick, shorten >=2pt, shorten <=2pt},
  notshown/.style={-Latex, dashed, line width=0.5pt, opacity=.35, shorten >=2pt, shorten <=2pt}
]

\node[eqnode]                  (ptasep) {Parallel\\TASEP};
\node[eqnode, right=of ptasep] (tasep)  {TASEP};
\node[eqnode, right=of tasep]  (rbm)    {RBM};

\node[eqnode, above=8mm of ptasep] (png) {2DTL (PNG)};
\node[eqnode, above=20mm of tasep] (kp)  {KP (KPZ FP)};

\node[eqnode, below=of ptasep] (rb)    {Bernoulli\\with Blocking};
\node[eqnode, below=of tasep]  (pusht) {Push-TASEP};
\node[eqnode] (lb) at ($(rbm|-rb)$) {Bernoulli\\with Pushing};

\coordinate (KP_inS)    at (kp.south);
\coordinate (KP_below)  at ($(kp.south)+(0,-2mm)$);
\coordinate (RBM_inW)   at ($(rbm.west)+(0,0)$);
\coordinate (TASEP_inW) at ($(tasep.west)+(-.5mm,0)$);
\coordinate (RBM_inS)   at ($(rbm.south)+(0,-.5mm)$);

\draw[shown] (ptasep.north)      to[bend left=10] (png.south);
\draw[shown] (ptasep.east)       to[bend left=8]  (TASEP_inW);
\draw[shown] (ptasep.east)       to[bend left=15] (RBM_inW);
\draw[shown] (ptasep.north east) to[bend left=12] (KP_inS);

\draw[shown] (tasep.east)        to[bend left=8]  (RBM_inW);
\draw[shown] (tasep.north)       to[bend left=10] (KP_inS);

\draw[shown] (rbm.north)         to[bend left=8]  (KP_inS);

\draw[notshown] (png.10) to[out=10, in=180] (KP_below) -- (KP_inS);

\draw[notshown] (rb.north east)  to[bend left=14] (TASEP_inW);
\draw[notshown] (rb.north east)  to[bend left=8] (RBM_inW);
\draw[notshown] (rb.north east)  to[bend left=18] (KP_inS);

\draw[notshown] (lb.west)        to[bend left=10] (pusht.east);
\draw[notshown] (lb.north)       to[bend left=10] (RBM_inS);
\draw[notshown] (lb.north west)  to[bend left=14] (KP_inS);

\draw[notshown] (pusht.north east) to[bend left=8]  (RBM_inS);
\draw[notshown] (pusht.north east) to[bend left=14] (KP_inS);

\end{tikzpicture}
\caption{Bilinear equations scaling relationships. Solid arrows indicate demonstrated scaling limits, dashed arrows indicate scaling limits not shown here.}
\label{fig:scaling-limits}
\end{figure}
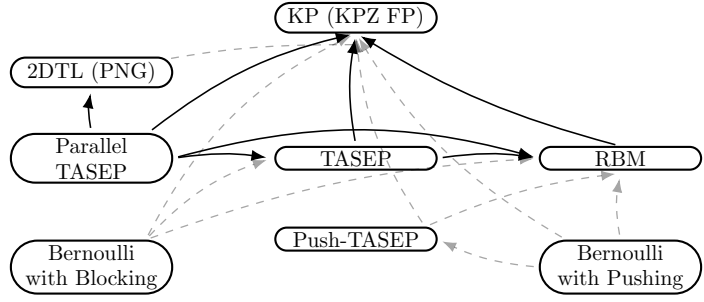
\subsection{RBM Equation}
\begin{example} (\text{KPZ Scaling}).
    Under the change of variables
    \begin{equation}
        T = \epsilon^{3/2}t, \quad X = \tfrac\epsilon2( t - n), \quad A = -\epsilon^{1/2}(t + a + n), \label{RBM KPZ coord}
    \end{equation}
  the RBM equation formally reduces, as $\epsilon \rightarrow 0$, to 
    \begin{align}
       \left[D_TD_A + \tfrac{1}{4}D_X^2 + \tfrac{1}{12}D_A^4\right]F \cdot F  = 0, \qquad F = F(T, X, A).\label{KP2 for RBM}
    \end{align}
\end{example}
    
\begin{proof}[Derivation] It will be convenient to relabel $n = \lfloor m/2\rfloor + 1$ so that the RBM equation becomes 
\begin{equation*}
    \left[D_t - \tfrac{1}{2}D_a^2 \right]e^{D_m} F \cdot F = 0.
\end{equation*}
Using the change of variables in \eqref{RBM KPZ coord}, we have 
\begin{gather*}
      D_t - \tfrac{1}{2}D_a^2 = - \epsilon^{1/2} D_A + \tfrac{\epsilon}{2}(D_X - D_A^2) + \epsilon^{3/2}D_T, \\
      \shortintertext{ and }
       e^{D_m} =  1 - \tfrac{\epsilon^{1/2}}{2}D_A - \tfrac{\epsilon}{4}(D_X - \tfrac{1}{2}D_A^2) + \tfrac{\epsilon^{3/2}}{8}(D_XD_A - \tfrac{1}{6}D_A^3) + o(\epsilon^{3/2}).
\end{gather*}
Therefore, as operators, we have 
\begin{equation}\label{RBM expansion}
\begin{split}
      \left[D_t - \tfrac{1}{2}D_a^2 \right]e^{D_m} 
    = &-\epsilon^{1/2} D_A +  \tfrac{\epsilon}{2}D_X + \epsilon^{3/2}\left[D_T + \tfrac{1}{8}D_A^3\right]  \\
    \quad &-\tfrac{\epsilon^2}{2} \left[D_TD_A+ \tfrac14D_X^2 - \tfrac18D_XD_A^2 + \tfrac{1}{12}D_A^4\right] + o(\epsilon^2).
\end{split}
\end{equation}
Now, applying \eqref{RBM expansion} to $F\cdot F$, and noting that $P(D_1, D_2, \dots) G\cdot G=0$ whenever $P$ is a polynomial of odd total order, we get to leading-order
\begin{align}
    0 &= \left[D_t - \frac{1}{2}D_a^2 \right]e^{D_m} F \cdot F 
    = -\tfrac{1}{2}\epsilon^2 \left[D_TD_A + \tfrac14D_X^2 + \tfrac{1}{12}D_A^4\right]F \cdot F + o(\epsilon^2) 
\end{align}
as desired.
\end{proof}
\subsection{TASEP Equation}
\begin{example}
We obtain the following scaling limits for the TASEP equation: 
\begin{enumerate}[label=(\roman*), leftmargin=2.4em, itemsep=2pt, topsep=2pt]
    \item \label{TSEPKPZSCALIN}(\text{KPZ Scaling}). Under the change of variables 
    \begin{equation}
        T = \tfrac{1}{2}\epsilon^{3/2}t, \quad X = \tfrac{\epsilon}{2}(a+2), \quad A= \epsilon^{1/2}\left(\tfrac{1}{2}t-2n-a\right),\label{cov for TASEP to KP}
    \end{equation}
  the TASEP equation formally reduces, as $\epsilon \rightarrow 0$, to
        \begin{align}
        \left[D_TD_A + \tfrac{1}{4}D_X^2 + \tfrac{1}{12}D_A^4\right]F \cdot F = 0. \label{KP2 for TASEP}
    \end{align}
    \item \label{TASEPRBMSCALEING}(\text{RBM Scaling}). Under the change of variables 
    \begin{equation}
        T= \epsilon^2 t, \qquad  A=\epsilon(a-t),
    \end{equation}
   the TASEP equation formally reduces, as $\epsilon \rightarrow 0$, to
    \begin{align}
        \left[D_T - \tfrac{1}{2}D_A^2\right]F_n \cdot F_{n-1} = 0. \label{RBM for TASEP}
    \end{align}
\end{enumerate}
\end{example}
\begin{proof}[Derivations]
\ref{TSEPKPZSCALIN}: Relabeling $n = \lfloor m/2\rfloor + 1$,  the TASEP equation becomes 
\begin{equation*}
    \left[D_t - (e^{-D_a} - 1)\right]e^{D_m} F\cdot F= 0.
\end{equation*}
Using the change of variables in \eqref{cov for TASEP to KP}, we have 
\begin{align*}
    D_t &=\tfrac12 \epsilon^{1/2}D_A +  \tfrac12 \epsilon^{3/2}D_T, \qquad  e^{D_m} = 1 - \epsilon^{1/2}D_A + \tfrac12 \epsilon D_A^2 - \tfrac16 \epsilon^{3/2} D_A^3 + o(\epsilon^{3/2}). \\
    \shortintertext{ and }
    e^{-D_a} -1 &=  \epsilon^{1/2}D_A  +\tfrac12\epsilon(D_A^2 -D_X) + \tfrac12\epsilon^{3/2}(-D_AD_X + \tfrac{1}{3}D_A^3)  + \tfrac{1}{4}\epsilon^2 (\tfrac12D_X^2 - D_A^2D_X + \tfrac{1}{6}D_A^4) \\
    &\qquad + o(\epsilon^2)
\end{align*}
Therefore, as operators, we have 
\begin{align*}
    \left[D_t - (e^{-D_a} - 1)\right]e^{D_m} &= -\tfrac{1}{2}\epsilon^{1/2}D_A + \tfrac12\epsilon  D_X + \tfrac12\epsilon^{3/2}(D_T + \tfrac16 D_A^3) \\
    &\quad -\tfrac12\epsilon^2( D_TD_A +  \tfrac14 D_X^2 + \tfrac1{12}D_A^4) + o(\epsilon^2).
\end{align*}
Applying this to $F\cdot F$ and dropping odd total order Hirota derivative terms, we obtain
\begin{align*}
    -\tfrac12\epsilon^2\left[D_TD_A + \tfrac{1}{4}D_X^2 + \tfrac{1}{12}D_A^4\right]F\cdot F + o(\epsilon^2) = 0.
\end{align*}

\noindent\ref{TASEPRBMSCALEING}: As operators, we have 
\begin{align*}
    \left[D_t -e^{-D_a} + 1\right] &=  \left[\epsilon^2 D_T - \epsilon D_A - (1- \epsilon D_A + \tfrac{\epsilon^2}{2}D_A^2) + 1\right] + o(\epsilon^2) \\
    &= \epsilon^2 \left[D_T - \tfrac{1}{2}D_A^2\right] + o(\epsilon^2).
\end{align*}
\end{proof}
\subsection{Parallel TASEP}\label{sec: Parallel-TASEP-Scal}
To start this subsection, let us first make the direct connection to the HBDE. Let $n=\lfloor m/2\rfloor + 1$, so the Parallel TASEP formula becomes 
\begin{align}\label{PTASEP with m COV}
    \left[e^{D_t} - pe^{-D_a} -(1-p)\right]e^{D_m}F_{t, a, m} \cdot F_{t, a, m}= 0. 
\end{align}
We now introduce the change of variables 
\begin{align}
    t= t, \quad x = a, \quad r = t-a-m. \label{PTASEP to HBDE COV}
\end{align}
In these new coordinates, the Parallel TASEP formula transforms into the familiar HBDE, i.e.\ 
\begin{align}
    \left[e^{D_t} -pe^{D_x} -(1-p)e^{D_r}\right]F_{t, x, r}\cdot F_{t, x, r} = 0. \label{ptasep HBDE form}
\end{align}
\begin{example}
We have the following scaling limits of the Parallel TASEP equation: 
\begin{enumerate}[label=(\roman*), leftmargin=2.4em, itemsep=2pt, topsep=2pt]
    \item \label{PTASEPSCALINGKPZ}(\text{KPZ scaling}). Under the change of variables, 
    \begin{equation}
        T = (\tfrac{p}{2q})^{1/4}\epsilon^{3/2}t, \quad X = \tfrac{\epsilon}{\sqrt{2p}}(a+2), \quad A= \epsilon^{1/2}\bigl((\tfrac{2^{1/4}(1-\sqrt{q})}{(pq)^{1/4}})t-(\tfrac{2}{pq})^{1/4}(2n+a)\bigr),
    \end{equation}
    with $p \in (0, 1)$, $q = 1-p$, the Parallel TASEP equation formally reduces, as $\epsilon \rightarrow 0$, to
        \begin{align}
        \left[D_TD_A + \frac{1}{4}D_X^2 + \frac{1}{12}D_A^4\right]F \cdot F = 0 . \label{KP2 for parallelTASEP}
    \end{align}
    \item \label{PTASEPSCALINGRBM}(\text{RBM scaling}). Under the change of variables
    \begin{equation}
        T= \epsilon^2t, \quad  A = \epsilon^{1/2}(a-\epsilon t), \quad p = \epsilon,
    \end{equation} 
     the Parallel TASEP equation formally reduces, as $\epsilon \rightarrow 0$, to
    \begin{align}
        \left[D_T - \frac{1}{2}D_A^2\right]F_n \cdot F_{n-1} = 0. \label{RBM for ParallelTASEP}
    \end{align}
    \item \label{PTASEPSCALING2DTL}(\text{2DTL Scaling}). Under the change of variables 
    \begin{equation}
        T=\epsilon t, \quad  X = \epsilon x, \quad  p = 1-4\epsilon^2,
    \end{equation}
    we have \eqref{ptasep HBDE form} formally reduces, as $\epsilon \rightarrow 0$, to 
    \begin{align}
        \left[\frac{1}{2}D_T^2 - \frac{1}{2}D_X^2 - 4(e^{D_r} - 1)\right]F \cdot F = 0. \label{2D Toda for ParallelTASEP}
    \end{align}
    \item \label{PTASEPSCALINGTASEP}(\text{TASEP Scaling}). Under the change of variables
    \begin{equation}
        T= \epsilon t, \quad p = \epsilon,
    \end{equation}
    the Parallel TASEP equation formally reduces, as $\epsilon \rightarrow 0$, to 
    \begin{align}
        \left[D_T - ( e^{-D_a} -1)\right]F_{n}\cdot F_{n-1} &= 0. \label{tasep for parallel tasep}
    \end{align}
\end{enumerate}
\end{example}
\begin{proof}[Derivations] \ref{PTASEPSCALINGKPZ}: 
    Relabeling $n = \lfloor m/2\rfloor + 1$, we have
\begin{align*}
    e^{D_t + D_m} &= e^{c_1 \epsilon^{3/2}D_T + \epsilon^{1/2}(c_2-c_4) D_A } \\
     &= 1 + \epsilon^{1/2}dD_A + \epsilon \frac{d^2}{2} D_A^2 + \epsilon^{3/2}(c_1 D_T + \frac{d^3}{6}D_A^3) + \epsilon^2( c_1d D_T D_A+\frac{d^4}{24}D_A^4 ) + o(\epsilon^2), 
\end{align*}
with $c_1 = (\tfrac{p}{2q})^{1/4}, c_2 = (\tfrac{2^{1/4}(1-\sqrt{q})}{(pq)^{1/4}}), c_4 = (\tfrac{2^{1/4}}{(pq)^{1/4}}), d= (c_2-c_4) = -\sqrt{q}c_4$. Also 
\begin{align*}
   -p e^{-D_a + D_m} &= -pe^{-c_3\epsilon D_X} = -p(1 - c_3 \epsilon D_X + \frac{c_3^2}{2}\epsilon^2 D_X^2) + o(\epsilon^2)
\end{align*}
with $c_3 = \frac{1}{\sqrt{2p}}$, and 
\begin{align*}
    -q e^{D_m} &= -q(1- \epsilon^{1/2}c_4 D_A + \epsilon \frac{c_4^2}{2} D_A^2 - \epsilon^{3/2}\frac{c_4^3}{6}D_A^3 + \epsilon^2 \frac{c_4^4}{24}D_A^4 ) + o(\epsilon^2).
\end{align*}
Therefore, applying \eqref{PTASEP with m COV} in the scaled variables to $F\cdot F$, dropping vanishing odd total order Hirota terms, and noticing $d = - \sqrt{q}c_4$ so the $D_A^2$ drops, we obtain  
\begin{align*}
    -\epsilon^2 \left[ D_T D_A + \frac{1}{4}D_X^2 + \frac{1}{12}D_A^4\right] F\cdot F + o(\epsilon^2) = 0. 
\end{align*}
\noindent\ref{PTASEPSCALINGRBM}: Expanding, we have
\begin{align*}
    0&= \left[1+ \epsilon^2 D_T - \epsilon^{3/2}D_A-\epsilon(1-\epsilon^{1/2}D_A + \frac{\epsilon}{2}D_A^2) - 1+ \epsilon\right]F_{n}\cdot F_{n-1} + o(\epsilon^2) \\
    &= \epsilon^2 \left[D_T - \frac{1}{2}D_A^2\right]F_n \cdot F_{n-1} + o(\epsilon^2).
\end{align*}

\noindent\ref{PTASEPSCALING2DTL}: We have 
\begin{align*}
    &\left[1+\epsilon D_T + \tfrac12\epsilon^2D_T^2 - (1-4\epsilon^2)(1+ \epsilon D_X + \tfrac{1}{2}\epsilon^2 D_X^2) -4\epsilon^2 e^{D_r} \right] + o(\epsilon^2)
\end{align*}
Applying this to $F\cdot F$ and dropping vanishing odd total order Hirota terms, we have 
\begin{align*}
    0&= \left[e^{D_t} -pe^{D_x} -(1-p)e^{D_r}\right]F\cdot F \\
    &= \epsilon^2\left[\tfrac12D_T^2 - \tfrac12D_X^2 - 4(e^{D_r}-1)\right]F\cdot F + o(\epsilon^2). 
\end{align*}

\noindent\ref{PTASEPSCALINGTASEP}: We have 
\begin{align*}
    0 &= \left[1+ \epsilon D_T - \epsilon e^{-D_a} - (1-\epsilon)\right] F_n \cdot F_{n-1}  + o(\epsilon) \\
    &=  \epsilon \left[D_T - ( e^{-D_a} -1)\right]F_{n}\cdot F_{n-1} + o(\epsilon).
\end{align*}
\end{proof}
\section{Zero-Curvature and Lax Pairs}\label{sec: 5}
\newcommand{\overbar}[1]{\mkern 1.5mu\overline{\mkern-1.5mu#1\mkern-1.5mu}\mkern 1.5mu}
\subsection{RBM Equation}
Fix a sequence of non-vanishing functions $\{F_{t, a, n}\}_{n \in \Z}$, and define 
\begin{align}
    a_{n} &= \frac{F_{n+1}F_{n-1}}{F_n^2}, \quad u_n = \partial_a \log(F_{n}), \quad \nabla^s_n u_n = \frac{u_{n+1}-u_{n-1}}{2}.
\end{align}
Let the backward shift $e^{-\partial_n}$ act on sequences by $(e^{-\partial_n}f)_n = f_{n-1}$ and define the  operator\footnote{Here, we mean $(\mathcal{R}f)_n = a_n f_{n-1}$, and similarly $(\mathcal{R}^2f)_n = a_na_{n-1}f_{n-2}$.}
\begin{align}
    \mathcal{R} \defeq a_ne^{-\partial_n}, \quad \text{ so that } \mathcal{R}^2 = a_n a_{n-1}e^{-2\partial_n}.
\end{align}
\begin{theorem}[\textbf{RBM Eq.\ Zero-Curvature Condition}]\label{ZC for RBM}
Fix $M \in \Z$ and a collection of non-vanishing functions $\{F_{t, a, n}\}_{n\in \Z}$ with boundary condition $F_{t,a,m} \equiv 1$ for all $m \leq M$. Consider the operators 
\begin{align}
    \mathcal{M} &\defeq \partial_t + (\nabla^s_n u_n)\mathcal{R} + \frac{1}{2}\mathcal{R}^2, \qquad 
    \bar{\mathcal{M}} \defeq \partial_a + \mathcal{R}, 
\end{align}
    acting on functions $f_n(t, a)$. Then the zero-curvature condition $[\mathcal{M}, \bar{\mathcal{M}}] = 0$ is equivalent to 
    \begin{align}
        \left[D_t - \frac{1}{2}D_a^2 \right]F_{t, a, n}\cdot F_{t,a, n-1} = 0. 
    \end{align}
\end{theorem}
\begin{proof}
    First, notice 
    \begin{align*}
        [\mathcal{M}, \bar{\mathcal{M}}] &= [\partial_t, \mathcal{R}] + [\nabla_n^s u_n \mathcal{R}, \partial_a] + [\nabla_n^s u_n \mathcal{R}, \mathcal{R}] + \tfrac{1}{2}[\mathcal{R}^2, \partial_a], 
    \end{align*}
    by linearity and since $\partial_a,\partial_t$ commute as well as $\mathcal{R}, \mathcal{R}^2$. 
    Let $v_n = \partial_t \log(F_n)$. Then
    \begin{align*}
        [\partial_t, \mathcal{R}] &= \partial_t a_n e^{-\partial_n} = a_n (v_{n+1} - 2v_n + v_{n-1}) e^{-\partial_n} = a_n \left(\frac{D_t F_{n+1}\cdot F_n}{F_{n+1}F_n} - \frac{D_t F_{n}\cdot F_{n-1}}{F_n F_{n-1}}\right)e^{-\partial_n},
    \end{align*}
   and
    \begin{align*}
          [\nabla_n^s u_n \mathcal{R}, \partial_a]&= -\frac{a_n}{2}\left((u_{n+1} - 2u_n + u_{n-1})(u_{n+1}-u_{n-1}) + \partial_a u_{n+1} - \partial_a u_{n-1} \right)e^{-\partial_n} \\
        &= -\frac{a_n}{2}( \partial_a u_{n+1} + \partial_a u_n + (u_{n+1}- u_n)^2 - \partial_a u_n - \partial_a u_{n-1} - (u_n - u_{n-1})^2) e^{-\partial_n} \\
        &= -a_n \left(\frac{1}{2}\frac{D_a^2 F_{n+1}\cdot F_{n}}{F_{n+1}F_n} - \frac{1}{2}\frac{D_a^2F_n \cdot F_{n-1}}{F_{n}F_{n-1}}\right)e^{-\partial_n}.
    \end{align*}
    Moreover, 
    \begin{align*}
        [\nabla_n^s u_n \mathcal{R}, \mathcal{R}] &= \frac{1}{2}a_na_{n-1} \left( (u_{n+1}-u_{n}) - (u_{n-1}-u_{n-2}) \right)e^{-2\partial_n} = \frac{1}{2}\partial_a (a_n a_{n-1})e^{-2\partial_n} \\
        &= -\tfrac{1}{2}[\mathcal{R}^2, \partial_a].
    \end{align*}
    Therefore, we have
    \begin{align*}
        [\mathcal{M}, \bar{\mathcal{M}}] &= a_n\left(\mathcal{K}_{n+1} - \mathcal{K}_{n}\right)e^{-\partial_n}, \qquad  \mathcal{K}_{n} = \frac{D_t F_{n}\cdot F_{n-1} - \frac{1}{2}D_a^2 F_{n}\cdot F_{n-1}}{F_{n}F_{n-1}}.
    \end{align*}
    If we require $[\mathcal{M}, \bar{\mathcal{M}}] = 0$, then we must have $\mathcal{K}_{n}$ is independent of $n$.  Specializing to $n \leq M$, we see we must have
    \begin{align}
        \mathcal{K}_{n} \equiv 0. 
    \end{align}
    Moreover it is obvious that if $\mathcal{K}_n \equiv 0$, then $[\mathcal{M}, \bar{\mathcal{M}}] =0$ by the above calculation, giving the theorem.
\end{proof}
\begin{corollary}[\textbf{RBM Eq.\ Lax Pair}]
    In the same setting as Thm.\ \ref{ZC for RBM}, define 
    \begin{align}
        L = \partial_a + \mathcal{R}, \quad P = (\nabla^s_n u_n) \mathcal{R} + \frac{1}{2}\mathcal{R}^2.
    \end{align}
    Then $(L, P)$ form a Lax pair for the RBM equation, i.e.\ 
    \begin{align}
        \partial_t L + [P, L] = 0 \Longleftrightarrow \left[D_t - \frac{1}{2}D_a^2 \right]F_{t, a, n}\cdot F_{t,a, n-1} = 0. 
    \end{align}
\end{corollary}
\subsection{TASEP Equation}
Fix a sequence of non-vanishing functions $\{F_{t, a, n}\}_{a, n \in \Z}$, and define 
\begin{align}
    r_{a, n} = \frac{F_{a-1, n+1}}{F_{a, n}}. 
\end{align} 
Consider the shift operators $e^{-\partial_n}, e^{\partial_a}$ acting on functions as $(e^{-\partial_n}f)_{a, n} = f_{a,n-1}$ and $(e^{\partial_a}f)_{a, n} = f_{a+1,n}$, respectively.
\begin{theorem}[\textbf{TASEP Eq.\ Zero-Curvature Condition}]\label{TASEP zero-curv thm}
    Fix $M \in \Z$ and a collection of non-vanishing functions $\{F_{t, a, n}\}_{a, n\in \Z}$ with boundary condition $F_{t,a,m} \equiv 1$ for all $m \leq M$. Consider the operators 
    \begin{align}
        \mathcal{M} &\defeq \partial_t - \frac{r_{a, n}}{r_{a+1, n-1}}e^{\partial_a - \partial_n}, \quad 
        \bar{\mathcal{M}} \defeq e^{-\partial_a} + \frac{r_{a, n}}{r_{a, n-1}}e^{-\partial_n},
    \end{align}
    acting on functions $f_{a, n}(t)$. Then the zero-curvature condition $[\mathcal{M}, \bar{\mathcal{M}}] = 0$ is equivalent to 
    \begin{align}
       \left[D_t - (e^{-D_a} -1)\right]F_{t, a, n}\cdot F_{t, a ,n-1}= 0. 
    \end{align}
\end{theorem}
\begin{proof}
    Let $v_{a,n} = \partial_t \log F_{t, a, n}$. We compute 
    \begin{align*}
        [\partial_t, r_{a, n}r_{a, n-1}^{-1}e^{- \partial_n}] &= \frac{r_{a, n}}{r_{a, n-1}}\left(v_{a-1, n+1} - v_{a-1, n} - (v_{a, n} - v_{a, n-1})\right)e^{-\partial_n} \\
        &= \frac{r_{a, n}}{r_{a, n-1}}\left(\frac{D_t F_{a-1, n+1}\cdot F_{a-1, n}}{F_{a-1, n+1}F_{a-1, n}} - \frac{D_t F_{a, n}\cdot F_{a, n-1}}{F_{a, n}F_{a, n-1}}\right)e^{-\partial_n}, 
    \end{align*}
    and
    \begin{align*}
        [ -r_{a, n}r_{a+1, n-1}^{-1}e^{\partial_a-\partial_n}, e^{-\partial_a}] &= -\left(r_{a-1, n}r_{a, n-1}^{-1} - r_{a, n}r_{a+1, n-1}^{-1}\right)e^{-\partial_n} \\
        &=  r_{a, n}r_{a, n-1}^{-1}\left(r_{a-1, n}r_{a, n}^{-1} - r_{a, n-1}r_{a+1, n-1}^{-1}\right)e^{-\partial_n} \\
        &= r_{a, n}r_{a, n-1}^{-1}\left(\frac{e^{-D_a }F_{a-1, n+1}\cdot F_{a-1, n}}{F_{a-1, n+1}F_{a-1, n}} - \frac{e^{-D_a}F_{a, n}\cdot F_{a, n-1}}{F_{a, n}F_{a, n-1}}\right)e^{-\partial_n}.
    \end{align*}
    We also have 
    \begin{align*}
        [r_{a, n}r_{a+1, n-1}^{-1}e^{\partial_a-\partial_n}, r_{a, n}r_{a, n-1}^{-1}e^{-\partial_n}] &= \left(r_{a, n}r_{a+1, n-2}^{-1} - r_{a, n}r_{a+1, n-2}^{-1}\right)e^{\partial_a - 2\partial_n} = 0
    \end{align*}
    and clearly $[\partial_t, e^{-\partial_a}] = 0$. Therefore, we have 
    \begin{align*}
        [\mathcal{M}, \bar{\mathcal{M}}] &= \frac{r_{a, n}}{r_{a, n-1}}(\mathcal{K}_{a-1, n+1} - \mathcal{K}_{a, n})e^{-\partial_n}, \qquad \mathcal{K}_{a, n} = \frac{D_t F_{a, n}\cdot F_{a, n-1} -e^{-D_a}F_{a, n}\cdot F_{a, n-1} }{F_{a, n}F_{a, n-1}}.
    \end{align*}
    If we require $[\mathcal{M}, \bar{\mathcal{M}}] = 0$, then we must have $\mathcal{K}_{a, n}$ is independent of shifts $(a, n)\rightarrow (a\pm1, n\mp1)$. Specializing to the boundary, we see we require   
    \begin{align}
        \mathcal{K}_{a, n} \equiv -1.
    \end{align}
    Moreover, it is obvious that if $\mathcal{K}_{a, n} \equiv -1$, then $[\mathcal{M}, \bar{\mathcal{M}}] =0$ by the above calculation, yielding the theorem.
\end{proof}
\begin{remark}
    From the proof, it is obvious that there are various boundary conditions we can consider for $F_{t, a, m}$ which will still give the above theorem. 
\end{remark}
\begin{corollary}[\textbf{TASEP Eq.\ Lax Pair}] In the same setting as Thm.\ \ref{TASEP zero-curv thm}, define 
\begin{align}
    L = e^{-\partial_a} + \frac{r_{a, n}}{r_{a, n-1}}e^{-\partial_n}, \quad P = -\frac{r_{a, n}}{r_{a+1, n-1}}e^{\partial_a -\partial_n}.
\end{align}
Then $(L, P)$ form a Lax pair for the TASEP equation, i.e.\ 
\begin{align*}
    \partial_t L + [P, L] = 0 \Longleftrightarrow \left[D_t - (e^{-D_a} -1)\right]F_{t, a, n}\cdot F_{t, a ,n-1}= 0.
\end{align*}
\end{corollary}
\subsection{Parallel TASEP}
Fix a sequence of non-vanishing functions $\{F_{t, a, n}\}_{t, a, n \in \Z}$, and define 
\begin{align*}
    r_{t, a, n} &= \frac{F_{t+1, a-1, n+1}}{F_{t, a, n}}.
\end{align*}
Consider the shift operators $e^{-\partial_n}, e^{\pm\partial_a}, e^{\pm\partial_t}$ acting on functions as $(e^{-\partial_n}f)_{t, a, n} = f_{t, a,n-1}$, $(e^{\pm\partial_a}f)_{t, a, n} = f_{t, a\pm1,n}$, and $(e^{\pm\partial_t}f)_{t, a, n} = f_{t\pm1,a,n}$ respectively.
\begin{theorem}[\textbf{Parallel TASEP Eq.\ Zero-Curvature Condition}]
    Fix $M \in \Z$ and a collection of non-vanishing functions $\{F_{t, a, n}\}_{t, a, n\in \Z}$ with boundary condition $F_{t,a,m} \equiv 1$ for all $m \leq M$. Consider the operators 
    \begin{align}
        \mathcal{M} &= e^{\partial_t} - c \frac{r_{t, a, n}}{r_{t, a+1, n-1}}e^{\partial_a - \partial_n}, \quad 
        \bar{\mathcal{M}} = - \bar{c}e^{-\partial_a} + \frac{r_{t, a, n}}{r_{t-1, a, n-1}}e^{-\partial_t - \partial_n},
    \end{align}
    acting on functions $f_{t, a, n}$, where $c, \bar{c}$ are arbitrary constants such that $c\bar{c}= p$. Then the zero-curvature condition $[\mathcal{M}, \bar{\mathcal{M}}] = 0$ is equivalent to 
    \begin{align}
        \left[e^{D_t} -pe^{-D_a} -(1-p)\right] F_{t, a, n}\cdot F_{t, a, n-1} = 0.
    \end{align}
\end{theorem}
\begin{remark}
    The presented zero-curvature conditions can be obtained through the change of variables from well-known zero-curvature conditions of the HBDE (see \cite{Zabrodin_19972}). 
\end{remark}
\begin{proof}
    We compute 
    \begin{align*}
        [e^{\partial_t}, r_{t, a, n}r_{t-1, a, n-1}^{-1}e^{-\partial_t - \partial_n}] &= \left(r_{t+1, a, n}r_{t, a, n-1}^{-1} - r_{t, a, n}r_{t-1, a, n-1}^{-1}\right)e^{-\partial_n} \\
        &= \left(\frac{F_{t+2, a-1, n+1} F_{t,a, n-1}}{F_{t+1, a, n}F_{t+1, a-1, n}} - \frac{F_{t+1, a-1, n+1}F_{t-1, a, n-1}}{F_{t, a, n}F_{t, a-1, n}}\right)\\
        &= \frac{F_{t+1, a-1, n+1}F_{t, a, n-1}}{F_{t+1, a, n} F_{t, a-1, n}}\left(\frac{F_{t+2, a-1, n+1}F_{t, a-1, n}}{F_{t+1, a-1, n+1}F_{t+1, a-1, n}} - \frac{F_{t+1, a, n} F_{t-1, a, n-1}}{F_{t, a, n}F_{t, a, n-1}}\right),
    \end{align*}
    and
    \begin{align*}
         &c\bar{c}[r_{t,a, n}r_{t, a+1, n-1}^{-1}e^{\partial_a -\partial_n},e^{-\partial_a}]
        = p\left(r_{t, a, n}r_{t, a +1, n-1}^{-1} - r_{t, a-1, n}r_{t, a, n-1}^{-1}\right)e^{-\partial_n} \\
        &\quad = p\left(\frac{F_{t+1, a-1, n+1}F_{t, a+1, n-1}}{F_{t, a, n}F_{t+1, a, n}}- \frac{F_{t+1, a-2, n+1}F_{t, a, n-1}}{F_{t, a-1, n}F_{t+1, a-1, n}}\right)e^{-\partial_n} \\
        &\quad = p\frac{F_{t+1, a-1, n+1}F_{t, a, n-1}}{F_{t+1, a, n}F_{t, a-1, n}}\left(\frac{F_{t, a-1, n}F_{t, a+1, n-1}}{F_{t, a, n}F_{t, a, n-1}}- \frac{F_{t+1, a-2, n+1}F_{t+1, a, n}}{F_{t+1, a-1, n+1}F_{t+1, a-1, n}}\right)e^{-\partial_n}.
    \end{align*}
    A straightforward calculation shows that all remaining terms commute. Therefore, we have
    \begin{align*}
        [\mathcal{M}, \bar{\mathcal{M}}]&= \frac{F_{t+1, a-1, n+1}F_{t, a, n-1}}{F_{t+1, a, n}F_{t, a-1, n}} \left(\mathcal{K}_{t+1, a-1, n+1} - \mathcal{K}_{t, a, n}\right)e^{-\partial_n}, \\
        \shortintertext{where}
        \mathcal{K}_{t,a, n} &= \frac{F_{t+1, a, n}F_{t-1, a, n-1} - pF_{t, a-1, n}F_{t, a+1, n-1}}{F_{t,a, n}F_{t,a, n-1}}.
    \end{align*}
    If we require $[\mathcal{M}, \bar{\mathcal{M}}] = 0$, then we must have $\mathcal{K}_{t, a, n}$ is independent of shifts $(t, a, n) \rightarrow (t\mp1, a\pm 1, n\mp1)$. Specializing to the boundary, we obtain
    \begin{equation}
        \mathcal{K}_{t,a, n} \equiv 1-p. 
    \end{equation}
    Moreover, it is obvious that if $\mathcal{K}_{t,a, n} \equiv 1-p$, then $[\mathcal{M}, \bar{\mathcal{M}}] = 0$ by the above calculation, yielding the theorem. 
\end{proof}

\appendix

\section{Fredholm Determinants and Elementary Lemmas}\label{Fredholm Determinants and Elemantary Lemmas}
If $K$ is a trace-class integral operator acting on the Hilbert space $\mathcal{H} = L^2(X, \mu)$ through its kernel $Kf(x) = \int_X K(x, y)f(y)d\mu(y)$, its Fredholm determinant is defined by
\begin{align}
    \det(I+K) = \sum_{n=0}^{\infty} \frac{1}{n!} \int_{X^n} \det \bigl[ K(x_i, x_j) \bigr]_{i,j=1}^n \, d\mu(x_1)\cdots d\mu(x_n).
\end{align}
We now collect several standard lemmas for Fredholm determinants that will be used throughout the paper. 
For background and proofs, see \cite{Sim05} or \cite[Sec.~2]{QR14}.
\begin{lemma}[\textbf{Cyclicity}]\label{AB BA}
If $A \in \mathcal{I}_1$ and $B$ is a bounded operator on $\mathcal{H}$, then $AB, BA \in \mathcal{I}_1$, and
    \begin{align*}
        \det(I-AB) &= \det(I-BA).
    \end{align*}
\end{lemma}
\begin{lemma}[\textbf{Transpose Invariance}]\label{lem:det-transpose}
    Let $K$ be a trace-class integral operator, and let $K^t$ be its transpose. Then 
    \begin{align*}
        \det(I-K^t) = \det(I-K).
    \end{align*}
\end{lemma}

\begin{lemma}[\textbf{Parameter  Differentiation}]\label{Fredholm det der}
    Let $z \mapsto K_z$ be $C^1$ in trace norm on an open set, with $I-K_z$ invertible. Write $R_z = (I-K_z)^{-1}$. Then 
    \begin{align*}
        \partial_z \det(I - K_z) = -\det(I-K_z)\mathrm{tr}(R_z \partial_z K_z).
    \end{align*}
\end{lemma}
\begin{lemma}[\textbf{Resolvent Derivative}]\label{lem:resolvent-derivative}
    Under the assumptions of Lem.\ ~\ref{Fredholm det der}, 
    \begin{align*}
        \partial_z R_z = R_z (\partial_z K_z ) R_z
    \end{align*}
\end{lemma}
\begin{lemma}[\textbf{Rank-One Perturbations}] \label{Fredholm rank one}
Let $A, B$ be bounded with $I-A$, $I-B$ invertible, and suppose $A = B + \psi \otimes \phi$. Let $ F_A = \det(I-A),  F_B = \det(I-B), R_A = (I-A)^{-1}, R_B = (I-B)^{-1}$. Then
\begin{align*}
    \frac{F_A}{F_B} &= 1- \langle R_B \psi, \phi\rangle, \quad \frac{F_B}{F_A} = 1+ \langle R_A \psi, \phi \rangle. 
\end{align*}
\end{lemma}
\begin{lemma}[\textbf{Rank-One Resolvent Identity}]\label{Fredholm fg}
    Under the assumptions of Lem.\ ~\ref{Fredholm rank one}, we have for any $f, g \in H$, 
    \begin{align} \label{rank one id}
        \langle R_A f, g\rangle   = \langle R_Bf, g\rangle   +  \frac{F_B}{F_A}   \langle R_B\psi, g\rangle  \langle R_Bf, \phi\rangle .
    \end{align}
    In particular, 
    \begin{align} \label{sp rank one id}
        \langle R_A \psi, g\rangle   = \frac{F_B}{F_A}\langle R_B\psi, g\rangle  , \quad \langle R_A f, \phi\rangle   = \frac{F_B}{F_A}\langle R_B f, \phi\rangle  .
    \end{align}
\end{lemma}
\begin{proof}
     Since $A, B$ are rank-one perturbations, we have (together with Lem.\ ~\ref{Fredholm rank one})
    \begin{align*}
        R_A = R_B + (F_B/F_A) R_B \psi \otimes \phi R_B.
    \end{align*}
    Taking inner products with $f, g \in \mathcal{H}$ gives \eqref{rank one id}, and \eqref{sp rank one id} follows by setting  $f = \psi$ or $g=\phi$.  
\end{proof}
\section{The KPZ Fixed Point and KP}\label{sec: B}
In this section, we rederive the fact that the one-point distributions of the KPZ fixed point satisfy the bilinear form of KP. First, we start with a general Fredholm determinant solution theorem, which is in essence the one-point case of \cite[Thm.\ 1.3]{Quastel_2022} (see also \cite[Sec.\ 4]{Quastel_2022} for a short history and the relationship between KP and Fredholm determinant solutions). The only real difference is that by pushing the differential relations explicitly onto the parameter space (as was the technique to produce solutions to our equations in Sec.\ 2), we believe we obtain a simpler proof.

First, we define a sufficient regularity class. Let $U = I \times S \times \R$, with $I, S \subseteq \R$ open. We say a family of trace-class integral operators $K_{t, x, a} \in C_r^{1, 2, 4}(U, \mathcal{I}_1)$ if for all $(t, x, a) \in U$, and every multindex $\aaa = (\aaa_1, \aaa_2, \aaa_3)$ with $0\leq \aaa_1 \leq 1, 0 \leq \aaa_2 \leq 2, 0\leq \aaa_3 \leq 4$,  we have $\partial^{\mathbf{\aaa}}K_{t, x, a}$ exists and depends continuously on $(t, x, a)$ in trace norm, and for a.e.\ $(u, v) \in X \times X, \lim_{r \rightarrow \infty} K_{t, x, r}(u, v) = 0$ with $\int_a^{\infty} \abs{\partial^{\mathbf{\aaa}}K_{t, x, r}(u, v)} dr < \infty $ for all such $\aaa \neq (0, 0, 0)$. 

\begin{theorem}
    Let $K_{t, x, a} \in C^{1,2,4}_r \bigl(U,    \mathcal{I}_1\bigr)$ be a family of trace-class integral operators acting on $ L^2(X, \mu)$ such that the following three conditions hold: 
    \begin{enumerate}
        \item ($a$--flows): 
       $\begin{aligned}[t]\label{KP a flows}
            \partial_a K_{t, x, a} &= -\psi_{t, x, a} \otimes \phi_{t, x, a}.  
        \end{aligned}$
        \item ($x$--flows): $\begin{aligned}
            \partial_x \psi_{t, x, a} = \partial_a^2 \psi_{t, x, a}, \quad \partial_x \phi_{t ,x, a}  =-\partial_a^2 \phi_{t, x, a}.
        \end{aligned}$
        \item($t$--flows): $\begin{aligned}\label{KP t flows}
            \partial_t \psi_{t, x, a} &= -\frac{1}{3}\partial_a^3 \psi_{t, x, a}, \medspace  \partial_t \phi_{t, x, a} = -\frac{1}{3}\partial_a^3 \phi_{t, x, a}.
        \end{aligned}$
    \end{enumerate}
Suppose further $I- K_{t, x, a}$ is invertible for all $(t, x, a) \in V$, with $V \subseteq U$ open. Then 
\begin{equation}
    F_{t, x, a} = \det(I-K_{t, x, a})_{L^2(X, \mu)}, \qquad (t, x, a) \in V
\end{equation} satisfies
    \begin{align}\label{Appendix KP}
        \Bigl[D_tD_a + \frac{1}{4}D_x^2 + \frac{1}{12}D_a^4\Bigr]F_{t, x, a}\cdot F_{t, x, a} = 0. 
    \end{align}
\end{theorem}
\begin{proof}
    First, note that due to \eqref{KP a flows}--\eqref{KP t flows} and since $K_{t, x,a} \in C_r^{1,2, 4}(U, \mathcal{I}_1)$, we have 
    \begin{align}
        K_{t, x, a} &= \int_a^{\infty}\psi_{t, x, r}\otimes \phi_{t, x, r}dr,  \\
        \partial_x K_{t, x, a} &= -\partial_a \psi_{t, x, a} \otimes \phi_{t, x, a} + \psi_{t, x, a} \otimes \partial_a \phi_{t, x ,a} , \label{KP part x}\\
        \partial_t K_{t, x, a} &= \frac{1}{3}(\partial_a^2 \psi_{t, x, a}\otimes \phi_{t, x, a} - \partial_a \psi_{t, x, a}\otimes \partial_a \phi_{t, x, a} + \psi_{t, x, a}\otimes \partial_a^2 \phi_{t, x, a}). \label{KP part t}
    \end{align}
    To see \eqref{KP part x}, we compute 
    \begin{align*}
        \partial_x K_{t, x, a}  &= \int_a^{\infty} \partial_r^2 \psi_{t,x, r} \otimes \phi_{t, x, r} - \psi_{t, x, r} \otimes \partial_r^2 \phi_{t, x, r} dr = \int_a^{\infty} \partial_r (\partial_r \psi_{t,x, r} \otimes \phi_{t, x, r} - \psi_{t, x, r} \otimes \partial_r \phi_{t, x, r} )dr \\
        &= -\partial_a \psi_{t, x, a} \otimes \phi_{t, x, a} + \psi_{t, x, a} \otimes \partial_a \phi_{t, x ,a} .
    \end{align*}
    To see \eqref{KP part t}, we compute 
\begin{align*}
    \partial_t K_{t, x, a} &= -\frac{1}{3}\int_{a}^{\infty} \partial_r^3 \psi_{t, x, r} \otimes \phi_{t, x, r} + \psi_{t,x, r} \otimes \partial_r^3 \phi_{t, x, r} dr \\
    &= -\frac{1}{3}\int_a^{\infty} \partial_r(\partial_r^2 \psi \otimes \phi - \partial_r \psi \otimes \partial_r \phi + \psi \otimes \partial_r^2 \phi) dr =  \frac{1}{3}(\partial_a^2 \psi \otimes \phi - \partial_a \psi \otimes \partial_a \phi + \psi \otimes \partial_a^2 \phi).
\end{align*}
Using Lem.\ \ref{Fredholm det der}, we have the following first order derivatives
\begin{gather*}
    \partial_a F = F \langle R \psi, \phi\rangle, \quad 
    \partial_x F = F (\langle R \partial_a \psi, \phi\rangle - \langle R \psi, \partial_a \phi \rangle ),\\
    \partial_t F = -\frac{1}{3}F(\langle R \partial_a^2 \psi , \phi\rangle - \langle R \partial_a \psi, \partial_a \phi \rangle + \langle R\psi, \partial_a^2 \phi\rangle ). 
\end{gather*}
Similarly, using Lem.\ \ref{Fredholm det der}--\ref{lem:resolvent-derivative} we obtain
\begin{align*}
    \partial_a \partial_t F &= F^{-1}\partial_a F \partial_t F -\frac{1}{3}F(-\langle R\psi, \phi\rangle \langle R \partial_a^2 \psi, \phi\rangle + \langle R \partial_a^3 \psi, \phi\rangle )\\
    &\qquad -\frac{1}{3}F( \langle R\psi, \partial_a \phi\rangle \langle R \partial_a \psi, \phi\rangle - \langle R \psi, \partial_a^2 \phi\rangle \langle R \psi, \phi\rangle  + \langle R \psi, \partial_a^3 \phi\rangle), \\ 
    \partial_x^2 F
    &= F^{-1}(\partial_x F)^2 - F (\langle R\partial_a \psi, \phi\rangle^2  -2\langle R \psi, \phi\rangle \langle R\partial_a \psi, \partial_a \phi\rangle + \langle R \psi, \partial_a \phi\rangle^2)\\
    &\qquad - F(-\langle R \partial_a^3\psi, \phi\rangle + \langle R \partial_a \psi, \partial_a^2 \phi\rangle  + \langle R \partial_a^2 \psi, \partial_a \phi\rangle - \langle R \psi, \partial_a^3 \phi\rangle ),  
\end{align*}
and (note the quadratic terms cancel)
\begin{align*}
    \partial_a^2 F &= F ( \langle R \partial_a \psi, \phi\rangle + \langle R \psi, \partial_a \phi \rangle ), \quad
    \partial_a^3 F = F (\langle R \partial_a^2 \psi, \phi\rangle + 2\langle R \partial_a \psi, \partial_a \phi\rangle + \langle R \psi, \partial_a^2 \phi\rangle ),
\end{align*}
so that at fourth order we have
\begin{align*}
    \partial_a^4 F &= 2F ( \langle R \psi, \phi\rangle \langle R \partial_a \psi, \partial_a \phi\rangle - \langle R \psi, \partial_a \phi\rangle \langle R \partial_a \psi, \phi\rangle ) \\
    &\qquad + F (\langle R \partial_a^3 \psi, \phi \rangle + 3\langle R \partial_a^2 \psi, \partial_a \phi \rangle + 3 \langle R \partial_a \psi, \partial_a^2 \phi \rangle + \langle R \psi, \partial_a^3 \phi \rangle).
\end{align*}
Therefore, writing out \eqref{Appendix KP} explicitly and collecting terms, we obtain
\begin{align*}
    &\frac{1}{F^2}[D_tD_a + \frac{1}{4}D_x^2 + \frac{1}{12}D_a^4] F\cdot F \\
    &=  \frac{1}{F^2}( (F \partial_{t, a}F - \partial_{t}F \partial_a F)+ \frac{1}{4}(F \partial_x^2 F - (\partial_x F)^2) + \frac{1}{12}(F\partial_a^4 F - 4 \partial_a F \partial_a^3 F + 3 (\partial_a^2 F)^2)) \\
    &= 0,
\end{align*}
as required.
\end{proof}
Next, we confirm the one-point distributions of the KPZ fixed point satisfy the conditions of the above theorem. For brevity, we will only work with one-sided data for the fixed point, i.e.\ we will consider $\hh_0 \in \text{UC}$ with $\hh_0(x) = -\infty$ for $x > L$ for some arbitrary $L$. By shift invariance, we might as well take $L = 0$. Now in the notation of \cite{MQR17}, define 
\begin{align*}
    \mathcal{S}_{t, x}(u, v) &= t^{-1/3}e^{\frac{2x^3}{3t^2}-\frac{(u-v)x}{t}} \mathrm{Ai}(-t^{-1/3}(u-v) + t^{-4/3}x^2). \\
    \mathcal{S}_{t, x}^{\text{hypo}(\hh_0^-)}(u, v) &= \E_{B(0) = u}[\mathcal{S}_{t, x-\tau}(B(\tau), v)1_{\tau < \infty}],
\end{align*}
where $\mathrm{Ai}(\cdot)$ is the Airy function, $B(\cdot)$ is a Brownian motion with diffusion coefficient $2$, and $\tau$ is the hitting time of the hypograph of $\hh_0^-(\cdot) = \hh_0(-\cdot)$. 
From \cite[Prop.\ 3.6]{MQR17}, we have 
\begin{align}\label{fpppp kernel}
    \PP_{\hh_0}(\hh(t, x) \leq a) = \det(I - \chi_a (\mathcal{S}_{t, -x}^{\text{hypo}(\hh_0^-)})^* \mathcal{S}_{t, x} \chi_a)_{L^2(\R)},
\end{align}
where $\chi_a(z) =1_{\{z \geq a\}}$. Let $A(u, v) = \chi_{a}(u)\mathcal{S}_{t, -x}^{\text{hypo}(\hh_0^-)}(v, u)$, $B(u, v) = \mathcal{S}_{t, x}(u, v) \chi_a(v)$. Take 
\begin{align}
    \psi_{t, x, a}(u) = \mathcal{S}_{t, x}(u, a), \quad \phi_{t, x, a}^{\text{hypo}(\hh_0^-)}(v) = \mathcal{S}_{t, -x}^{\text{hypo}(\hh_0^-)}(v, a).
\end{align}
Then, using Lem.\ \ref{AB BA}, we can rewrite the kernel appearing in \eqref{fpppp kernel} as 
\begin{align*}
    K_{t, x, a}= \int_a^{\infty} \psi_{t, x, r} \otimes \phi_{t, x, r}^{\text{hypo}(\hh_0^-)}dr,
\end{align*}
without changing the value of the Fredholm determinant. The flow conditions now follow from a straightforward computation, using $\mathrm{Ai}''(z) = z\mathrm{Ai}(z)$, and the required regularity follows from \cite[Appendix A]{MQR17}.
\bibliographystyle{amsalpha}   
\bibliography{Hrefs}
\end{document}